\documentclass[11pt, reqno]{amsart}
\usepackage{authblk}
\usepackage{graphicx} 
\usepackage{xcolor}
\usepackage{dsfont}
\usepackage{epstopdf}
\usepackage{framed}
\usepackage{pifont}
\usepackage[foot]{amsaddr}
\usepackage{tikz}
\usetikzlibrary{quotes}
\usepackage{siunitx}
\usetikzlibrary{positioning}
\usepackage{pgf}
\usepackage[normalem]{ulem}
\usepackage{multirow}
\usepackage{geometry}
\usepackage{hyperref}
\usepackage[sort]{cite} 
\usepackage{newpxtext}
\usepackage{todonotes}

\definecolor{colorJRblue}{rgb}{0.,0.,1.}
\definecolor{colorJRred}{rgb}{1.,0.,0.}
\definecolor{colorJRgreen}{rgb}{0.,0.6,0.}

\DeclareMathOperator{\sech}{sech}
\DeclareMathOperator{\diag}{diag}
\def\P{P}

\def\cP{\check{P}}

\def\uE{\underline{E}}
\def\calT{\mathcal{T}}
\def\calL{\mathcal{L}}
\def\calD{\mathcal{D}}

\def\calO{\mathcal{O}}
\def\vb{V^*_1}

\def\cP{\check{P}}

\def\calF{\mathcal{F}}
\def\eps{\varepsilon}
\def\al{\alpha}
\def\be{\beta}

\def\ga{\gamma}

\def\O{\mathcal{O}}

\def\calF{\mathcal{F}}

\def\R{\mathbb{R}}
\def\N{\mathbb{N}}

\def\ta{\underline{a}}
\def\tE{\tilde{E}}
\def\te{b}

\def\ec{e}
\def\Np{{N'}} 
\def\Geps{G_\eps}
\def\Gred{G_0^{\rm red}} 
\def\Gerr{G_\eps^{\rm err}} 
\def\Jnil{\mathcal{J}}
\def\GCM{\mathcal{G}_{\rm CM}}
\def\Gnot{{S_\eps}}

\newtheorem{theorem}{Theorem}[section] 
\newtheorem{remarkss}{Remarks}[]
\newtheorem{remarks}{Remarks}[section]
\newtheorem{remark}[remarks]{Remark}

\newtheorem{lemma}[theorem]{Lemma}

\newtheorem{proposition}[theorem]{Proposition}

\newtheorem{deff}[remarkss]{Definition}
\newtheorem{main}[theorem]{Main Result}
\newtheorem{condition}[theorem]{Condition}

\allowdisplaybreaks

\title[Chaos and singularities of fronts in multi-component FitzHugh-Nagumo-type systems]{Chaotic motion and singularity structures of front solutions 
in multi-component FitzHugh-Nagumo-type systems}

\author{Martina Chirilus-Bruckner*}
\address[*]{Mathematisch Instituut, Leiden University, P.O.Box 9512, 2300 RA Leiden, the Netherlands}
\author{Peter van Heijster**}                  \address[**]{Biometris, Wageningen University \& Research, 6708 PB Wageningen, the Netherlands}
\author{Jens D.M.\ Rademacher***}
\address[***]{Universit\"at  Hamburg, MIN Faculty, Department of Mathematics, 20146 Hamburg, Germany}

\begin{document}

\maketitle


\begin{abstract}

We study the dynamics of front solutions in a certain class of multi-component reaction-diffusion systems, where one fast component governed by an Allen-Cahn equation is weakly coupled to a system of $N$ linear slow reaction-diffusion equations. By using geometric singular perturbation theory, Evans function analysis and center manifold reduction, we demonstrate that and how the complexity of the front motion can be controlled by the choice of coupling function and the dimension $N$ of the slow part of the multi-component reaction-diffusion system. On the one hand, we show how to imprint and unfold a given scalar singularity structure. On the other hand, for $N\geq 3$ we show how chaotic behaviour of the front speed arises from the unfolding of a nilpotent singularity via the breaking of a Shil'nikov homoclinic orbit. The rigorous analysis is complemented by a numerical study that is heavily guided by our analytic findings.\\

\end{abstract}

\pagebreak
\tableofcontents
\pagebreak


\section{Introduction}
\label{S:intro}
The dynamics of spatial phase separation is a classical and broadly studied phenomenon, in particular through sharp interface limits and phase field models, that can also have multiple components \cite[e.g]{
Garcke1998, Steinbach2023,WheelerMcFadden}. In this paper we {take a mathematical view and restrict to a one-dimensional extended homogeneous domain. We} are interested in 
the dynamics of interfaces that is caused by the interaction between the components in a multi-scale system with a distinguished phase field component. 
In \cite{CBDvHR15, CBvHHR19}, together with collaborators, we analysed the dynamics of localised single front solutions ({\it i.e.}, interfaces) in a three-component 
singularly perturbed partial differential equation (PDE). This consists of a {\it{fast}} Allen-Cahn-type phase field PDE which is linearly coupled to two {\it{slow}} linear reaction-diffusion equations with constant coefficients. This system of PDEs turned out to exhibit surprisingly complex analytically tractable behaviour. The dynamics of (the position of) a single front solution includes a butterfly catastrophe~\cite{CBDvHR15} and a symmetric Bogdanov-Takens bifurcation~\cite{CBvHHR19}, leading to velocity changing fronts and stable travelling breathers. However, in the hitherto considered settings the 
dynamics of these front solutions is determined by a two-dimensional 
ordinary differential equation (ODE), which precludes the occurrence of temporally chaotic dynamics. In this paper, we extend the model studied in~\cite{CBDvHR15,CBvHHR19} and show rigorously 
that any given scalar singularity structure can be generated and unfolded, and that chaotic behaviour of the front speed occurs.

In more detail, we consider here the following multi-component reaction-diffusion system with one fast Allen-Cahn-type phase field component and $N$ slow linear components of the form
\begin{align}\label{eq:multi-component-RD}
\left\{
\begin{aligned}
 \partial_t U & =   \varepsilon^2 \partial_x^2 U+ U - U^3 - \varepsilon  \mathcal{F}^N(\vec{V})\,, \\[.2cm]
 \tau_j \partial_t V_j & =  \varepsilon^2 d_j^2 \partial_x^2 V_j  + \varepsilon^2( U -  V_j )\, , \qquad j = 1, \ldots, N\,,
 \end{aligned}
\right.
\end{align}
for $ (U, \vec{V})(x,t) :=  (U, V_1, \ldots, V_N)(x,t) \in \mathbb{R}^{1+N}, N \in \mathbb{N}$, and where $ x \in \mathbb{R}, t \geq 0 $. Here, $ \tau_j, d_j > 0$ are parameters and $ 0 < \varepsilon \ll 1 $ is a distinguished small parameter that yields a sharp interface (in the $U$-profile) and gives access to perturbation methods, see also Figure~\ref{f:intro_front}.
The pointwise coupling function $ \mathcal{F}^N:\mathbb{R}^{N}\to\mathbb{R}$ is assumed to be sufficiently smooth and tangible results will involve conditions on expansion coefficients of this coupling function that we view as parameters. The coupling structure, where $\vec{V}$-components are coupled only via $U$, turns out to be simple enough to admit arbitrary number of components in the analysis. Since the equations for $\vec{V}$ can be solved in integral form, we can also view the systems as a scalar diffusive phase-field equation with nonlocal terms.

System \eqref{eq:multi-component-RD} can also be considered as a generalised FitzHugh-Nagumo equation with multiple diffusive inhibitors or refractory variables. The latter often arises in biological modelling \cite{SCHWEISGUTH2019659} and for $N=3$ system \eqref{eq:multi-component-RD} is in this sense linked to the Hodgkin-Huxley equations, although the refractory dynamics is usually nonlinear \cite{HodgkinHuxley}; see also \cite{GIUNTA} for a similarly related model with $N=2$. For the specific form of \eqref{eq:multi-component-RD}, upon rescaling time, the $\eps$-dependence can be moved entirely into the equation for $U$ that takes a `fast reaction' form. For $N=1$ this gives the same $\eps$-scaling as in the FitzHugh-Nagumo equations studied in \cite{Alfaro2008} in higher dimensional bounded domains. 
We refer to \S\ref{s:related} for a specific application background and previous results on \eqref{eq:multi-component-RD} with $N=2$. 

The dynamic single front solutions that we are concerned with are solutions of~\eqref{eq:multi-component-RD} with one interface that spatially asymptote for fixed $t$ to the constant steady states of \eqref{eq:multi-component-RD}, {\it i.e.},
\begin{align*}
 \lim_{x \rightarrow \pm \infty} 
 ( U, \vec{V})(x,t) =
  \pm (1, 1, \ldots, 1) + \mathcal{O}(\varepsilon) \, ,
\end{align*}
see also Figure~\ref{f:intro_front}. The relevant special type
are {\bf uniformly travelling front solutions}
\begin{align}\label{eq:traveling_wave_intro}
(U, \vec{V} )(x,t) = ( U,\vec{V} )_{\rm TF}(x-\varepsilon^2 c t),\;\qquad c \in \mathbb{R}\, ,
\end{align}
for $x, t \in \mathbb{R}$. In the case of $c=0$ we refer to these as {\bf stationary front solutions} 
$$( U, \vec{V} )(x,t) = ( U, \vec{V} )_{\rm SF}(x)\,.$$ See panels (a) and (b) of Figure~\ref{f:intro_front_paths} for the temporal dynamics of the front positions for these types of dynamic single front solutions.
\begin{figure}[!t]
\includegraphics[width = 0.8\textwidth]{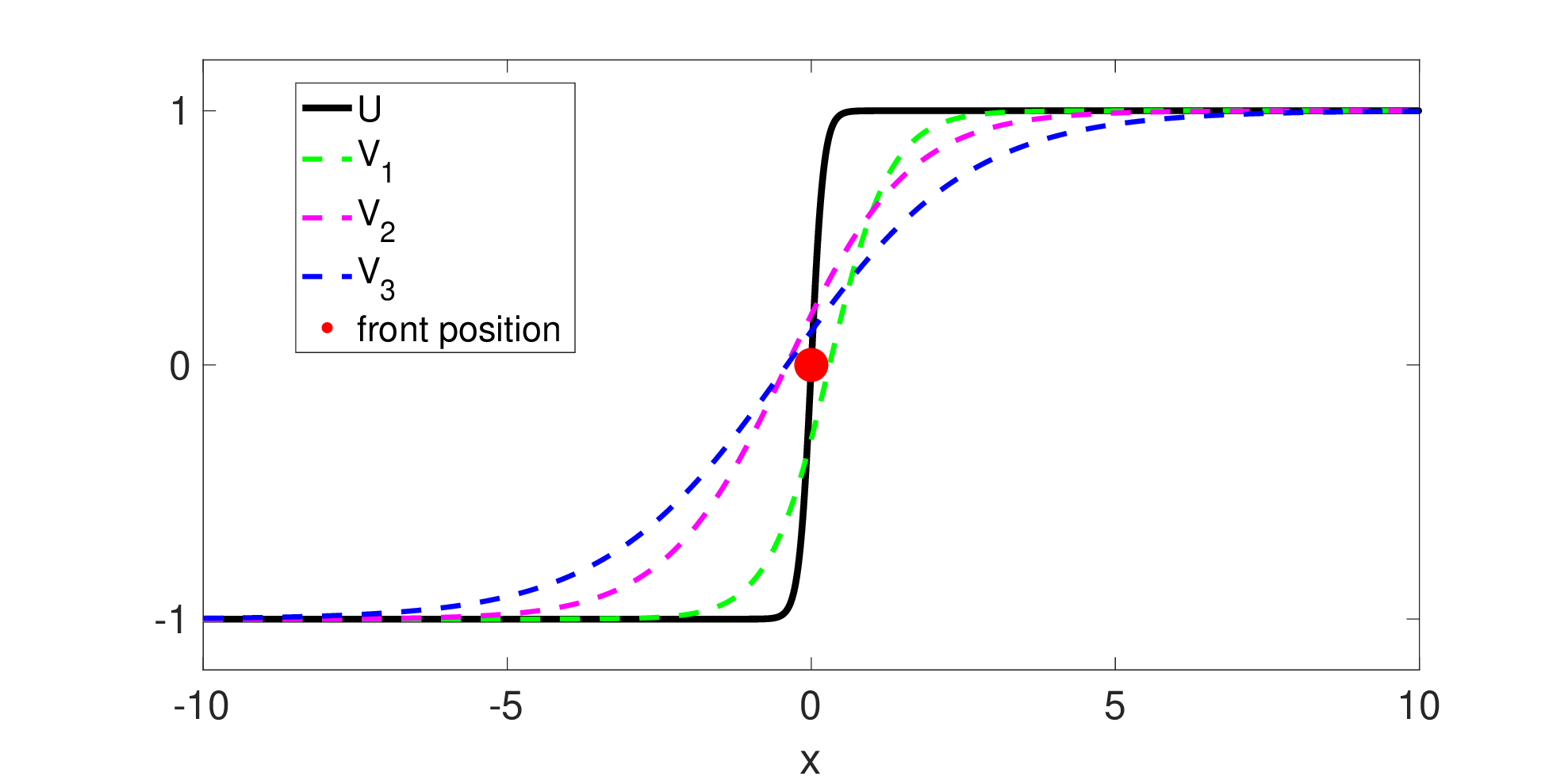}
\caption{Snapshot in time of the spatial profile of a single front solution. Various types of dynamics for the front position, indicated with the red dot, are displayed in Figure~\ref{f:intro_front_paths}. See also Figure~\ref{f:trans_branches_spec} for more details on parameter settings.}
   \label{f:intro_front}
\end{figure}

One can anticipate the existence of single front solutions in \eqref{eq:multi-component-RD} as follows. 
For $\calF^N(\vec{V})=0$ the first equation of \eqref{eq:multi-component-RD} is the one-dimensional Allen-Cahn equation on the real line, which is well known to feature the stationary front solutions $U(x,t) = \pm \tanh(x/ (\sqrt{2}\eps))$. Due to the {small} prefactor $\eps$ of $\calF^N(\vec{V})$, these provide the seed for the front solutions of the full system~\eqref{eq:multi-component-RD}. The Allen-Cahn equation on the real line by itself, however, does not generate non-trivial dynamic single front solutions, but we note that multi-front configurations feature slow {\it metastable} coarsening dynamics, see, for instance, \cite{We21} for a recent account. The origin of the single front motion that occurs in the fully coupled system \eqref{eq:multi-component-RD} can be explained from the asymmetric variant of the Allen-Cahn equation
\begin{align}
\label{AC}
\partial_t \bar{U}  =   \varepsilon^2 \partial_x^2 \bar{U} + \bar{U} - \bar{U}^3 + \varepsilon \gamma \, , \quad \gamma \in \mathbb{R} \,,
\end{align}
{It is well known that this} equation possesses uniformly travelling front solutions as in \eqref{eq:traveling_wave_intro} with a speed $c = \calO(\gamma)$ related to Nagumo fronts and  sometimes called Huxley waves. {The equation for the $U$-component of \eqref{eq:multi-component-RD} results from \eqref{AC} when replacing $\gamma$ by} the coupling function $\calF^N(\vec{V})$. Hence,~\eqref{eq:multi-component-RD} can be expected to generate self-organised $\vec{V}$-dependent front motion.
\begin{figure}[!t]
\includegraphics[width = 0.95\textwidth]{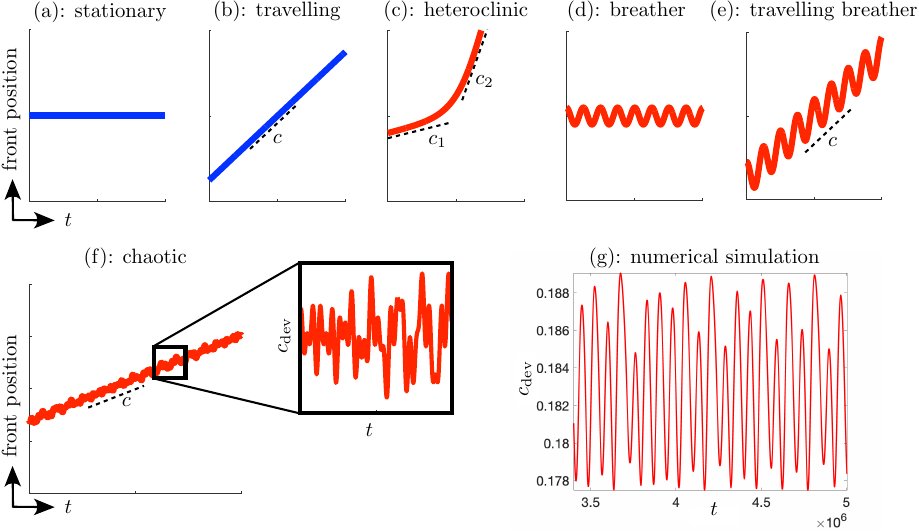}
\caption{(a-f): Sketches of possible temporal dynamics for the front position, see Figure~\ref{f:intro_front}. (a): Stationary front solution $(U, \vec{V})_{\rm SF}$ with $c = 0$, (b): Uniformly travelling front solution $(U, \vec{V})_{\rm TF}$ with speed $c \neq 0$ (c):~Dynamic heteroclinic front solution connecting two different travelling front solutions, (d): Dynamic breathing front solution (with periodically varying front position), (e): Dynamic breathing travelling front solution, (f):~Dynamic chaotic front solution. (g): Numerical simulation of \eqref{eq:multi-component-RD} with $N=3$ displaying complex front dynamics. Here, $c_{\rm dev}$ is the deviation from the mean front speed; cf.\ inset in panel (f). See \S\ref{s:num} for details. }
\label{f:intro_front_paths}
\end{figure}

In \cite{CBDvHR15, CBvHHR19} it was proven that for $N=2$ and $0<\eps\ll1$, affine $\calF^2(\vec{V})$ can create non-trivial {\bf dynamic front solutions} of \eqref{eq:multi-component-RD} of the form 
\begin{align} \label{DFS}
  ( U, \vec{V})(x,t) = ( U, \vec{V})_{\rm SF}(x-a(t)) + \mathcal{R}(x,t) \, ,  
\end{align}
with small $\mathcal{R}$ and the front position $a(t)$ tracking the location of the interface; see \S\ref{S:CMR}. These are neither stationary nor uniformly travelling front solutions, but are characterised by the front position $a(t)$ whose dynamics is governed by a system of ODEs. See panels (c)-(f) of Figure~\ref{f:intro_front_paths} for four different types of dynamic front solutions, where the heteroclinic, the breather and the travelling breather were found in \cite{CBDvHR15} and \cite{CBvHHR19}. The {occurrence of} chaotically travelling front is one of the main highlights of the present work.

\subsection{Main results} 
In this paper we significantly extend these insights from \cite{CBDvHR15, CBvHHR19}, and, in particular, generalise the rigorous ODE reduction obtained via center manifold theory. This then provides a rigorous description of dynamic front solutions bifurcating from stationary front solutions in the multi-component model \eqref{eq:multi-component-RD} for any $N$. In particular, we show that 
for any  $\eps>0$ sufficiently small
the front position $ a = a(t) $ of \eqref{DFS}  
satisfies 
\begin{align}
\label{FR}
 \dfrac{d}{dt} a & = \varepsilon^2 c_1, 
\end{align}
and $c_1$ is determined through an $\Np$-dimensional ODE, with $\Np \leq N$ if $\tau_j$, $1\leq j\leq\Np$, are pairwise distinct. This ODE entails the variable ~$\vec{c}=(c_1,\ldots,c_{\Np})$ and has the 
structure
\begin{align}\label{eq:main_ODE}\left\{
\begin{array}{lcl}
  \dfrac{d}{dt} c_{k} & = & \varepsilon^2 c_{k+1} \,,  \qquad  k = 1, \ldots, \Np-1, \quad \textnormal{(for $N'>1$)}\\[.4cm]
  \dfrac{d}{dt} c_{\Np}   & = & \varepsilon^2 \Geps(\vec{c})  \,,
\end{array} \right.
\end{align}
of an $\Np$-th order equation, and hence {\bf nilpotent} linear part. The nonlinear function $\Geps$ is determined by the coupling function $\calF^N(\vec{V})$ and the parameters of \eqref{eq:multi-component-RD}. Importantly, {as a result of the spatial translation symmetry of \eqref{eq:multi-component-RD}, it} does not dependent on the front position $a(t)$ directly. Henceforth, we refer to \eqref{eq:main_ODE} as the {\bf speed ODE}, while we call \eqref{eq:main_ODE} combined with \eqref{FR} the {\bf reduced ODE} (since it corresponds to the PDE reduced to the center manifold).
Here, the dimension of the center manifold is thus $\Np+1$. The possible dynamics of~\eqref{eq:main_ODE}, in particular, its dimension~$\Np$, depends on the specific details of the coupling function $\calF^N(\vec{V})$ and the parameters of \eqref{eq:multi-component-RD}.

A simple, but important observation linking dynamic front solutions of \eqref{eq:multi-component-RD} to their stationary and uniformly travelling counterparts{, see \eqref{eq:traveling_wave_intro},} is the following: Fixed points of~\eqref{eq:main_ODE} { ({\it i.e.} roots of $\Geps$)} with $c_1(t)=0$ so $a(t) = a(0)$ correspond to stationary front solutions, while fixed points of~\eqref{eq:main_ODE} with $c_1(t) = c \in \mathbb{R},$ so $a(t) = a(0) + c\,t$, correspond to uniformly travelling ones.

The main specific results of this paper can be summarised as follows -- see Theorem~\ref{thm:one_slow_sing} and Theorem~\ref{thm:chaos} for the exact statements. 
{These fundamentally rely on applying the theoretical framework that we develop in this paper for the existence, the location of point spectrum, and the structure of the center manifold reduction (Proposition~\ref{prop:reduced_system}) of dynamic fronts in \eqref{eq:multi-component-RD}.}

\begin{main}\label{thm:intro} 
There exists $\varepsilon_0 >0$ such that for all $\varepsilon \in (0, \varepsilon_0)$  the dynamic front solutions \eqref{DFS} of \eqref{eq:multi-component-RD} are described by reduced ODE \eqref{FR}-\eqref{eq:main_ODE}. Furthermore, the following holds.
\begin{itemize}
\item[(a)] \underline{Singularities for $N = 1$:} {The} coupling function $\calF^1(V)$ {can be chosen such that} the speed ODE \eqref{eq:main_ODE} is scalar (so $\Np = 1$) and {in it one can} generate and unfold any scalar singularity structure. In particular, {for $\tau_1/(2\sqrt{d_1}) \neq \sqrt{2}/3$ one has} a 1-to-1 correspondence between the Taylor coefficients of $\calF^1$ and $\Geps$ to any order $M \in \mathbb{N}$.
\item[(b)] \underline{Chaos for $N = 3$:} 
For $\tau_1,\tau_2,\tau_3$ pairwise distinct and { $\sum_{j=1}^3 \partial_{V_j}^2\calF^3(0)\tau_j^2/d_j^2\neq 0$,
there are choices of $\calF(0), \nabla\calF^3(0)$} ensuring that the dimension of the speed ODE~\eqref{eq:main_ODE} is $\Np = 3$ and that it features chaotic dynamics through Shil'nikov homoclinic bifurcations.
\end{itemize}
\end{main}
The type of chaotic dynamics that we identify in the speed ODE~\eqref{eq:main_ODE} for $N'=3$ is linked to the existence of persistent strange attractors. Classic scenarios that lead to these are global bifurcations of special homoclinics or heteroclinic cycles, see~\cite{HS10} for an overview. Our specific scenario features saddle-focus homoclinics that are now commonly called Shil'nikov homoclinic orbits, see~\cite{Shi65}. It is known that every neighbourhood of these contains infinitely many horseshoes and that breaking a Shil'nikov homoclinic orbit leads to the destruction of horseshoes and to persistent 
strange attractors, see \cite[e.g.]{MV93,Tr81}. Early examples and numerical evidence for such homoclinic phenomena was provided in~\cite[e.g.]{ACT82}. 

Rigorously verifying such global bifurcations in a given ODE can be a formidable task. This is even more so for our case, since the speed ODE is derived via center manifold reduction from a PDE. Hence, explicit control of the dynamics of the speed ODE is analytically and algebraically involved. Our inspiration stems from a collection of papers concerning the emergence of this type of chaos from singularities, see the monograph \cite{BarrientosBook} for an overview. In particular, we follow~\cite{IR05} which specifies conditions for the occurrence of Shil'nikov homoclinics (from certain heteroclinic cycles that are often called T-points, see also \cite{DIK01}) in an unfolding around a nilpotent organising center. In this way,  via {\it local} bifurcation analysis one gets control over {\it global} bifurcations of Shil'nikov homoclinics which, in turn, yield persistent strange attractors.

As will become clear in the course of this paper, nilpotent organising centres naturally occur in the speed ODE \eqref{eq:main_ODE} due to the special form of the reaction-diffusion system~ \eqref{eq:multi-component-RD}. This makes the just described route to chaos via unfolding a conceivable and algebraically convenient scenario, whose detailed rigorous justification is one of the main novelties of this paper. We note that our results are structurally stable and thus remain valid for systems that result from perturbing the specific right hand side in \eqref{eq:multi-component-RD}.We caution the reader that we are \emph{not} concerned with complicated spatial dynamics, {\it{e.g.}} associated to homoclinics (in the spatial variable) as in the classical FitzHugh-Nagumo system~\cite{carter2016stability, sandstede1998stability}. 


\begin{remark}[Imprinting singularity structures and chaotic motion for general $N$] \label{r:impactN}
Central to the proof of both (a) and (b) in the Main Result~\ref{thm:intro} is having explicit control over (i) the dimension $N'$ and (ii) the dynamically relevant terms of the speed ODE. For the singularity imprinting it is particularly convenient to set $N = 1$, since it allows to prove the 1-to-1 correspondence between the Taylor coefficients of $\calF^1$ and $\Geps$ to any order $M \in \mathbb{N}$. Similarly, the choice $N = 3$ and quadratic coupling function $\calF^3$ allowed to display how chaotic motion can be realised along the technically least involved case. For general $N$ the special structure of \eqref{eq:multi-component-RD} still allows to clearly track the impact of the choice of parameters and coupling function (see \S\ref{subsec:recipe} and, in particular, Table~\ref{T:overview}, and Remark~\ref{R:GEN}) on the speed ODE, but it is technically more involved than the minimal cases $N = 1$ and $N = 3$. 

A technically straightforward extension to general~$N$ is possible for coupling functions that allow the splitting {with small parameter $\delta\geq 0$ of the form} 
$$\calF^N(\vec{V}) = \calF^{k}(\underline{\vec{V}}) + \delta \mathcal{\overline{F}}^{N}(\vec{V}), \qquad k=1,3\,, N\geq k$$ with $\underline{\vec{V}} \in \mathbb{R}^{k}$ part of $\vec{V} \in \mathbb{R}^{N}$ and $\calF^{k}, k=1, 3,$ as chosen in Main Result~\ref{thm:intro}. In particular, setting $\delta=0$ decouples~\eqref{eq:multi-component-RD} from the linear $V_j$-components, which also do not influence any stability properties. A more detailed discussion of such a straightforward extension is given in \S\ref{subsec:sing}.
\end{remark}

\subsection{Relation to previous work}\label{s:related} 
In \cite{CBDvHR15, CBvHHR19}, 
(with varying co-authors) we analysed the dynamics of single front solutions in~\eqref{eq:multi-component-RD} with affine coupling term $\mathcal{F}^N=\mathcal{F}^N_\mathrm{aff}$ of the form
\begin{align}
\label{N:affine}
\mathcal{F}^N_\mathrm{aff}(\vec{V}) :=  \gamma + \sum_{j=1}^{N} \alpha_j V_j\,,
\end{align}
and for $N=2$. That is, we analysed {dynamic fronts in}
\begin{align}
\label{1F2S}
\left\{
\begin{aligned}
 \partial_t U & =   \varepsilon^2 \partial_x^2 U + U - U^3 - \varepsilon  \mathcal{F}^2_\mathrm{aff}(\vec{V})\,, \\[.2cm]
 \tau_j \partial_t V_j & =  \varepsilon^2 d_j^2 \partial_x^2 V_j  + \varepsilon^2( U -  V_j )\, , \qquad j = 1, 2.
\end{aligned}
\right.
\end{align}
See Remark~\ref{R:par} below concerning notational differences between \eqref{1F2S} and \cite{CBDvHR15, CBvHHR19}. The dimensional version of \eqref{1F2S} was introduced in the mid-1990s as a phenomenological model for gas-discharge \cite{PUR14, Woesler1996}. It has been studied intensively, also in different forms and asymptotic scalings, ever since~\cite[e.g.]{P02,brown2023analysing, CBDvHR15, CBvHHR19, DvHK09, GAP06, NTU03, P98,P97, teramoto2021traveling, vHCNT16, van2019pinned, vHDK08, vHDKNU11, vHDKP10,vHS11, vHS12, vHS14, VE07,YTN07}. As mentioned, we were able to understand various aspects of the dynamics resulting from a single front solution analytically in \cite{CBDvHR15, CBvHHR19}. This involved, in particular, the successive examination of existence, stability and bifurcations of stationary and uniformly travelling front solutions, and a center manifold reduction. In \cite{CBDvHR15}, we restricted ourselves to a parameter regime for which the speed ODE \eqref{eq:main_ODE} turned out to be scalar ({\it i.e.}, $\Np=1$) and features the imprint of a (partial) butterfly catastrophe, while in~\cite{CBvHHR19} we determined a parameter regime giving rise to a planar speed ODE \eqref{eq:main_ODE} ({\it i.e.}, $\Np=2$) with a symmetric Bogdanov-Takens bifurcation. This latter case allowed to explain the occurrence of stable standing and travelling breathers (where ``breather'' in this case meant that the front position moves periodically in time). See panels (d) and (e) of Figure~\ref{f:intro_front_paths}. 

System \eqref{1F2S}
enjoys the symmetry
\begin{align}
\label{SYM}
(U,\vec{V}, \gamma) \mapsto (-U,-\vec{V}, -\gamma)\,,
\end{align}
and its combination with the spatial reflection symmetry gives the symmetry
\begin{align}
\label{SYM2}
(c, x, U,\vec{V}, \gamma) \mapsto (-c, -x, -U,-\vec{V}, -\gamma)\,,
\end{align}
where $c$ is the speed of a uniformly travelling front solution. 
The multi-component system~\eqref{eq:multi-component-RD} (in travelling wave coordinates) with an affine coupling function $\calF^N(\vec{V}) = \calF^N_{\rm aff}(\vec{V})$ still enjoys these symmetries, which precludes the realisation of the particular normal form in the speed ODE \eqref{eq:main_ODE} that will generate the chaotic dynamics of Main Result~\ref{thm:intro}. In this paper we therefore break~\eqref{SYM2} by allowing the coupling function $\mathcal{F}^N(\vec{V})$ to also contain nonlinear terms.


\begin{remark}[Notation in previous papers]
\label{R:par} 
To ease reading, we compare the notation to several earlier papers about \eqref{1F2S}: $ (\alpha, \beta) $ corresponds to $ (\alpha_1, \alpha_2) $ here, $ (\hat{\tau}, \hat{\theta}) $  to $(\tau_1,\tau_2)$, and $D$ to $d_2$. The choice $d_1=1$ in \cite[e.g.]{CBDvHR15, CBvHHR19} can be achieved without loss of generality by spatial rescaling.
\end{remark}


\begin{remark}[Breaking the symmetry via advection] 
\label{ADV} In order to break \eqref{SYM2} alternatively to nonlinear coupling (with even powers), one can break the spatial reflection symmetry. For instance, this can be realised by adding advection terms to the slow components of \eqref{eq:multi-component-RD} as in 
\begin{align}   \label{eq:ADV}
\left\{
\begin{aligned}
 \partial_t U & =   \varepsilon^2 \partial_x^2 U+ U - U^3 - \varepsilon  \mathcal{F}^N_\mathrm{aff}(\vec{V})\,, \\[.2cm]
 \tau_j \partial_t V_j & =  \varepsilon^2 d_j^2 \partial_x^2 V_j   + \eps \delta_j \partial_x V_j + \varepsilon^2( U -  V_j )\, , \qquad j = 1, \ldots, N \,,
 \end{aligned}
\right.
\end{align}
such that \eqref{SYM2} no longer holds (while \eqref{SYM} is not violated). This type of advection occurs in models that contain some coupling with fluid flow, e.g., FitzHugh-Nagumo-type equations with advection as in \eqref{eq:ADV} have recently been linked to turbulent pipe flow, cf.\ \cite{barkley2015rise, engel2021traveling}.
{Although our approach generally applies, we} do not pursue this direction further in this paper. 
\end{remark}


\subsection{Related work on chaos in PDEs}
While the occurrence of chaos in high-dimensional systems of nonlinear ODEs and PDEs is sometimes
considered ``folklore", actual proofs of this behaviour -- in particular, in the context of PDEs -- remain rather rare. 
There are, however, several related results that we would like to highlight, without giving an exhaustive list or full overview.

First, we point out two instances of single pulses in the context of reaction-diffusion systems without any spatial or temporal forcing. In \cite{VD13} the dynamics near stationary pulse solutions in a multi-scale reaction-diffusion systems on the real line was studied. Chaotic motion in the vicinity of a Hopf bifurcation was numerically observed, but since the reduced dynamics was planar, it was postulated that the observed chaos could be related to a numerical approximation error of the translation symmetry. In \cite{DSZ15} a new formal reduction procedure, coined {\it{extended center manifold reduction}}, for studying pattern formation in rather general multi-scale two-component reaction-diffusion systems on bounded domains was introduced. This formal procedure led to a reduced ODE system featuring chaotic motion in a model for phytoplankton-nutrients.

Another way to create chaotic motion in reaction-diffusion systems is by introducing spatial or temporal forcing of the PDE. In \cite{NTYU07} the dynamics of pulses was studied for a variant of the Gray-Scott model with spatially periodic coefficients (along other studies of heterogeneous terms in the three-component system as treated here (so $N=2$)). The dynamics of pulses was formally reduced to an ODE system for which numerical simulations suggest chaotic behaviour of the pulse path reminiscent of the Lorenz attractor.

Since the dimension of the reduced ODE is related to the number of eigenvalues near zero of the linearisation around an interface ({\it{e.g.}}, around a front and pulse), which is in turn related to the number of interfaces, another approach to generate chaotic dynamics is through interacting fronts or pulses. An example of such a scenario is studied in~\cite{NU01} for interacting pulse solutions in Gray-Scott-type reaction-diffusion systems. Such interactions lead to the more complex concept of spatio-temporal chaos. An example with forcing and rigorous proofs concerns interacting pulses in the Ginzburg-Landau equation \cite{TZ10}. Spatio-temporal chaos in abstract dissipative systems has been rigorously studied~\cite{MZ09}, including an example for chaos in the periodically forced Swift-Hohenberg equation.

Finally, we direct the attention to the many inspiring papers in the mathematical physics literature on chaotic motion in fluid dynamics. An early paper in this context that conjectured the occurrence of Shil'nikov homoclinics was \cite{ACST85}, this was later proved in \cite{IR05} -- the paper that guided our analysis here. We also point the reader to \cite{DKT87,GIUNTA,HKTP10,KLM99,KMTW86,KW87,MTKW83,RK17}, to name just a few of the papers in which the described phenomena relate to those that we establish in this paper. 


\subsection{Outline and approach}
In \S\ref{S:1+N} our focus lies on using spatial dynamics and Geometric Singular Perturbation Theory~\cite[e.g.]{F79, J95, K99} (GSPT) in order to derive the existence condition for stationary and uniformly travelling front solutions with sufficiently detailed information for the subsequent analysis. In particular, the results of \S\ref{S:1+N} will be used in \S\ref{subsec:sing} to derive the degree $M$ of $\Geps(\cdot)$ as stated in the first part of Main Result~\ref{thm:intro}.

In~\S\ref{S:1+N_S} we employ the so-called Non-Local Eigenvalue Problem (NLEP) method~\cite{D98, D01, D02, vHDK08} to determine the Evans function, which provides all relevant information on the part of the spectrum that can lead to instabilities and bifurcations of the front solutions. We also identify parameters for which an up to $(N+1)$-fold zero eigenvalue of a stationary front solution occurs, which is the main organising center for the dynamics. Notably, in these computations Vandermonde matrices occur repeatedly, see Lemma~\ref{L:VDM}, which originate from the specific slow-fast structure and coupling form of~\eqref{eq:multi-component-RD}. 
In combination with the Singular Limit Eigenvalue Problem (SLEP) method \cite{nishiura1987stability, nishiura1990singular} we prove the key result that an $(\ell+1)$-fold zero eigenvalue comes with a full Jordan chain for all $0 \leq \ell \leq N$. This implies that the speed ODE~\eqref{eq:main_ODE}, where $\ell=\Np$, will always have a nilpotent linear part, which again relates to the slow-fast structure and coupling form of \eqref{eq:multi-component-RD}. This forms the backbone for deriving the second part of Main Result~\ref{thm:intro}.

In \S\ref{S:DFS} we first perform a center manifold reduction and normal form analysis, see~\S\ref{S:CMR}, which form the basis for our main results as informally summarised in Main Result~\ref{thm:intro}. 
These main results can be viewed as two highlight illustrations of what can be done based on the general framework provided in \S\ref{S:1+N}-\S\ref{subsec:recipe}. 
For instance, in \S\ref{S:sing} we show that Taylor expansions of $ \mathcal{F}^N(\vec{V})$ are in an (to order $\eps$) explicit 1-to-1 correspondence with the (near zero) locations of equilibria of \eqref{eq:main_ODE}, {\it{i.e.}}, the velocities of front solutions in \eqref{eq:multi-component-RD}. This may be exploited for phenomenological modelling by means of reaction-diffusion equations.
Furthermore, in \S\ref{S:1+3} we show that for $N = 3$ the ``simplest'' nonlinear coupling function 
\begin{align}\label{NONLnew}
\mathcal{F}^N(\vec{V}) = \mathcal{F}^N_\beta(\vec{V}) := \gamma+\sum_{j=1}^N \alpha_j V_j + \sum_{j=1}^N \beta_j V_j^2, \qquad \alpha_j, \beta_j, \gamma \in \mathbb{R} \, ,
\end{align}
with (to order $\eps$) explicitly selected coefficients, generates a Shil'nikov bifurcation for~\eqref{eq:main_ODE}. 
By structural stability, this also occurs when higher order terms are added to $\calF^N$.
We emphasise that our normal form analysis for the reduced ODE on the center manifold does not follow the tedious standard route of computing various projections, but rather makes use of a combination of existence and stability information that is then woven into the reduced system. In \S\ref{s:num} we numerically corroborate several of our findings (see \S\ref{subsec:recipe} for an overview).

We end the paper with a brief summary of the main results and a discussion on
possible further research directions, see \S\ref{S:DIS}.


\subsection{Notation, definitions, conventions}
Before we continue, we introduce some crucial notation that we use throughout the paper.

\begin{deff}[Slow and fast variables]\label{CORD} The spatial variable $x$ in \eqref{eq:multi-component-RD}  will be referred to as the {\it{slow variable}}. We will also make use of the {\it{slow travelling wave variable}}
$$
y:= x - \eps^2 c t\,,
$$
where $\eps^2 c$ represents the leading order speed (in terms of the slow variable) of a uniformly travelling front solution. 
The related {\it{fast (travelling wave) variables}} $z$ and $\xi$ are given by
$$
z:=\dfrac{y}{\eps} = \dfrac{x}{\eps} - \eps c t =: \xi -\eps c t \,.
$$
\end{deff}

\begin{deff}[Slow and fast fields]\label{fields} To study uniformly travelling and stationary front solutions we will split the spatial domain into two {\it{slow fields}} and one {\it{fast field}}. In the slow scaling ({\it{i.e.}} in $x$ and $y$) these are given by
$$
I_s^- :=  (-\infty, -\sqrt\eps), \qquad
I_f :=  [-\sqrt\eps,\sqrt\eps]\,, \qquad
I_s^+ :=  (\sqrt\eps, \infty)\,, 
$$ 
where we note that the centering of $I_f$ around zero is without loss of generality by the translation invariance of the system \cite{vHDK08}. 
\end{deff}

\begin{deff}[Taylor expansion]
\label{TAYDEF}
For $f: U \subset \mathbb{R} \rightarrow \mathbb{R}$, with $U$ an open neighbourhood of $x_0 = 0$, and $f$ sufficiently smooth in $U$, we denote its approximating Taylor, or Maclaurin, polynomial up to order $\mathcal{O}(n)$, for some $n \in \mathbb{N}$, by 
$$\mathcal{T}_f^n(x):=\sum_{k=0}^n \frac{1}{k!}f^{(k)}(0) x^k\, $$
and, if $f$ is real analytic on $U$, we write {$\mathcal{T}_{f}$ for the infinite series.}
\end{deff}

\begin{remark}[Convention for perturbation results with respect to $0 <\varepsilon \ll 1$]
In the following we will repeatedly use the formulation ``Let $\varepsilon$ be sufficiently small ..." for the more precise statement ``There exists $\varepsilon_0 >0$ such that for all $\varepsilon \in (0, \varepsilon_0)$...". 
\end{remark}


\section{Existence of front solutions}
\label{S:1+N}
We study stationary or uniformly travelling front solutions in the $N+1$-component system~\eqref{eq:multi-component-RD} for $\eps>0$ sufficiently small. The nonlinear coupling function~$\mathcal{F}^N(\vec{V})$ is assumed to be smooth without further conditions at this stage. However, for convenience, we split it into an affine part from \eqref{N:affine} and a nonlinear part $\mathcal{F}^N_{\rm{nl}}(\vec{V})= \mathcal{O}(|\vec{V}|^2)$, {\it{i.e.}},
\begin{align}
\label{N:NONL22}
\mathcal{F}^N(\vec{V}) := \mathcal{F}^N_{\rm{aff}}(\vec{V}) + \mathcal{F}^N_{\rm{nl}}(\vec{V}) = \gamma + \sum_{j=1}^{N} \alpha_j V_j + \mathcal{F}^N_{\rm{nl}}(\vec{V}).
\end{align}


\subsection{Uniformly travelling front solutions: existence condition and singularities}
{As defined in  \eqref{eq:traveling_wave_intro}, a uniformly travelling or stationary front $Z_{\rm TW}$ is a heteroclinic orbit in the travelling wave ODE given by \eqref{eq:multi-component-RD} with $\partial_t$ replaced by $-\eps^2 c \partial_x$.} 
Concerning the existence and leading order expansion of such $Z_{\rm TW}$, we next formulate straightforward generalisations of the results from \cite{CBDvHR15}, whose proofs are rooted in spatial dynamics and GSPT. 
In particular, the existence of front solutions near the singular limit is equivalent to solving an algebraic equation 
{$\Gamma_\eps(c)=0$} 
in terms of the system parameters and the speed $c$. 
{Since $Z_{\rm TW}$ is also a fixed point of the speed ODE \eqref{eq:main_ODE}
, {\it i.e.} $ 
 \Geps(c, 0,\ldots, 0)= 0$, the functions  $\Gamma_\eps$ and $G_\eps(\cdot,0,\ldots,0)$ have the same roots.} 
We will also discuss some implications for singularity structures.  
\begin{lemma}[Uniformly travelling front solutions: existence condition]
\label{L:N}
For any choice of 
system parameters of \eqref{eq:multi-component-RD} with \eqref{N:NONL22} there exists at least one $c\in \mathbb{R}$ solving
\begin{align}
\label{E:N}
0= \Gamma_0(c) := \mathcal{F}^N(\vec{V}^*) - \frac13\sqrt2 c = \gamma + \sum_{j=1}^{N} \alpha_j V_j^* + \mathcal{F}^N_{\rm{nl}}(\vec{V}^*) - \frac13\sqrt2 c \,,
\end{align}
with
\begin{align}
\label{VSTARN}
\vec{V}^* = (V^*_1, \ldots, V^*_N)\,, \qquad
V^*_j =  
 \frac{c \tau_j}{\sqrt{4 d^2_j+c^2 \tau_j^2}}\,,  \qquad
 j = 1,2,\ldots,N \,.
\end{align}

For any sufficiently small $\eps>0$, 
there exist an adjusted choice of system parameters such that~\eqref{eq:multi-component-RD} supports a uniformly travelling front solution $Z_{\rm TF} = (U_{\rm TF},\vec{V}_{\rm TF})$ with leading order speed $\eps^2 c$ (in the slow scaling). For $c\neq 0$ this can be realised by an $\calO(\eps)$-adjustment of the parameter $\gamma$. Stationary front solutions, {\it{i.e.}}, $c=0$, exist for $\gamma=0$ and arbitrary other system parameters, and no adjustment is required for $\eps > 0$ sufficiently small.

The leading order profile $Z_{\rm TF}^0$ of the travelling front solution $Z_{\rm TF}$, in the co-moving coordinate $y = x - \varepsilon^2 c t$ (see Definition~\ref{CORD}) is given by
\begin{align}\label{N:PROFILES}
U_{\rm TF}^0(y) &= \left\{
\begin{aligned}
-1 &\,,& y \in I_s^-\,, \\
\tanh{(y/(\sqrt2 \eps))} &\,,& y \in I_f\,, \\ 
1 &\,,& y \in I_s^+\,,
\end{aligned}
\right. \quad
V^{0}_{{\rm{TF,}}j}(y) = \left\{
\begin{aligned}
(V_j^{*}+1)e^{\Lambda_j^+ y} -1 &\,,& y \in I_s^-\,, \\
V_j^{*} &\,,& y \in I_f\,, \\ 
(V_j^{*}-1)e^{\Lambda_j^- y} +1 &\,,& y \in I_s^+\,,
\end{aligned}
\right.
\end{align}
with $V^*_j$ given in \eqref{VSTARN}, $I_s^\pm$ and $I_f$ in Definition~\ref{fields},
and 
\begin{align*}
 \Lambda_j^{\pm} = - \frac{c \tau_j}{2 d_j^2} \pm \frac{1}{2 d_j^2}\sqrt{4 d^2_j+c^2 \tau_j^2}\,, \qquad
 j = 1,2, \dots, N \,.
\end{align*}
Conversely, there is $\delta>0$ such that for $\eps>0$ sufficiently small, any uniformly travelling front solution that is $\delta$-close to $Z_{\rm TF}^0$ belongs to this family.
\end{lemma}

\begin{proof}
The fact that \eqref{E:N} always has at least one solution follows from the fact that $\vec{V}^*$ \eqref{VSTARN} is bounded as function of $c$ and $\calF^N(\vec{V}^*)$ is smooth, while the last term in \eqref{E:N} is linear in $c$. 

The remaining statements directly follow from {a straightforward generalisation of} the results and proofs of Proposition~1, Corollary 1, and Lemma 12 of \cite{CBDvHR15}, so we omit the details. Here, we made the possibility to adjust $\gamma$ explicit, which is due to $\partial_\gamma \Gamma_0=1\neq 0$.
\end{proof}

Lemma~\ref{L:N} shows how the nonlinearity $\mathcal{F}_{\rm nl}^N(\vec{V})$ of the coupling function $\mathcal{F}^N(\vec{V})$ \eqref{N:NONL22} influences the structure of $\Gamma_0(c)$ \eqref{E:N}, which we call the {\it{existence function}}, and where the subscript $0$ indicates its leading order nature.  
The fact that the $\vec{V}^*$-dependent part of $\calF^N(\vec{V}^*)$ is always bounded in $c$  means that the coupling is a bounded perturbation of the relation $c =   3 \sqrt2 \gamma/2$, which relates to the speed of the travelling front solutions in the original Allen-Cahn equation \eqref{AC}.
Next, we are more specifically interested in the role of $N$ in the singularity structure of the existence condition in the sense of the possible Taylor polynomials $\calT_{\Gamma_0}^k(c)$ of $\Gamma_0(c)$ to order $k$, see Definition~\ref{TAYDEF}. The following proposition shows that with $N$ slow components and  $\calF^N(\vec{V})=\calF^N_\beta(\vec{V})$ \eqref{NONLnew}, we can typically realise an arbitrary $\mathcal{T}^{2N}_{\Gamma_0}(c)$ by means of the parameters $\gamma$ and $\alpha_i, \beta_i$. In \S\ref{S:sing} we will see that for $N=1$ we can also realise an arbitrary $\mathcal{T}^{2N}_{\Gamma_0}(c)$ by a suitable choice of $\calF^1(V_1)$ {that typically needs to be more general than} \eqref{NONLnew}. The advantage of fixing the simplest $\calF^N(\vec{V})=\calF^N_\beta(\vec{V})$ {but allowing general} $N$ is that, simultaneous to $\mathcal{T}^{2N}_{\Gamma_0}(c)$, we can control the dimension $\Np$ of the reduced dynamics {in the range $1\leq \Np\leq N$. This} relates to the multiplicity of the zero eigenvalue being at most $N+1$, cf.\ 
{Proposition~\ref{L:N+1}.} For $N=1$ and generic $\calF^1(V_1)$ this means that the dimension of the non-trivial reduced dynamics is at most one-dimensional, see also~\S\ref{S:sing}. 
Before we state the result, we first recall the definition of a double factorial. 
\begin{deff}[Double factorial]\label{def:double_factorial} The double factorial $n!!$ for $n \in \mathbb{N}$ denotes the product of all the positive integers up to $n$ that have the same parity (odd or even) as $n$.
\end{deff}


\begin{proposition}[Order of the singularity for the existence condition]
\label{L:EX}
Set $\mathcal{F}^N(\vec{V})= \mathcal{F}^N_\beta(\vec{V})$~\eqref{NONLnew}. Then, the Taylor series of $\Gamma_0(c)$~\eqref{E:N} near $c=0$ is given by
\begin{align}\label{EQ:EX_AFF}
\begin{aligned}
\mathcal{T}_{\Gamma_0}(c) =&
\gamma +  \left( \frac12 \sum_{j=1}^N\left(\alpha_j \dfrac{\tau_j}{d_j}\right)  - \frac13\sqrt2\right) c 
+ \sum_{k=1}^{\infty} (-1)^{k+1} \dfrac{1}{2^{2k}} \left( \sum_{j=1}^N \beta_j \dfrac{\tau_j^{2k}}{d_j^{2k}}\right) c^{2k} \\&
+ \sum_{k=1}^\infty (-1)^{k} \dfrac{(2k-1)!!}{(2k)!!2^{2k+1}} \left( \sum_{j=1}^N \alpha_j \dfrac{\tau_j^{2k+1}}{d_j^{2k+1}}\right) c^{2k+1}\,.
\end{aligned}
\end{align}
Assume $\tau_j/d_j >0$ are pairwise distinct for $1\leq j\leq N$. Then the coefficients of the Taylor polynomial 
$\mathcal{T}^{2N}_{\Gamma_0}(c)$,
see Definition~\ref{TAYDEF}, from the existence condition are in 1-to-1 correspondence with the $2N+1$ parameters $\gamma$, $\alpha_j, \beta_j$, $j=1,\ldots, N$. 
That is, there are parameter combinations such that $\calT_{\Gamma_0}(c) = \mathcal{O}(c^{\ell})$  for $\ell = 0,1, \ldots, 2N+1$.
Specifically, $\calT_{\Gamma_0}^{2N-1}(c)=0$, {\it{i.e.}}, $\calT_{\Gamma_0}(c) = \mathcal{O}(c^{2N})$, for $\gamma=0$ and 
\begin{align}
\label{EX:MAXa}
\alpha_j
=
\dfrac{2\sqrt2}{3}  \dfrac{d_j}{\tau_j} \prod_{k =1, k \neq j}^N \dfrac{(\tau_k/d_k)^2}{(\tau_k/d_k)^2-(\tau_j/d_j)^2}
\,, \qquad j=1, \ldots, N\,,
\end{align}
and for $(\beta_j)_{j=1}^N$ in the one-dimensional subset solving
\begin{align}
\label{EX:MAXb}
\sum_{j=1}^N \beta_j \dfrac{\tau_j^{2k}}{d_j^{2k}}=0
\,, \qquad k=1, \ldots, N-1\,.
\end{align}
Finally, $\mathcal{T}^{2N}_{\Gamma_0}(c)=0$, {\it i.e.}, $\calT_{\Gamma_0}(c) = \mathcal{O}(c^{2N+1})$, is the maximal degree of degeneracy with $\calF^N(\vec{V})=\calF^N_\beta(\vec{V})$ and requires $\beta_j =0$ for $j=1, \ldots, N$ so that $\mathcal{F}^N_\beta(\vec{V})$ \eqref{NONLnew} is affine, {\it{i.e.}}, $\mathcal{F}^N_{\rm{nl}}(\vec{V}) =0$ in~\eqref{N:NONL22}. 
\end{proposition}

Proposition~\ref{L:EX} generalises \cite[Lemma 2]{CBDvHR15} related to the maximal order of the singularity for~\eqref{1F2S}, {\it i.e.}, $N=2$, which in that case was restricted to $\mathcal{O}(c^5)$. As shown in \cite{CBDvHR15}  this leads to the coexistence of at most five uniformly travelling front solutions with different speeds. Proposition~\ref{L:EX} shows that increasing $N$ increases the possible number of coexisting front solutions with different speeds and, hence, also a possibly very intricate interplay of these.  
The proof of Proposition~\ref{L:EX} relies on the fact that Vandermonde matrices are invertible under mild conditions. As we need this result several times throughout this paper, we state it here for completeness. Next, and throughout the remainder of the paper, we use the short hand notation $\vec{0}_{k}$ to denote, depending on the context, a zero row or column vector of length $k$. Similarly, we use $\mathds{1}_{k}$ to denote a row or column vector of length $k$ with all ones.

\begin{lemma}[Vandermonde matrix]\label{L:VDM}
Let $M$ be a square $m \times m$-Vandermonde matrix\footnote{We note that $M^t$ is a more standard definition of a Vandermonde matrix.}
$$M:=\begin{pmatrix} m_{ik} \end{pmatrix} = \begin{pmatrix} 
m_{k}^{i-1}
\end{pmatrix}_{i,k=1}^{m}, \quad m_1,\ldots, m_m\in\R.$$
Then, the Vandermonde determinant is given by $$\det{(M)}=\prod_{1 \leq i < k \leq m} (m_k-m_i).$$   Moreover, if the Vandermonde determinant is non-zero, then the equation
$$
M \vec{x} = \begin{pmatrix} b\\ \vec{0}_{m-1}\end{pmatrix}\,,
$$
where $b \in \mathbb{R}$ and $\vec{x} \in \mathbb{R}^m$ ,
is solved by
\begin{align}
\label{VDM_ID}
\vec{x} = M^{-1} \begin{pmatrix} b\\ \vec{0}_{m-1}\end{pmatrix} = b \begin{pmatrix} \prod\limits_{k=1, k \neq i}^m\dfrac{ m_k}{m_k-m_i}\end{pmatrix}_{i=1}^m \,.
\end{align}
\end{lemma}
\begin{proof}[Proof of Lemma~\ref{L:VDM}]
The first part of this lemma is a well-known result and can, for instance, be shown by using row operations and an iterative process. We omit the details. For the second part of the proof see \cite[e.g]{MS58}.
\end{proof}

\begin{proof}[Proof of Proposition~\ref{L:EX}]
Expression \eqref{EQ:EX_AFF} follows directly from Taylor expanding \eqref{E:N} near $c=0$.
Equating the leading order term to zero gives $\gamma=0$. 
Furthermore, 
observe that the odd powered $c$-terms of \eqref{EQ:EX_AFF} do not depend on $\beta_j$, while the even powered $c$-terms do not depend on $\alpha_j$.\footnote{In particular, all the even powered $c$-terms of \eqref{EQ:EX_AFF} drop out when $\be_j=0$ for all $j$ and $\gamma=0$. That is, if the coupling function $\mathcal{F}^N(\vec{V})$~\eqref{NONLnew} is affine, then the existence function $\Gamma_0(c)$ is an odd function. This again relates to the original necessity to break the symmetry of~\eqref{eq:multi-component-RD}, see also Remark~\ref{ADV}.} This allows us to equate the first order correction term up to the $2N$-th order correction term of \eqref{EQ:EX_AFF} to zero as two independent matrix-vector equations. In particular, equating the first $N$ odd-powered $c$-terms of \eqref{EQ:EX_AFF} up to 
$\mathcal{O}(c^{2N-1})$ to zero gives 
\begin{align}\label{VECC}
\begin{aligned}
M^{\rm o} \vec \alpha = (m_{kj}^{\rm o}) \vec \alpha := 
\begin{pmatrix}
\left(\dfrac{\tau_j}{d_j}\right)^{2k-1} 
\end{pmatrix}_{k,j=1}^{N}
\begin{pmatrix}
\alpha_j
\end{pmatrix}_{j=1}^N = \begin{pmatrix}
\dfrac{2\sqrt2}{3} \\[2mm] \vec{0}_{N-1}
\end{pmatrix}\,.
\end{aligned}
\end{align}
Equating the first $N$ even-powered $c$-terms of \eqref{EQ:EX_AFF}, excluding the zeroth order term, up to 
$\mathcal{O}(c^{2N})$ to zero gives 
\begin{align}\label{EQ:EX2}
\begin{aligned}
M^{\rm e } \vec \beta = (m_{kj}^{\rm e}) \vec \beta:=
\begin{pmatrix}
\left(\dfrac{\tau_j}{d_j}\right)^{2k} 
\end{pmatrix}_{k,j=1}^{N}
\begin{pmatrix}
\beta_j
\end{pmatrix}_{j=1}^N = 
 \vec{0}_{N}
\,.
\end{aligned}
\end{align}
Next, we observe that both $N \times N$ square matrices $M^{\rm o}$ and $M^{\rm e}$ can be written as generalised Vandermonde matrices, that is,  
they both can be written as the
product of a Vandermonde matrix and a diagonal matrix. Upon introducing 
$\chi_j := (\tau_j/d_j)^2 \neq 0$ (by assumption) we have
$$
M^{\rm o} =\begin{pmatrix} \left(\dfrac{\tau_j}{d_j}\right)^{2k-1} 
\end{pmatrix}_{k,j=1}^{N} 
= 
\begin{pmatrix}
\chi_j ^{k-1} 
\end{pmatrix}_{k,j=1}^{N} 
\left(\diag\begin{pmatrix}
\sqrt{\chi_j}
\end{pmatrix}\right)_{j=1}^N\,,
$$
and
$$
M^{\rm e} =
\begin{pmatrix}
\left(\dfrac{\tau_j}{d_j}\right)^{2k} 
\end{pmatrix}_{k,j=1}^{N}
=
\begin{pmatrix}
\chi_j ^{k-1} 
\end{pmatrix}_{k,j=1}^{N} 
\left(\diag\begin{pmatrix}
\chi_j
\end{pmatrix}\right)_{j=1}^N\,.
$$
The matrix $(
\chi_j^{k-1})_{k,j=1}^{N}$ is now a standard Vandermonde matrix and, by Lemma~\ref{L:VDM}, we have that 
$$
\det(
\chi_j^{k-1})_{k,j=1}^{N} = \prod\limits_{1 \leq j < k \leq N} (\chi_k-\chi_j), $$
which is unequal to zero if $\chi_j$ are pairwise distinct for $j=1,\ldots,N$, {\it i.e.}, $\tau_j/d_j \neq \tau_k/d_k >0$, $j\neq k$.
Hence, by assumption, $(
\chi_j^{k-1})_{k,j=1}^{N}$ is invertible, and, consequently, $M^{\rm o}$ and $M^{\rm e}$ are invertible. 
This means that \eqref{VECC} and \eqref{EQ:EX2} can be uniquely solved for arbitrary right hand sides, which proves the claimed relation of the coefficients of the Taylor polynomials with the system parameters.
In addition, by \eqref{VDM_ID} of Lemma~\ref{L:VDM} we get from \eqref{EQ:EX2} that $\vec{\beta} = 0$
and from \eqref{VECC} that $\alpha_j$ is as given in \eqref{EX:MAXa}
\begin{align*}
\begin{aligned}
\vec{\alpha} &=
\begin{pmatrix}  
\alpha_j
\end{pmatrix}_{j=1}^N
 = 
\left(M^{\rm o}\right)^{-1}
 \begin{pmatrix}
\dfrac{2\sqrt2}{3} \\ \vec{0}_{N-1}
\end{pmatrix} =
\left(\diag\begin{pmatrix}
\dfrac{1}{\sqrt\chi_j} 
\end{pmatrix}\right)_{j=1}^N
\left(\begin{pmatrix}
\chi_j^{k-1} 
\end{pmatrix}_{k,j=1}^{N} \right)^{-1}  \begin{pmatrix}
\dfrac{2\sqrt2}{3} \\ \vec{0}_{N-1}\, 
\end{pmatrix}\\&=
\begin{pmatrix}
\dfrac{2\sqrt2}{3}  \dfrac{1}{\sqrt\chi_j} \displaystyle\prod_{k =1, k \neq j}^N \dfrac{\chi_k}{\chi_k-\chi_j}
\end{pmatrix}_{j=1}^N \,
=
\begin{pmatrix}
\dfrac{2\sqrt2}{3}  \dfrac{d_j}{\tau_j} \prod\limits_{k =1, k \neq j}^N \dfrac{(\tau_k/d_k)^2}{(\tau_k/d_k)^2-(\tau_j/d_j)^2}
\end{pmatrix}_{j=1}^N \,.
\end{aligned}
\end{align*}
That is, the existence function 
$\Gamma_0(c) = \mathcal{O}(c^{2N+1})$ if $\alpha_j$ is as given in \eqref{EX:MAXa} and the coupling term $\mathcal{F}^N_\beta(\vec{V})$ \eqref{NONLnew} is affine. Furthermore, it directly follows by equating 
only the first $2N-1$ terms
that $\Gamma_0(c) = \mathcal{O}(c^{2N})$ if $\alpha_j$ is still as given in \eqref{EX:MAXa}, while $\vec{\beta}$ solves the underdetermined solvable system~\eqref{EX:MAXb}. 

To complete the proof, we show that, in addition, it is not possible that also the 
next order odd-powered $c$-term of the Taylor series of the existence function
$\Gamma_0(c)$ equates to zero.
Say, $\Gamma_0(c) = \mathcal{O}(c^{2N+2})$, then equating the first $N+1$ odd-powered $c$-terms of \eqref{EQ:EX_AFF} to zero gives 
\begin{align*}
\begin{aligned}
\begin{pmatrix}
\left(\dfrac{\tau_j}{d_j}\right)^{2k-1} 
\end{pmatrix}_{k,j=1,1}^{N+1,N}
\begin{pmatrix}
\alpha_j
\end{pmatrix}_{j=1}^N = \begin{pmatrix}
\dfrac{2\sqrt2}{3} \\[2mm] \vec{0}_{N}
\end{pmatrix}\,.
\end{aligned}
\end{align*}
This is an overdetermined system and the last $N$ equations 
\begin{align}\label{EQ:EX22}
\begin{aligned}
\begin{pmatrix}
\left(\dfrac{\tau_j}{d_j}\right)^{2k-1} 
\end{pmatrix}_{k,j=2,1}^{N+1,N}
\begin{pmatrix}
\alpha_j
\end{pmatrix}_{j=1}^N = 
\vec{0}_{N}
\,,
\end{aligned}
\end{align}
give an invertible 
generalised Vandermonde matrix. 
Consequently, \eqref{EQ:EX22} is uniquely solvable and yields the trivial solution, {\it i.e.}, $\vec{\alpha}=\vec{0}_N$. This contradicts \eqref{EX:MAXa}.
\end{proof}

We 
illustrate the role of the quadratic terms in the coupling function $\calF^N$ in \eqref{eq:multi-component-RD} (which were absent in our previous papers \cite{CBDvHR15,CBvHHR19}) for the bifurcation scenarios of uniformly travelling front solutions. To this end, we set $N=2$, take $\calF=\mathcal{F}^N_\beta(\vec{V})$~\eqref{NONLnew}, so that Proposition~\ref{L:EX} guarantees control of the Taylor expansion of the existence function up to order $\mathcal{O}(c^5)$ in \eqref{EQ:EX_AFF}. In Figure~\ref{fig:butterfly} we plot the singularity structure for a specific set of parameters, which shows how the quadratic terms break the symmetry and deform the symmetric butterfly catastrophe, see also Figures~6 and 7 in~\cite{CBDvHR15}.
\begin{figure}
\centering
\includegraphics[width=0.7\textwidth]{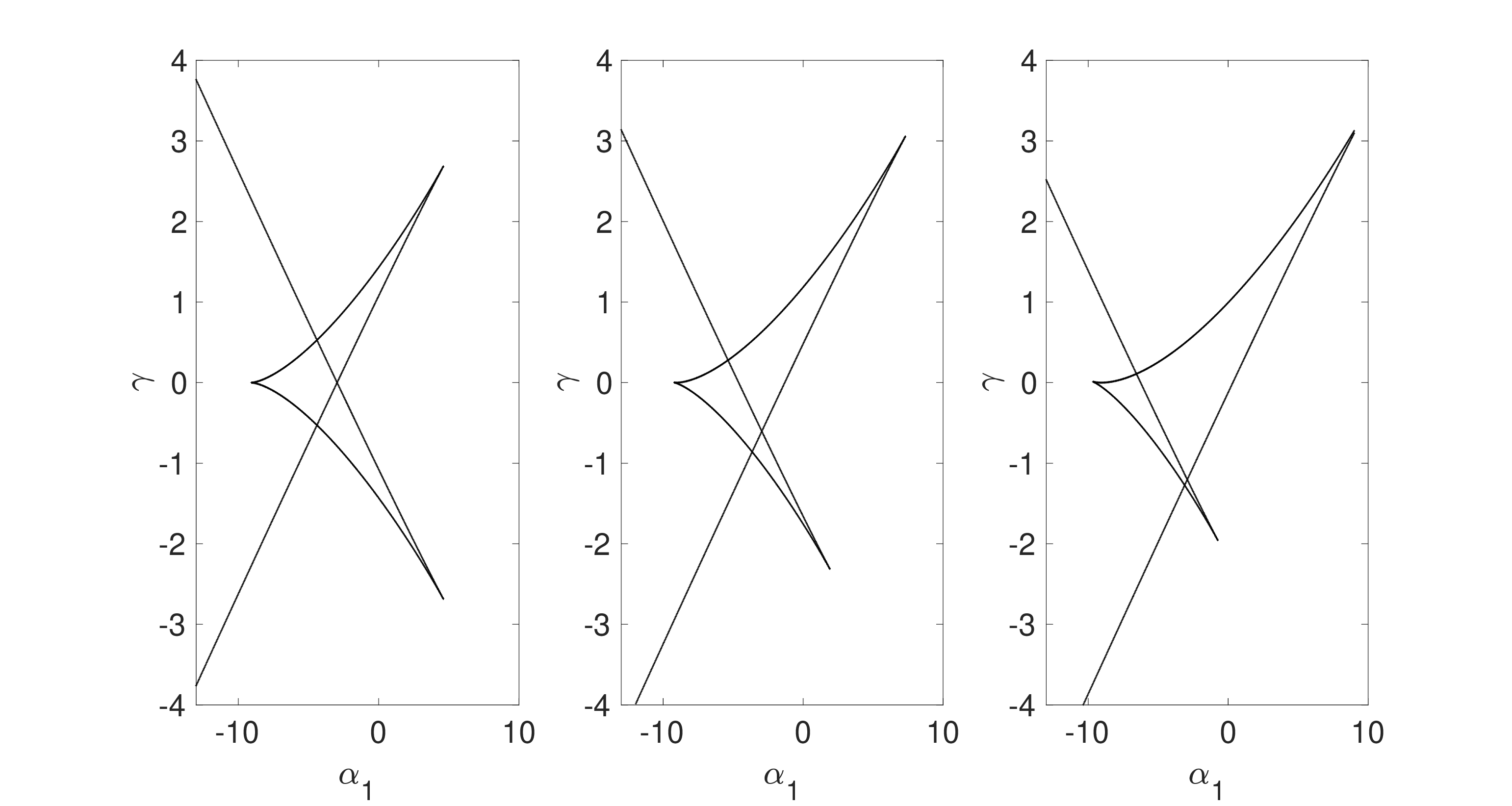}
\caption{
Curves of fold points of the existence condition $\Gamma_0(c)=0$~\eqref{E:N} in the $(\alpha_1, \gamma)$-plane for $\beta_1 = 0, 5, 10$ in the left, middle and right panel, respectively; the remaining system parameters are $ \alpha_2  = 4, \beta_2 = 0, \tau_1 = 1, \tau_2 = 5, d_1 = 1$ and $d_2 = 2$. The middle and right panels highlight the symmetry breaking effect of $\beta$ needed for a complete unfolding.}
\label{fig:butterfly}
\end{figure}

Next, we turn to the persistence of the results for $\eps>0$ sufficiently small.


\begin{proposition}[Singularities for $0<\eps\ll1$]\label{P:EXeps}
Consider \eqref{eq:multi-component-RD} with \eqref{N:NONL22}. There is $\delta>0$ and a function 
$\Gamma_\eps(c)$ depending smoothly on $\calF^N(\vec{V})$ and the parameters of \eqref{eq:multi-component-RD}, and that limits to $\Gamma_0(c)$ as $\eps\to0$ such that the following holds.
For any $\eps>0$ sufficiently small, the system parameters for which there exist travelling front solutions $Z_{\rm TF}$ that are $\delta$-close to $Z_{\rm TF}^0$ are in 1-to-1 correspondence with roots of  $\Gamma_\eps(c)$. 
Moreover, for $\mathcal{F}^N(\vec{V})= \mathcal{F}^N_\beta(\vec{V})$ from \eqref{NONLnew}, arbitrary Taylor polynomials of degree $2N$ of $\Gamma_\eps$ with respect to $c$  can be realised by suitable choice of the parameters of \eqref{eq:multi-component-RD} and $\mathcal{F}^N_\beta(\vec{V})$.
\end{proposition}

\begin{proof}
The function $\Gamma_\eps(c)$ and $\delta>0$ for the existence statement results from the statement and proof of Lemma~\ref{L:N} (analogous to \cite{CBDvHR15,CBvHHR19} for $N=2,3$). 
By Proposition~\ref{L:EX}, when choosing parameters $\gamma, \alpha_j, 
\beta_j$ from an open neighbourhood of zero in $\R^{2N+1}$, the coefficient vectors of the Taylor polynomials of $\Gamma_0$ form an open neighbourhood of zero in $\R^{2N+1}$. By continuity, {the perturbation of the latter for $\eps>0$ sufficiently small contains an open neighbourhood of zero}, which proves the claim.
\end{proof}

\subsection{Global attractor}
Before we analyse stability of the front solutions, we note that the PDE \eqref{eq:multi-component-RD} possesses a global attractor that, in particular, attracts all dynamic front solutions; see also \cite{Alfaro2008}. 


\begin{lemma}[Trapping region]\label{lem:trap}
For each $r>1$ there is $\eps_r>0$ such that for all $\eps\in(0,\eps_r)$ the set $M_r=\{-r<U,V_1,\ldots, V_N<r\}\subset\R^{N+1}$ is a trapping region for the PDE dynamics of \eqref{eq:multi-component-RD} in the sense that if $(U(x,t_0), \vec{V}(x,t_0))\in \overline{M_r}$ for all $x$, then $(U(x,t), \vec{V}(x,t) \in M_r$ for all $x$ and $t>t_0$. In particular, \eqref{eq:multi-component-RD} possesses a global attractor, whose elements have range contained in $M_{r_\eps}$ for {an $r_\eps\in(1, r]$}, and $r_\eps\to 1$ as $\eps\to 0$.
\end{lemma}

\begin{proof}
In absence of diffusion each $M_r$ is a trapping region for the kinetics ODE 
\begin{align*}
\left\{
\begin{aligned}
 \partial_t U & =    U - U^3 - \varepsilon  \mathcal{F}^N(\vec{V}) \\[.2cm]
 \tau_j \partial_t V_j & =   \varepsilon^2( U -  V_j )\, , \qquad j = 1, \ldots, N\,, 
 \end{aligned}
\right.
\end{align*}
for all $\eps>0$ sufficiently small. In particular, since $\calF^N(\vec{V})$ is smooth and $\vec{V}$ bounded on $\partial M_r$, there is for given $r>1$ an $\eps_r$ such that the vector field of the kinetics points into $M_r$, except at the corners where $\partial_t V_j=0$. However, at a corner $\partial_t U$ is non-zero and $(\partial_t U,0)$ is tangent to $\partial M_r$ away from the corner. In combination, $M_r$ traps the ODE flow, and conversely, {given any $r>1$,} for each {$\eps\in(0,\eps_r)$ there is $r_\eps\leq r$ such that $M_R$ is trapping for all $R\geq r_\eps$. In addition, 
$r_\eps\to 1$ as $\eps\to 0$ since $\eps\calF^N(\vec{V}) \to 0$ 
uniformly on $M_r$ and $r_0=1$ for trapping the $U$-component.} 
For a reaction-diffusion PDE with diagonal diffusion, like~\eqref{eq:multi-component-RD}, this implies the claims\cite[e.g. \S14.B]{Smoller}. 
Briefly, suppose a trajectory for initial state inside $M_r$ touches the boundary at some time $t>0$. If this occurs at $U=\pm r$ the sign of $\partial^2_{x} U$ is $\mp 1$ or zero, and analogously for $V_j$. 
Hence, the diffusion, independent of the spatial scale, does not work against the kinetics so that $M_r$ indeed {traps as claimed}. Since boundedness of the sup-norm implies global existence for parabolic flows, it follows that there exists a global attractor as claimed.
\end{proof}


\section{Stability of front solutions}
\label{S:1+N_S}
In this section, we determine the stability properties of the front solutions as constructed in the previous section, see in particular Proposition~\ref{L:S} below. 
More precisely, we determine the critical spectrum of the linearisation of \eqref{eq:multi-component-RD} around the travelling front solution  $Z_{\rm TF}$. 
We show that stationary front solutions $Z_{\rm SF}$, {\it i.e.}, travelling front solutions $Z_{\rm TF}$ with speed zero, can have a zero eigenvalue of multiplicity (up to) $N+1$ with $N$ generalised eigenfunctions, see Proposition~\ref{L:N+1}, Lemma~\ref{lem:evals} and Lemma~\ref{L:GEIGSG}. These
results will also form the basis for the upcoming center manifold reduction and subsequent analysis of \S\ref{S:DFS}.

Upon introducing $Z=(U,\vec{V})= (U,V_1, \ldots, V_N)$, we write \eqref{eq:multi-component-RD} as 
\begin{align}\label{eq:Zeqn}
 \partial_t Z = \mathcal{G}(Z; \P, \eps)\,,
\end{align}
with $\P$ the abstract short hand notation denoting all system parameters and 
$$
\mathcal{G}(Z; \P, \eps) := M(\vec{\tau})^{-1} \begin{pmatrix}
\eps^2 \partial_{x}^2 U  +  U - U^3 - \varepsilon \mathcal{F}^N(\vec{V})\\[2mm]
\eps^2 d_1^2 \partial_x^2{V_1}  + \eps^2(U - V_1) \\ \vdots \\
\eps^2 d_N^2 \partial_x^2{V_N}  + \eps^2(U - V_N)
\end{pmatrix}\,,
$$
with $M(\vec{\tau}):= \diag(1, \tau_1, \ldots, \tau_N)$. The linearisation around a uniformly travelling front solution $Z_{\rm TF}$ gives rise to the linear operator 
$\partial_Z \mathcal{G}(Z_{\rm TF}; \P, \eps)$ whose spectrum determines the spectral stability of the travelling front solution. 
We will be mostly concerned with stationary front solutions $Z_{\rm SF}=(U_{\rm SF}, \vec{V}_{\rm SF})$ and denote 
\begin{align}
\label{Ldef}
\mathcal{L} := \partial_Z \mathcal{G}(Z_{\rm SF}; \P, \eps)\,.
\end{align}
However, if the context is clear, we also use $\mathcal{L}$ for $\partial_Z \mathcal{G}(Z_{\rm TF}; \P, \eps)$.
In the following, by the spectrum associated with a uniformly travelling or stationary front solution we mean the spectrum of the corresponding operator, that is, $\sigma(\mathcal{L})$, which can be decomposed (as described in \cite[e.g.]{henry2006geometric, sandstede2002stability}) into point spectrum and essential spectrum:
$\sigma(\mathcal{L}) = \sigma_{\rm ess}(\mathcal{L}) \cup \sigma_{\rm pt}(\mathcal{L}) \, .$

\subsection{Travelling front solutions: stability condition}
\begin{proposition}[Stability of uniformly travelling front solutions and the Evans function]
\label{L:S}
{Consider} a uniformly travelling front solution $Z_{\rm TF} = (U_{\rm TF},\vec{V}_{\rm TF})$ {given by Lemma~\ref{L:N}}.
For $\varepsilon > 0$ sufficiently small this front solution is orbitally exponentially stable if and only if, besides the zero solution, all roots of
\begin{align}
\label{EV}
E_0(\lambda) = 
\lambda + 3\sqrt2\sum_{j=1}^N  
\partial_j\mathcal{F}^N(\vec{V}^{*})
\left(\frac{1}{\sqrt{c^2 \tau_j^2+ 4 d_j^2 (\lambda \tau_j+1)}}-\frac{1}{\sqrt{4 d_j^2+ c^2 \tau_j^2}} \right) \,,
\end{align}
have negative real part. Here, $\vec{V}^{*}$ is given in \eqref{VSTARN} and
$
\partial_j \mathcal{F}^N(\vec{V}^*) :=\left.\dfrac{\partial \mathcal{F}^N(\vec{V})}{\partial V_j}\right|_{\vec{V} = \vec{V}^*} \,$. 

{More precisely, for real parts bigger than 
$-\eps^2/\max\{\tau_j: j = 1,2, \ldots, N\}  +  \calO(\eps^3)$, 
the point spectrum associated to $Z_{\rm TF}$ is given by 
$\{\eps^2 \lambda\,:\, \lambda \text{ is a root of } E_\eps\}$,
where $E_\eps:\mathbb{C}\to\mathbb{C}$ are analytic functions, smooth in the parameters $P$ and in $\eps>0$, and continuous in $\eps\geq 0$.}
\end{proposition}
\begin{proof}

This follows from combining the upcoming Lemma~\ref{L:SE} and Lemma~\ref{L:SP}.
\end{proof}
The proposition generalises Theorem~1 of~\cite{CBDvHR15}, which, in turn, depends on several intermediate results of that paper. We will incautiously call \eqref{EV} the {\it{Evans function}}
{since it defines the factor of} the Evans function \cite{AGJ90}{, whose roots give the (relevant) point spectrum with multiplicity. It 
always has the root $\lambda=0$, which is} indeed an eigenvalue for $\eps \ge 0$ and it is related to the translation invariance of the system and its corresponding eigenfunction is $\partial_y Z_{\rm TF}$. 
The following lemma generalises Lemma~3.2 of \cite{vHDK08} related to the essential spectrum of a standing pulse solution in the original three-component system~\eqref{1F2S}, see also Lemma~4 of~\cite{CBDvHR15}.
\begin{lemma}[Essential spectrum]
\label{L:SE}
There exists $C > 0$ such that for $\varepsilon > 0$ sufficiently small the essential spectrum 
$\sigma_{\rm ess}(\mathcal{L})$ 
associated to a uniformly travelling front solution $Z_{\rm TF}$ of \eqref{eq:multi-component-RD} with \eqref{N:NONL22} as constructed in Proposition~\ref{P:EXeps} lies in the left-half plane 
$\{\lambda \in \mathbb{C} \, | \, \Re(\lambda) < -\eps^2\chi\}$, with $\chi = \min\limits_{j = 1,2, \ldots, N} \left\{1/\tau_j  -  \eps C\right\}.$
\end{lemma}
Since the essential spectrum is contained in the left half plane, it does not lead to instabilities.
However, we note that it is $\eps^2$-close to the imaginary axis, see also~\cite{CBDvHR15, CBvHHR19,vHDK08}.
\begin{proof}[Proof of Lemma~\ref{L:SE}]
In the fast co-moving frame $z := y/\eps := \xi - \varepsilon c t$, see Definition~\ref{CORD}, the $N+1$-component system 
\eqref{eq:multi-component-RD} becomes
\begin{align}\label{EQ:1FNS_TW}
\left\{
\begin{aligned}
\partial_t U -\varepsilon c \partial_z U_z&= \partial^2_zU  + \ U - U^3 - \varepsilon \mathcal{F}^N(\vec{V}) \, ,\\[.3cm]
\tau_j \partial_t V_j - \varepsilon c \tau_j \partial_zV_j&=  d_j^2 \partial^2_z V_j  + \ \varepsilon^2 (U - V_j) \ , \qquad j = 1, \ldots, N \,.
\end{aligned}
\right.
\end{align}
To determine bounds on the essential spectrum of~$Z_{\rm TF}$, we linearise around the background states of $Z_{\rm TF}$, that is,  around 
$(U,\vec{V})=u_\varepsilon^{\pm}\mathds{1}_{N+1}$ 
with
$$
u_\varepsilon^{\pm} = \pm 1 - \frac{1}{2} \varepsilon \mathcal{F}^N(\pm \mathds{1}_{N}) + \mathcal{O}(\varepsilon^2),$$
where we recall that $\mathds{1}_{N}$ is a vector of length $N$ with all ones.
In particular, we substitute
$$
(U, \vec{V}) = (U, V_1, \ldots, V_N) = u_\varepsilon^{\pm}\mathds{1}_{N+1}+ (u,v_1, \ldots, v_N)e^{i k z + \omega t}\,,
$$
into \eqref{EQ:1FNS_TW} and linearise to obtain the matrix-vector equation $M_1 \vec{w} = 0$,
where $\vec{w} := (u,\vec{v}) =  (u,v_1, \ldots, v_N)^t$ and
$M_1$ is the $(N+1) \times (N+1)$-matrix 
\begin{align*}
M_1  = 
\begin{pmatrix}
K{_\eps(\omega,k)} & -\eps\partial_j \mathcal{F}^N(u_\eps^\pm\mathds{1}_{N} )\\[2mm]    \eps^2 u^\pm_\eps \mathds{1}_{N} 
& \diag(-\tau_j \omega + \eps i k c \tau_j - d_j^2 k^2 -\eps^2)
\end{pmatrix}
\,,
\end{align*}
with 
$
K_\eps(\omega,k) = -k^2-2-\omega + \eps i k c \mp  3 \eps \mathcal{F}^N(\pm \mathds{1}_{N}) + \mathcal{O}(\eps^2) 
$. Its {determinant} 
reads
\[Q_\eps(\omega, k) =P_\eps(\omega,k) + \eps^3 R_\eps(\omega,k)\,,\] 
where $R_\eps$ is a polynomial of degree $N-1$ in $\omega$ and $2N-2$ in $k$ and
\[
P_\eps(\omega,k) = K_\eps(\omega,k) \prod_{j=1}^{N} (-\tau_j \omega + \eps i k c \tau_j - d_j^2 k^2 -\eps^2).
\]
Since $P_\eps$ is a polynomial of degree $N+1$ in $\omega$ and $2N+2$ in $k$, the non-leading coefficients of $Q_\eps$ are those of $P_\eps$ with an order $\eps^3$-perturbation due to $R_\eps$. 
The solutions of $P_0(\omega,k)=0$ are non-degenerate and, to the relevant order in $\eps$, the real parts are given by
$$
\Re(\omega_0) = -2 - k^2\,, \qquad \Re(\omega_j) =  -\frac{d_j^2 k^2 +\eps^2}{\tau_j}\,, \qquad j=1,\ldots, N.
$$
Therefore, the solutions of $Q_\eps(\omega,k)=0$ are $\eps^3$-near to these, locally uniformly in $k$. Consequently, the essential spectrum is in the left half plane for sufficiently small $\eps>0$.
\end{proof}

We generalise Lemma~6 of \cite{CBDvHR15} to obtain the following result {about the point spectrum.}
\begin{lemma}[Point spectrum]
\label{L:SP}
Let $Z_{\rm TF}$ be a uniformly travelling front solution of \eqref{eq:multi-component-RD} with~\eqref{N:NONL22} 
as constructed in Lemma~\ref{L:N} and Proposition~\ref{P:EXeps}. For $\varepsilon > 0$ sufficiently small the point spectrum~$\sigma_{\rm pt}(\mathcal{L})$ associated to $Z_{\rm TF}$
whose
real parts are
bigger than the maximum real part of the essential spectrum is given by 
\begin{align*}
\{\eps^2 \lambda\,:\, \lambda \text{ is a root of } E_\eps\},    
\end{align*}
where $E_\eps:\mathbb{C}\to\mathbb{C}$ are analytic functions, smooth in the parameters $P$ and in $\eps>0$, and continuous in $\eps\geq 0$ with limit $E_0$ from \eqref{EV}.
\end{lemma}
\begin{proof}
Linearising \eqref{EQ:1FNS_TW} around a uniformly travelling front solution $Z_{\rm TF}= (U_{\rm TF}, \vec{V}_{\rm TF})$
results in the following eigenvalue problem 
\begin{align}\label{EQ:1FNS_EVPW}
\left\{
\begin{aligned}
\lambda u &= \left(\partial_{z}^2  + \ 1 - 3U_{\rm TF}^2 + \eps c \partial_z \right) u - 
\varepsilon \nabla \mathcal{F}^N(\vec{V}_{\rm TF}) \cdot  \vec{v}\,,\\
\tau_j \lambda v_j &=  \left(d_j^2 \partial_{z}^2 + \varepsilon c \tau_j \partial_z - \eps^2\right) v_j + \eps^2 u, \qquad \qquad j = 1,2, \ldots, N \,,
\end{aligned}
\right.
\end{align}
where $\nabla$ is the usual gradient operator and $\vec{v}= (v_1, \ldots, v_n)$.
Using the slow-fast separation of the problem at hand {implies the structure of} the Evans function~\cite{AGJ90, D98, D01, D02, vHDK08} given by 
$$
\mathcal{D}(\lambda) = d(\lambda) \mathcal{D}_f(\lambda) \mathcal{D}_s(\lambda),
\quad \calD_f =t_1^+(\lambda),\;  \calD_s =  \det{[s_{ij}^+(\lambda)]}_{i,j=1}^{N}\,,
$$
where we suppress the explicit $\eps$ dependence in the notation. Here, $d(\lambda)$ is an analytic function without roots (determined up to a scaling constant), $\calD_f=t_1^+(\lambda)$ is also an analytic function and referred to as the {\it{fast-fast transmission function}},
and $s_{ij}^+(\lambda), i,j=1, \ldots N$, are meromorphic, so-called, {\it{slow-fast transmission functions}} corresponding to $\mathcal{D}_s(\lambda)$.
Notably, poles of $\mathcal{D}_s(\lambda)$ are necessarily roots of $ \mathcal{D}_f(\lambda)$ \cite{vHDK08}. Hence, roots of $\mathcal{D}_f(\lambda)$ need not be roots of $\mathcal{D}(\lambda)$, which relates to the so-called NLEP paradox~\cite{D98, D01, D02}. 

Following the proofs of \cite[Lemma~4.4]{vHDK08} and \cite[Lemma 6]{CBDvHR15},  the fast-fast transmission function satisfies
$
t_1^+(\lambda) = \lambda - \eps^2 \lambda_{\rm fast},
$
with, to leading order, 
$$
 \lambda_{\rm fast} =3\sqrt2\sum_{j=1}^N  \frac{\partial_j \mathcal{F}^N(\vec{V}^{*})}{\sqrt{4 d_j^2+ c^2 \tau_j^2}}   \,,
$$
where $\vec{V}^*$ is as given in \eqref{VSTARN}.
In particular, the roots $\lambda$ of $\mathcal{D}$ scale as $\eps^2$. 

The roots associated to the slow-fast transmission functions have the same $\mathcal{O}(\eps^2)$-asymptotic scaling. So, $\mathcal{D}$ is a function of $\hat\lambda = \eps^2 \lambda$. With abuse of notation, we write $\lambda$ for $\hat\lambda$ and  
the slow-fast transmission functions are given by 
\begin{align*}
s_{ii}^+(\lambda) = 1 - \frac{2\sqrt2}{\sqrt{G_i}} C_i\,, \quad i=1,\ldots, N, \quad {\rm and} \quad
s_{ij}^+(\lambda) =  - \frac{2\sqrt2}{\sqrt{G_j}} C_i\,, \quad i,j=1,\ldots, N, i \neq j\,,
\end{align*}
with $N$ integration constants $C_i$ and 
$
G_i = c^2 \tau_i^2+ 4 d_i^2 (\lambda \tau_i+1),
$
see \cite{CBDvHR15} for the details.
Upon defining the column vectors $\vec{n},\vec{m}$ via $n_j:=2\sqrt{2/G_j}$ and $m_j:=C_j, j=1,\ldots,N$ and upon using the matrix determinant lemma \cite[e.g.]{harville1998matrix}, we observe that we can write~$\mathcal{D}_s(\lambda)$ as 
$$
\mathcal{D}_s(\lambda) = \det{[s_{ij}^+(\lambda)]}_{i,j=1}^{N} = \det{[I - \vec{m}\vec{n}^t]} = 1- \sum_{j=1}^N\frac{2\sqrt2}{\sqrt{G_j}} C_j \,.
$$
Finally, from a solvability condition, again see \cite{CBDvHR15, vHDK08}, we get  
$
C_j = -3 \partial_j \mathcal{F}^N(\vec{V}^{*})/(2 t_1^+(\lambda)),
$
where we explicitly observe the NLEP paradox.

Combining the above expressions gives the Evans function 
\begin{align}
\label{EVANS}
\begin{aligned}
\mathcal{D}(\lambda) =& d(\lambda)t_1^+(\lambda)\left(1+ \left(\sum_{j=1}^N\frac{3\sqrt2}{\sqrt{c^2 \tau_j^2+ 4 d_j^2 (\lambda \tau_j+1)}} \left(  \frac{\partial_j \mathcal{F}^N(\vec{V}^{*}) }{t_1^+(\lambda)} \right) \right)\right)\, \\ 
=&
d(\lambda)\left( \lambda + 3\sqrt2\sum_{j=1}^N  
\partial_j \mathcal{F}^N(\vec{V}^{*})\left(\frac{1}{\sqrt{c^2 \tau_j^2+ 4 d_j^2 (\lambda \tau_j+1)}}-\frac{1}{\sqrt{4 d_j^2+ c^2 \tau_j^2}} \right)\right)
\,,
\end{aligned}
\end{align}
and the point spectrum of interest near the imaginary axis is thus determined by the roots of \eqref{EV}. {(For $c=0$ 
and disregarding multiplicity, see also the calculation in Appendix~\ref{A:E}.)}
\end{proof}

\subsection{Stationary front solutions: linear structure}
The main objective of this section is to prepare the center manifold reduction in \S\ref{S:DFS} which forms the basis of the main results. To this end one needs a fine tuning of the {\emph{critical spectrum}} of $\mathcal{L}$ from \eqref{Ldef}, which is equivalent to controlling the roots of the Evans function (as stated in Proposition~\ref{L:S}/Lemma~\ref{L:SP}). Here `critical' refers to eigenvalues that can become unstable upon parameter variation; although fixed at zero, we include the trivial translation eigenvalue among these. We analogously refer to critical roots of the Evans function. Since a stationary front solution $Z_{\rm SF}$ is a uniformly travelling front solution with $c = 0$, the Evans function \eqref{EV}/\eqref{EVANS} significantly simplifies. In particular, it becomes independent of the nonlinear part of \eqref{N:NONL22}, {\it i.e.}, it is independent of $\mathcal{F}_{\rm nl}^N(\vec{V})$ because $V^*_j=0$ for $c=0$, see \eqref{VSTARN}, and $\mathcal{F}_{\rm nl}^N(\vec{0})=0$ by construction. Hence, the Evans function \eqref{EV} reduces to
\begin{align}
\label{EV0}
E_0(\lambda) = 
\lambda + \frac{3\sqrt2}{2}\sum_{j=1}^N  
\dfrac{\alpha_j}{d_j} \left(\frac{1}{\sqrt{\lambda \tau_j+1}}-1 \right) = 0\,,
\end{align}
where $\alpha_j$ are the coefficients of the linear part of $\mathcal{F}^N(\vec{V})$, see \eqref{N:NONL22}.
We caution the reader that in the following $E_0$ and $E_\eps$ refer to the evaluation of the Evans functions at $c=0$. 


\subsubsection{Organising centers of maximal degeneracy $N+1$}\label{s:organising}
The following result shows that the root of the Evans function~\eqref{EV0} at $\lambda=0$ can have algebraic multiplicity up to $N+1$. This result will be pivotal in the upcoming sections, especially in the search for chaotic behaviour. Indeed, in \S\ref{s:gen_efunc} we prove that the generalised kernel of the linearisation $\calL$ from \eqref{Ldef} around $Z_{\rm SF}$ is of this relatively high dimension. 


\begin{proposition}[Multiplicity of the root $\lambda = 0$ for $c = 0$]
\label{L:N+1}
A stationary front solution~$Z_{\rm SF}$ of~\eqref{eq:multi-component-RD} with \eqref{N:NONL22} necessarily has $\gamma=0$ and 
the Taylor expansion of the Evans function $E_0(\lambda)$ at $c=0$, see \eqref{EV0}, in $\lambda=0$  
is given by 
\begin{align}
\label{eq:taylor_evansN}
\mathcal{T}_{E_0}(\lambda) =  
\left(1- \frac{3 \sqrt2}{4}\sum_{j=1}^N  \frac{\alpha_j \tau_j }{d_j} \right) \lambda 
+ \frac{3\sqrt2}{2}\sum_{k=2}^\infty \left( (-1)^k\frac{(2k-1)!!}{(2k)!!} \sum_{j=1}^N\ \frac{\alpha_j \tau_j ^k}{d_j} \right) \lambda^k \,,
\end{align}
with the double factorial as in Definition~\ref{def:double_factorial}.
Furthermore, the algebraic multiplicity of $\lambda=0$ as a root of the Evans function is $N+1$ if and only if $\tau_j \neq \tau_k > 0$ for $j\neq k$,
and 
\begin{align}
\label{N:ZERO}
\begin{aligned}
 \alpha_j
 = 
 \dfrac{2 \sqrt 2 d_j}{3 \tau_j} \prod_{k =1, k \neq j}^N \dfrac{\tau_k}{\tau_k-\tau_j}
\,, \qquad \qquad j=1, \ldots, N\,.
\end{aligned}
\end{align}
Moreover, the $N$ non-constant terms of the Taylor polynomial $\mathcal{T}^{N}_{E_0}(\lambda)$ of $E_0(\lambda)$ of degree $N$ are in 1-to-1 correspondence with the parameters $\alpha_j$, $j=1,\ldots,N$. 
In particular, there are parameter combinations such that $\mathcal{T}_{E_0}(\lambda) = \mathcal{O}(\lambda^{\ell+1})$  for any $\ell = 0,1, \ldots, N${, but none for $\ell>N$}.
\end{proposition}

As expected, condition \eqref{N:ZERO} is independent of the nonlinear part $\mathcal{F}_{\rm nl}^N(\vec{V})$ of the coupling function $\mathcal{F}^N(\vec{V})$~\eqref{N:NONL22}. 
Since \eqref{N:ZERO} means $\al_j \neq 0$ for all $j = 1, \ldots, N$, all components of the linear part of the coupling function must be non-trivial in order to create a zero root of the Evans function of algebraic multiplicity $N+1$. Furthermore, we note that $\ell$ relates to $\Np$ in Main Result~\ref{thm:intro}, {\it i.e.}, the dimension of the speed ODE. We refer to \S\ref{S:CMR} for more details.

\begin{proof}
The existence condition $\Gamma_0(c)=0$ of Lemma~\ref{L:N} immediately implies that a stationary front solution $Z_{\rm SF}$ necessarily has $\gamma=0$. 
Next, expanding the Evans function $E_0(\lambda)$ at $c=0$, see~\eqref{EV0}, around $\lambda=0$ results in \eqref{eq:taylor_evansN}. 
Equating \eqref{eq:taylor_evansN} to zero at each $\lambda$-order gives
\begin{align}
\label{TAYLOR}
\mathcal{O}(\lambda): \sum_{j=1}^N  \frac{\alpha_j \tau_j }{d_j}= \dfrac{2\sqrt2}{3}\,, \qquad 
\mathcal{O}(\lambda^k):\sum_{j=1}^N\ \frac{\alpha_j \tau_j^k}{d_j}= 0\,.
\end{align}
The first $n$ conditions in \eqref{TAYLOR} can be written as a matrix-vector equation
\begin{align}
\label{VDM}
M \begin{pmatrix}
\dfrac{\alpha_j}{d_j}
\end{pmatrix}_{j=1}^{N} 
:= 
\begin{pmatrix}
\tau_j^k
\end{pmatrix}_{k,j=1}^{n,N}
\begin{pmatrix}
\dfrac{\alpha_j}{d_j}
\end{pmatrix}_{j=1}^{N} = 
\begin{pmatrix}
\dfrac{2\sqrt2}{3} \\[2mm] \vec{0}_{n-1}
\end{pmatrix}\,.
\end{align}
For $n=N$, the matrix $M$ is, again, a generalised Vandermonde matrix
$$
\begin{pmatrix}
\tau_j^k
\end{pmatrix}_{k,j=1}^{N}
=
\begin{pmatrix}
\tau_j^{k-1}
\end{pmatrix}_{k,j=1}^{N}
(\diag(\tau_j))_{j=1}^N\,,
$$
and from \eqref{VDM_ID} of Lemma~\ref{L:VDM} we get
\begin{align*}
\begin{aligned}
\begin{pmatrix}
\dfrac{\alpha_j}{d_j}
\end{pmatrix}_{j=1}^{N} &= 
\left(\begin{pmatrix}
\tau_j^k
\end{pmatrix}_{k,j=1}^{N}\right)^{-1}
\begin{pmatrix}
\dfrac{2\sqrt2}{3} \\ \vec{0}_{N-1}
\end{pmatrix}
= 
\left(\diag\left(1/\tau_j\right)\right)_{j=1}^N
\left(\left(
\tau_j^{k-1}\right)_{k,j=1}^{N}\right)^{-1}
\begin{pmatrix}
\dfrac{2\sqrt2}{3} \\ \vec{0}_{N-1}
\end{pmatrix}\,
\\
&= \left( \dfrac{2 \sqrt 2 }{3 \tau_j} \prod_{k =1, k \neq j}^N \dfrac{\tau_k}{\tau_k-\tau_j}\right)_{j=1}^N\,.
\end{aligned}
\end{align*}
This proves \eqref{N:ZERO}. 

The observation that the multiplicity cannot be larger follows from a similar argument as in the proof of Proposition~\ref{L:EX}. Say the multiplicity is $N+2$ (or larger), then the first $N+1$ equalities of \eqref{TAYLOR} can be written as {in \eqref{VDM} with $n=N+1$.} This is an overdetermined system and the last $N$ equations
are given by
\begin{align}
\label{VDM1}
\begin{pmatrix}
\tau_j^k
\end{pmatrix}_{k,j=2,1}^{N+1,N}
\begin{pmatrix}
\dfrac{\alpha_j}{d_j}
\end{pmatrix}_{j=1}^{N} = 
 \vec{0}_{N}\,.
\end{align}
As before, the generalised Vandermonde matrix $(
\tau_j^k
)_{k,j=2,1}^{N+1,N}$ is an invertible matrix as long as $\tau_i \neq \tau_j \neq 0$. Consequently,
\eqref{VDM1}
is uniquely solvable and yields the trivial solution, {\it i.e.}, $\alpha_j=0$ for all $j=1, \dots, N$. 
This contradicts \eqref{N:ZERO} and proves the claim about the Taylor polynomials (whose constant coefficient is always zero).
\end{proof}


\subsubsection{Full linear unfolding at $\eps=0$ around maximal degeneracy of order $N+1$}\label{s:fulllinear}
Proposition~\ref{L:N+1} shows that parameters satisfying \eqref{N:ZERO} and $\gamma=0$, 
which we denote as
\begin{align}\label{e:Porg}
P=P_{\rm org}^0,
\end{align} 
create an organising center such that the Evans function $E_0$ possesses a root at zero of multiplicity $N+1$. 
\begin{remark}
    In the following, whenever we refer to $P_{\rm org}^0$ (and the related $P_{\rm org}^\eps, P_\Np^\eps$ introduced below), we assume that $\tau_j$ are pairwise distinct for $1\leq j\leq N$ or $1\leq j\leq\Np$, respectively.
\end{remark}
As a basis for further bifurcation analysis, we next show that this can be unfolded in the sense that arbitrary configurations of roots of $E_0$ can be realised by parameters $P=P_{\rm org}^0+\cP$ with $\cP$ near zero (except for the root related to the fixed translation eigenvalue at zero). Moreover, we show that also for any sufficiently small $\eps>0$
the actual eigenvalue configurations associated to a stationary front solution $Z_{\rm SF}$ can be controlled in this way: there are parameters $P=P_{\rm org}^{\eps}$ such that  $E_\eps$ from Lemma~\ref{L:SP} possesses an $(N+1)$-fold root and perturbing $P$ away from this value allows to realise any 
small eigenvalue configurations associated to a stationary front solution $Z_{\rm SF}$ (except for the fixed translation eigenvalue at zero). For this purpose, we analyse $E_0$ in more detail and fix $\eps=0$ for now. Before entering into
the analysis, we stress that this control of eigenvalues
is also true for any multiplicity~$\ell, 1 \leq \ell \leq N + 1,$ and can be accomplished via the same techniques, but with a smaller number of parameters. For readability, we only show the case of highest degeneracy in all detail. We remove the factor $\lambda$ that is due to the translation symmetry from $E_0$ and set 
\begin{align*}
\uE(\lambda):=E_0(\lambda)/\lambda,
\end{align*} 
which we consider as extended to an analytic function away from the branch cut (which relates to the essential spectrum) that lies in the open left half plane. In particular, for $P=P_{\rm org}^{0}$, {\it{i.e.}} $\cP=0$, $\uE$ possesses a root at $\lambda=0$ of multiplicity $N$. The first question is then, which locations of the roots of $\uE$ can be realised by suitable choices of $\cP\approx 0$.

The Weierstrass preparation theorem provides, for $\cP\approx 0$, a function $\tE(\lambda)\neq 0$ that is analytic in $\lambda$ and smooth in $\cP$, and functions $\ta_j$, $j=1,\ldots,N$, that are 
smooth in $\cP$, such that 
\begin{align}
\uE(\lambda) 
& = \sum_{i=0}^\infty \ec_i\lambda^i = (\lambda^N - \ta_N \lambda^{N-1} - \ta_{N-1} \lambda^{N-2}\ldots - \ta_1) \tE(\lambda),
\quad \tE(\lambda) = \sum_{i=0}^\infty \te_i\lambda^i.\label{e:evansWeierN}
\end{align}
In particular, $\ec_i, \te_i$ are smooth functions of $\cP$, and $\ec_i$ can be read off \eqref{eq:taylor_evansN}. Let the super-index~$*$ denote evaluation at $\cP=0$. The $N$-fold root at $\cP=0$ gives $\ta_j^*=0$, $j=1,\ldots,N$, and $\te_0^*=e_N^*\neq 0$. For $\cP\approx 0$, the $\ta_j$ are the coefficients of a characteristic polynomial for the roots of $\uE$ near zero, which will be related to the linearisation of the center manifold reduction in \S\ref{S:DFS}. In particular, $(\ta_j)_{j=1}^N$ directly control the location of the roots of $\uE$ near zero for $\cP\approx0$ and, as we will show, thereby the location of critical eigenvalues associated with $Z_{\rm SF}$ in the complex plane. The following lemma means that with $\cP\approx 0$ we can create arbitrary configurations of the roots of the Evans function $E_0$ in a neighbourhood of zero (except for the fixed translation eigenvalue at zero). It also provides quantitative relations for the coefficients, although the specific evaluation is tedious.


\begin{lemma}[Full linear unfolding around $\lambda = 0$ at $\eps=0$]\label{l:linearunfold}
There is $\delta>0$ such that for any $P=P_{\rm org}^{0}+\cP$ with $P_{\rm org}^{0}$ as in \eqref{e:Porg} and $|\cP|<\delta$ there is a lower triangular matrix $B=B(\cP)$ with the following properties. Its diagonal entries are all identical and non-zero, and at $\cP=0$ it reads
\[
B^*= 
\begin{pmatrix}
e_N^* & 0 &\cdots &0\\
e_{N+1}^* &\ddots & \ddots &\vdots\\
\vdots & \ddots & \ddots &0\\
e_{2N-1}^* & \cdots & e_{N+1}^*  & e_N^*
\end{pmatrix} \,.
\]
Furthermore, it satisfies 
\begin{align}
 (\ta_j)_{j=1}^{N} &= B^{-1} \cdot(\ec_j)_{j=0}^{N-1}, \label{e:WeierReducedEvans}\\
((\nabla_P \ta_j)^*)_{j=1}^N &= (B^*)^{-1}\cdot((\nabla_P \ec_j)^*)_{j=0}^{N-1},\label{e:controlReducedEvans}
\end{align}
where $\mathrm{rank} \big(((\nabla_P \ec_j)^*)_{j=0}^{N-1}\big)=N$, so that $\mathrm{rank} \big(((\nabla_P \ta_j)^*)_{j=1}^N\big) = N$.
In particular, there is $\delta_1>0$ such that for any $(\tilde\ta_j)_{j=1}^{N}$ with $|\tilde a_j|<\delta_1$, $j=1,\ldots,N$,  there is $\cP$ with $|\cP|<\delta$ such that for $P=P_{\rm org}^{0}+\cP$ we have $(\ta_j)_{j=1}^{N}=(\tilde\ta_j)_{j=1}^{N}$. 
Specifically, this can be realised by setting $P$ equal $P_{\rm org}^0$, {\it{i.e.}} $\cP=0$, except for a suitable adjustment of $\alpha_j$, $j=1,\ldots,N$. Moreover, all statements remain valid for any multiplicity $\ell+1, 1 \leq \ell \leq N,$ of the zero root of $E_0$ ({\it i.e.} multiplicity $\ell$ for $\uE$).
\end{lemma}

\begin{proof}
Comparing coefficients of powers of $\lambda$, the product structure of $\uE$ in \eqref{e:evansWeierN} gives the relations
\begin{align*}
\ec_0 = \ta_1 \te_0,\;
\ec_1 = \ta_2 \te_0 + \ta_1 \te_1,\;
\,\,\, \ldots,\;
\ec_{N-1} &= \ta_N \te_0 + \ta_{N-1} \te_1 +  \ldots + \ta_1 \te_{N-1},
\end{align*} 
which can be written as $(\ec_j)_{j=0}^{N-1} = B\cdot (\ta_j)_{j=1}^{N}$ with a lower triangular matrix $B=B(\cP)$, whose diagonal entries all equal to $b_0$. Since $b_0\neq 0$ near $\cP=0$ this is invertible, which gives \eqref{e:WeierReducedEvans}. In particular, $(\ec_{j})_{j=0}^{k-1}=0 \Leftrightarrow (\ta_j)_{j=1}^k=0$ for $k=1,\ldots,N$, which occurs for $k=N$ at $\cP=0$ by assumption. Further comparison of coefficients gives $e_N^*=\te_0^*$ and
\begin{align*}
\ec_{j+N}^* = \te_j^*, \quad j\geq 0,
\end{align*}
so that $B^*$ can be written in terms of $(\ec_{j+N}^*)_{j=0}^{N-1}$ as claimed. 
Upon differentiating  \eqref{e:WeierReducedEvans} and using $(\ec_j^*)_{j=0}^{N-1}=0$ we obtain the leading order quantitative relation \eqref{e:controlReducedEvans}.

From \eqref{eq:taylor_evansN} we compute, for $k=1,\ldots, N$, 
\[
\partial_{\alpha_j} \ec_{k-1} = (-1)^k\frac{(2k-1)!!}{(2k)!!}\frac{\tau_j^k}{d_j},
\]
so that, as in \eqref{VDM}, 
\begin{align*}
\det\big(\nabla_{(\alpha_1,\ldots,\alpha_N)} (\ec_{k-1})_{k=1}^{N}\big) &= 
\det\left(\left((-1)^k\frac{(2k-1)!!}{(2k)!!}\frac{\tau_j^k}{d_j}\right)_{1\leq k,j \leq N}\right)\\
&=\prod_{k=1}^N (-1)^k\frac{(2k-1)!!}{(2k)!!}\prod_{j=1}^N (d_j)^{-1}
\det\left(\left(\tau_j^k\right)_{1\leq k,j \leq N}\right),
\end{align*}
which is non-zero by the assumption that $\tau_j > 0$. The full rank implies that the range of the right hand side of \eqref{e:WeierReducedEvans} as a function of $\cP$ contains a neighbourhood of the origin.
\end{proof}

\subsubsection{Control of critical eigenvalues for $\eps$ sufficiently small}
The most relevant situation for our later purposes are parameters perturbed from $P_{\rm org}^0$ in \eqref{e:Porg}, where zero is an $(N+1)$-fold root of the Evans function and $N\leq 3$. For these cases we will show below, see Lemma~\ref{lem:criticalroots_N_leq_2} and \ref{lem:criticalroots_N_eq_3}, that there are no further unstable or critical roots. However, for general $N > 3$ we cannot rule out additional non-zero critical roots. Therefore, we introduce the following {\emph{condition}}, which is thus true for $ N \leq 3$ and currently a hypothesis for $N > 3$.

\begin{condition}[Non-zero roots of the Evans function]\label{hyp:criticalroots}
For the parameter set $P_{\rm org}^0$ from \eqref{e:Porg} with \eqref{N:ZERO} any non-zero root of the Evans function $E_0$ \eqref{EV0} has negative real part.
\end{condition}

\begin{lemma}[Critical eigenvalues]\label{lem:evals}
Let $\gamma = 0$, {\it i.e.,} a stationary front solution $Z_{\rm SF}$ as constructed in Proposition~\ref{P:EXeps} exists. Furthermore, let $\mathcal{L} = \partial_Z \mathcal{G}\left(Z_{\rm SF};P, \varepsilon \right)$ be the linearisation from~\eqref{Ldef} and $E_{\varepsilon}$ the Evans function determining its point spectrum as in Lemma~\ref{L:SP}. Finally, assume that Condition~\ref{hyp:criticalroots} is fulfilled. Then the following hold. 

\begin{itemize}
\item[(i)] (Persistence in $\varepsilon$ of control of eigenvalues) 
There is $\eta>0$ such that for any $\varepsilon > 0$ sufficiently small and any collection $(\nu_j)_{j=1}^{N}\in\mathbb{C}^{N}$ with $\nu_j\in \mathcal{U}:=\{z\in\mathbb{C}:|z|<\eta\}$ that is closed under complex conjugation the following holds. There exist parameters $P=P_{\rm org}^0 + \check{P}$ such that $E_\eps(\nu_j)=0$, $j=1,\ldots,N$ including multiplicities.  Moreover, $P$ can be chosen continuous in $\eps$, with $\gamma=0$ and equal to $P_{\rm org}^0$ except for possibly adjusted $(\alpha_j)_{j=1}^N$, and $|\cP|\to0$ as $(\nu_j)_{j=1}^{N}\to0$.
\item[(ii)] (Critical spectrum) 
With $\nu_j$ as in item (i) the critical and the only possible unstable eigenvalues are contained in
\[\sigma_{\rm pt, crit} :=\{0\}\cup\{\eps^2\nu_j, j=1,\ldots,N\}=  \sigma_{\rm pt} \cap \eps^2\mathcal{U},\]
which is a, within $\eps^2\mathcal{U}$, arbitrary set of $N$ points 
(up to containing zero and being closed under complex conjugation).
\item[(iii)] (Persistence in $\varepsilon$ of degeneracy order) 
In particular, for any $1\leq \ell\leq N$, there are parameters~$P^\eps_\ell$ such that the  spectrum with non-negative real part only consists of the zero eigenvalue of $\mathcal{L}$ and it has multiplicity $\ell + 1$. Moreover, $P^\eps_\ell$ can be obtained from $P_{\rm org}^0$ upon a continuous adjustment of $\alpha_j$ 
, $j=\ell+1,\ldots, N$.
\end{itemize}
\end{lemma}

Analogous to $P^0_{\rm org}$ from \eqref{e:Porg}, we denote by $P_{\rm org}^\eps:=P^\eps_N$  the $(N+1)$-fold zero degeneracy for $\eps\geq 0$ sufficiently small.

\begin{proof}
By Lemma~\ref{L:SP} it suffices to consider eigenvalues related to the roots of the Evans function $E_\eps$. Due to the scaling of point eigenvalues versus roots of $E_\eps$ at $P=P_{\rm org}^0$ all (if any) eigenvalues that are not roots of $E_\eps$ have real part less than $-\eps^2{\eta}$ for suitable ${\eta}>0$. By Condition~\ref{hyp:criticalroots}, for all (if any) roots of $E_0$ that do not result from perturbing the $(N+1)$-fold root at the origin, ${\eta}$ can be adjusted so that these also have real parts less than $-{\eta}$. Hence, item $(ii)$ follows from $(i)$ and item $(iii)$ directly follows from $(i)$ and $(ii)$.

Concerning item $(i)$, by Lemma~\ref{l:linearunfold} the set of coefficient-vectors of the $N$-th Taylor polynomial of $E_0(\lambda)/\lambda$ that are generated by $\cP\approx 0$ contains an open neighbourhood of $0$ in $\R^{N}$. 
Since these coefficients are continuous in $\eps$ for $E_\eps(\lambda)/\lambda$, an open neighbourhood of~$0$ in $\R^{N}$ is also realised for any sufficiently small $\eps>0$. Hence, the claims also hold for the roots of $E_\eps$ (and thus the actual eigenvalues of $\calL$) for any sufficiently small $\eps>0$.
\end{proof}

Since Condition~\ref{hyp:criticalroots} is the core assumption for obtaining the spectral set-up described in Lemma~\ref{lem:evals}, see also Figure~\ref{f:spectrum_unfold}, it remains to elaborate on when we can prove that it is fulfilled. {We first infer from} previous work that it holds true for $N \in \{1, 2\}$.


\begin{figure}[!t]
\includegraphics[width = 0.4\textwidth]{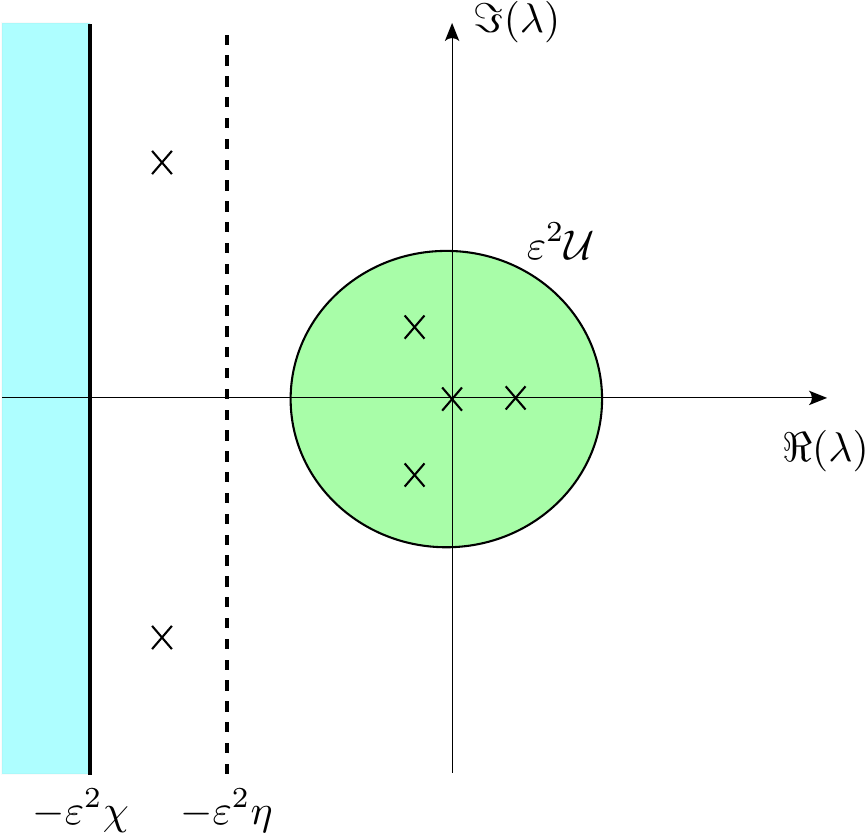}
\caption{Illustration of the spectrum of $\mathcal{L}$ from \eqref{Ldef} at $P = P^{\varepsilon}_{\mathrm{org}} + \check{P}$ for~$\eps,|\cP|$ sufficiently small (cf.\ Lemma~\ref{lem:evals}). 
The blue region bounded by $-\eps^2\chi$ contains (but is not equal) $\sigma_{\rm ess}$ (cf.\ Lemma~\ref{L:SE}). 
The $N+1$ eigenvalues (crosses) in the green region $\eps^2 \mathcal{U}$ can be fully controlled (up to one fixed zero and complex conjugation);  
any further eigenvalues lie to the left of the vertical dashed line; for $N\leq 3$, we prove $\eta\geq \chi$.
 This configuration is the basis for the center manifold reduction in Proposition~\ref{prop:reduced_system}.}
  \label{f:spectrum_unfold}
\end{figure}


\begin{lemma}[Number of eigenvalues for $N \in \{1,2\}$]\label{lem:criticalroots_N_leq_2} Let $\eps\geq 0$ be sufficiently small and assume $N \in \{1,2\}$. Then the Evans function $E_{\varepsilon}$ from Lemma~\ref{L:SP} possesses at most $N+1$ roots (counting multiplicity) 
and there are parameter settings realising any number $1, \ldots, N+1$, of roots. 
\end{lemma}


\begin{proof}
See \cite[Lemma 7]{CBDvHR15} for $c = 0$, and note that the addition of nonlinear coupling terms in $\mathcal{F}^N(\vec{V})$ does not change the Evans function in the case $c = 0$.
\end{proof}

Although we expect that the same holds for general $N \in \mathbb{N}$, we refrain from elaborating on this further and {focus on $N = 3$, $P\approx P_{\rm org}^0$}, which we need in order to prove the occurrence of chaotic motion for dynamic front solutions, see Main Result~\ref{thm:intro}(b). 

\begin{lemma}[Number of critical Evans function roots for $N = 3$]\label{lem:criticalroots_N_eq_3} 
Let $N =3 $ and let $P=P_{\rm org}^0$ from \eqref{e:Porg} such that there is a fourfold zero root of the Evans function $E_0$ from~\eqref{EV0} of a stationary front solution $Z_{\rm SF}$. Then for $P = P_{\rm org}^0 + \check{P}$ and any $\eps, |\cP|$ sufficiently small, the Evans function $E_\eps$ has precisely four roots (counting multiplicity). In particular, Condition~\ref{hyp:criticalroots} holds for $N=3$ and by varying  $\cP$ three of the critical eigenvalues can be placed at arbitrary locations in a neighbourhood of zero (up to respecting complex conjugation symmetry).
\end{lemma}

The proof of the Lemma is given in Appendix~\ref{app:numberroots} and is based on connecting the cases $N = 2$ and $N = 3$ by a homotopy argument. In particular, starting with a parameter set $P=P_{\rm org}^0$ for $N = 3$, where $\tau_2\neq \tau_3$, we show that during a homotopy to $\tau_2 = \tau_3$ (which is equivalent to $N = 2$) the number of roots always decreases by exactly one. 


\subsubsection{(Generalised) eigenfunctions}\label{s:gen_efunc}
Having established multiplicity $N+1$ for the zero root of the Evans function, it remains to study the Jordan-chain structure in order to be able to analyse the resulting reduced dynamics.
The case of the three-component model with affine coupling function $\mathcal{F}^N(\vec{V})$, {\it i.e.}, \eqref{1F2S} (which is \eqref{eq:multi-component-RD} with $N=2$ and \eqref{N:NONL22} with $\mathcal{F}_{\rm nl}^N(\vec{V}) \equiv 0$), was studied in \cite[\S2]{CBvHHR19}. There the SLEP method \cite{nishiura1987stability, nishiura1990singular} was used to prove that the triple zero eigenvalue (from Proposition~\ref{L:N+1}/Lemma~\ref{lem:evals}) results in a Jordan block structure of length three, {\it{i.e.}}, a maximal Jordan chain. Here, we prove that this generalises to the $(N+1)$-component case~\eqref{eq:multi-component-RD} with general nonlinear coupling function $\mathcal{F}^N(\vec{V})$~\eqref{N:NONL22}. In preparation, we first consider eigenfunctions $\Phi$ of an eigenvalue $\eps^2\lambda$, that is,
\begin{align*}
\mathcal{L}\Phi(\lambda) = \varepsilon^2 \lambda \Phi(\lambda) \,,
\end{align*}
where $\calL$ is from \eqref{Ldef} for a stationary front solution. For this we need to introduce the following auxiliary operators. Since we consider exponentially localised eigenfunctions only, we do not specify range and domain; a canonical choice would be range $L^2(\R)$ and domain $H^2(\R)$ for each component.
\begin{deff}[Operators]
\label{LL}
 Let $U_{\rm SF}$ denote the $U$-component of a stationary front solution $Z_{\rm SF}.
$\begin{align*}
&L^c:= \partial_{z}^2  + \ 1 - 3\tanh^2{(z/\sqrt2)} + \eps c \partial_z\,, \quad
L^0:=L^c |_{c=0}\,, \\
&L_\eps := \eps^2 \partial_x^2+1-3 U_{\rm SF}^2.
\end{align*}
\end{deff}

\begin{lemma}[Eigenfunctions $\Phi(\lambda)$ for $c = 0$] \label{L:EIGSF}
For $\eps>0$ sufficiently small let $\lambda$ be a root of the Evans function $E_\eps$ from Lemma~\ref{L:SP}, whose location is to leading order given by \eqref{EV0}. 
Then the corresponding eigenfunction $\Phi(\lambda)$ of the linearisation in the associated stationary front solution~$Z_{\rm SF}$ is to leading order (and up to an arbitrary scaling factor) given by
\begin{align}\label{N:UPROFILE_EIG}
\varphi_U(z) = \left\{
\begin{aligned}
\dfrac{\sqrt2}{2 \eps}\sech^2{(z/\sqrt2)} &\,,\, z \in I_f, \\ 
0 &\,,\, z \in I_s^\pm,
\end{aligned}
\right. \quad
\varphi_j(y) = \left\{
\begin{aligned}
&\dfrac{1}{h_j(\lambda)} & ,\; y \in I_f, \\ 
&\dfrac{1}{h_j(\lambda)} e^{\mp h_j(\lambda) y/d_j^2} & ,\; y \in I_s^\pm,
\end{aligned}
\right.
\end{align}
with 
$
h_j(\lambda)  = d_j\sqrt{\tau_j \lambda +1}\,.
$
\end{lemma}

\begin{proof}
This follows directly from adapting the proof of~\cite[Lemma 5]{CBvHHR19}\footnote{There is a mistake in the profile of the eigenfunction in~\cite[Lemma 5]{CBvHHR19}: the exponents of the slow components in the slow fields should have a $1/h$ instead of $h$.} to the current situation. The details can be found in Appendix~\ref{A:E}. 
\end{proof}
A suitably scaled $k$-th generalised eigenfunction $\Psi_k$ associated to an eigenfunction $\Psi_0 = \Phi(\lambda)$ from Lemma~\ref{L:EIGSF} satisfies
\begin{align}
\label{GEIG}
\mathcal{L} \Psi_k = \eps^2 \Psi_{k-1}.
\end{align} 
The $\eps^2$-scaling in \eqref{GEIG} is used for convenience and is related to the slow-fast scaling of the system. We focus on $\lambda=0$, the translation eigenvalue, whose eigenfunction on the $x$-scale is given by $\partial_x Z_{\rm SF}$.
Following \cite{CBvHHR19}, we introduce $\Psi_k = (\Psi_k^U, \vec{\Psi}_k^V)$, and write \eqref{GEIG} 
as
$$
\begin{pmatrix}
L_\eps & \eps A \\
\mathds{1}_{N} & \diag(d_j^2 \partial_x^2-1)
\end{pmatrix}
\begin{pmatrix}
\Psi_k^U \\[2mm] \vec{\Psi}_k^V
\end{pmatrix}=
M(\vec{\tau})
\begin{pmatrix}
\eps^2 \Psi_{k-1}^U \\[2mm]  \vec{\Psi}_{k-1}^V
\end{pmatrix}\,,
$$
with 
$L_\eps$ from Definition~\ref{LL}, and $A$ the row vector given by 
$A := (-\nabla \mathcal{F}^N(\vec{V}_{\rm SF}))^t$. 
The latter differs from the one considered in \cite{CBvHHR19} in terms of the nonlinearity $\mathcal{F}_{\rm nl}^N(\vec{V})$ of the coupling term $\mathcal{F}^N(\vec{V})$ \eqref{N:NONL22} and its length. In particular, in \cite{CBvHHR19} this row vector had constant coefficients, which {is not the case here in general}. However, $A$ is still a bounded operator and the proof of~\cite[Proposition~2]{CBvHHR19} generalises so that we obtain a `SLEP equation' similar to \cite[(22)]{CBvHHR19}. As in that case, the SLEP equation can be seen as a solvability condition {for the existence of a generalised eigenfunction}, where an operator involving $A$ is tested with $\partial_x U_{\rm SF}$, the $U$-component of $\partial_x Z_{\rm SF}$. {Since $U_{\rm SF}$ is to leading order a step function,} to leading order $A$ is just evaluated at $x=0$. From $\mathcal{F}^N(\vec{V})$ \eqref{N:NONL22} and since $\vec{V}^*=0$ \eqref{VSTARN} for stationary front solutions, we have that $A|_{x=0}$ is simply the constant row vector $(-\alpha_j)_{j=1}^N$ as in \cite{CBvHHR19} but now with $N$ components. That is, the SLEP equation becomes independent of the nonlinearity of $\mathcal{F}^N(\vec{V})$. This is as expected since also the Evans function \eqref{EV0} is not directly depending on the nonlinearity of $\mathcal{F}^N(\vec{V})$. Consequently, the SLEP results of \cite{CBvHHR19} generalise as follows ({ where we omit some details of the proof, which proceeds as that of  \cite[Proposition~2]{CBvHHR19} up to the mentioned modifications)}. 

\begin{lemma}
[Maximality of Jordan chain at $\lambda=0$]
\label{L:N+1_J}
Let $\eps$ be sufficiently small, 
$\ga=0$ and let the remaining system parameters be such that  the zero eigenvalue of $\calL$ from \eqref{Ldef} associated to $Z_{\rm SF}$ has multiplicity $\ell\in\{1,\ldots,N+1\}$, e.g., with $\alpha_j$, $j=1,\ldots,\ell -1$ as in Lemma~\ref{lem:evals}(iii). Then $\calL$ possesses an associated Jordan chain of length $\ell-1$.
\end{lemma}

From the leading order profile of the eigenfunction $\Psi_0= \Phi|_{\lambda=0}$, see~\eqref{N:UPROFILE_EIG}, together with the equation for the $k$-th generalised eigenfunction \eqref{GEIG}, we can now
iteratively compute the leading order profiles of these generalised eigenfunctions $\Psi_k$. In \cite{CBvHHR19} this is done for the first and second generalised eigenfunctions for the 
three-component system~\eqref{1F2S}. This derivation instantly generalises to the $N$-component case with nonlinear coupling term $\mathcal{F}^N(\vec{V})$~\eqref{N:NONL22} and is similar to the proof of Lemma~\ref{L:EIGSF} (related to the eigenfunction) given in Appendix~\ref{A:E}. 
From this it becomes clear that the crux in obtaining a leading order formula for the $k$-th generalised eigenfunction $\Psi_k, k=1, \ldots, N$, is the shape of the slow components.
In particular, mimicking the proof from~\cite{CBvHHR19} we observe that these shapes are determined by the leading order equations for the slow components in the slow fields. For the $k$-th generalised eigenfunction $\Psi_k$ these become
\begin{align} \label{GSlow}
d_j^2 (\Psi_{k,j})_{yy} =   \Psi_{k,j} + \tau_j \Psi_{k-1,j}\,, \qquad j=1,2,\ldots,N\,,
\end{align}
where $\Psi_{k,j}$, respectively $\Psi_{k-1,j}$, are understood to be the leading order $V_j$-component of the $k$-th, respectively $(k-1)$-th, generalised eigenfunction $\Psi_{k}$, respectively $\Psi_{k-1}$. In other words, we can relate these slow profiles of the generalised eigenfunctions to the solutions of a set of nonhomogeneous ODEs. From these
expressions on the slow scale we also obtain a solvability condition at the next $\O(\eps)$-level in the fast field, and this condition coincides  
with equating the $\mathcal{O}(\lambda^k)$-term of the Taylor expansion of the Evans function \eqref{EV0} to zero, see \eqref{TAYLOR}. Thus, consistent with Lemma~\ref{L:N+1_J}, this computational approach leads to a $k$-th generalised eigenfunction if $\lambda=0$ is a root of at least order $k$.

\begin{lemma}[ODE structure of the Jordan chain]
\label{L:EIGv}
Consider the coupled set of ODE\footnote{By rescaling space we can nondimensionalise \eqref{GSlow} to be independent from $d_j$.}
\begin{align}
\label{SET}
\left\{
\begin{aligned}
(v_{-}^{j})''(x) = v_-^j(x) + \tau v_-^{j-1}(x)\,, & \quad x \in (-\infty,0]\,, \\
(v_{+}^{j})''(x) = v_+^j(x) + \tau v_+^{j-1}(x)\,, & \quad x \in [0,\infty)\,,
\end{aligned}
\right.
\end{align}
with $j=1, \ldots, k$, 
together with the boundary and matching conditions
\begin{align} \label{BC}
v_{\pm}^{j}(\pm \infty) = 0\,, \qquad v_{-}^{j}(0) = v_{+}^{j}(0)\,, \qquad (v_{-}^{j})'(0) = (v_{+}^{j})'(0)\,.
\end{align}
If
$v^0_{\pm}(x) = e^{\mp x}/d$,\footnote{{\it i.e.} the leading order slow profile of the eigenfunction, see Lemma~\ref{L:EIGSF} (and Lemma 5 of \cite{CBvHHR19}) with $\lambda=0$.} then it holds true that 
the solutions to \eqref{SET} with \eqref{BC} are given by
\begin{align}
v^j_{+}(-x) &= v^j_{-}(x)\,, & j=0,1,  \ldots, k,\label{PAR}\\
v^{j}_{+}(x) &= (-1)^{j}\dfrac{(2j-1)!!}{(2j)!!} \tau^{j}v^0_{+}(x)  \sum_{i=0}^{j} \dfrac{2^i}{i!}  \dfrac{\tbinom{2j-i}{j}}{  \tbinom{2j}{j}}   x^i\,, & j=1,2,\ldots,k.\label{EQ1n}
\end{align}
\end{lemma}

\begin{proof} See Appendix~\ref{A:L}. 
\end{proof}
The first few terms $v^j_{+}(x)$ generated from \eqref{EQ1n} are 
\begin{align*}
\begin{aligned}
v_+^1 (x) &= - \frac{\tau }{2d}  (1+x) e^{-x}= v_-^1 (-x)\,,\\
v_+^2 (x) &= \frac{3 \tau^2}{8d}  \left(1+x+\dfrac{x^2}{3}\right)e^{-x} = v_-^2 (-x)\,,\\
v_+^3 (x) &= -\frac{5 \tau^3}{16d}  \left(1+x+\dfrac{2x^2}{5} + \dfrac{x^3}{15}\right)e^{-x} = v_-^3 (-x)\,,\\
v_+^4 (x) &= \frac{35 \tau^4}{128d}  \left(1+x+\dfrac{3x^2}{7}+\dfrac{2x^3}{21}+\dfrac{x^4}{105} \right)e^{-x} = v_-^4 (-x)\,.
\end{aligned}
\end{align*}
Indeed, the first two terms $v_+^{1,2}(x)$ coincide with the leading order expressions for the slow components of the first and second generalised eigenfunction in the slow field $I_s^+$ as derived in Lemma 6 and 7 of \cite{CBvHHR19}.
\begin{lemma}[Profiles of generalised eigenfunctions]
\label{L:GEIGSG}
Let $\ga=0$ and let the remaining system parameters be such that the zero eigenvalue of $\mathcal{L}$ is algebraically of order $\ell+1$, $ 1 \leq \ell \leq N$. For $\varepsilon > 0$ sufficiently small the $k$-th generalised eigenfunction $\Psi_k$, $k=1, \ldots, \ell,$ associated to a stationary front solution $Z_{\rm SF}=(U_{\rm SF},\vec{V}_{\rm SF})$ -- whose leading order profiles $(U^0_{\rm SF},\vec{V}^0_{\rm SF})$ can be deduced from \eqref{N:PROFILES} -- is to leading order given by $(\Psi_{k,U}, \Psi_{k,1}, \ldots, \Psi_{k,N})$ with
\begin{align}\label{N:UPROFILE_EIGl}
\begin{aligned}
\Psi_{k,U}(y)&= \left\{
\begin{aligned}
K_k  &\,,& y \in I_f, \\ 
-\dfrac12 \eps \sum_{j=1}^N \left(\alpha_j + \partial_j \mathcal{F}^N_{\rm{nl}}(\vec{V}_{\rm SF}^0(y))\right) \Psi_{k,j}(y)&\,,& y \in I_s^\pm,
\end{aligned} \right.\\
\Psi_{k,j}(y) &= \left\{
\begin{aligned}
(-1)^{k}\dfrac{(2k-1)!!}{(2k)!!} \dfrac{\tau_j^{k}}{d_j} &\,,& y \in I_f \,, \\ 
v_\pm^k\left(\dfrac{y}{d_j}\right)&\,,& y \in I_s^\pm\,,
\end{aligned}
\right.
\end{aligned}
\end{align}
with $K_1 = \eps/(3 \sqrt2),$ and $K_k=\mathcal{O}(\eps^2)$ for $k = 2, \ldots, \ell$, and
$v^k_{\pm}$ given by \eqref{PAR} and \eqref{EQ1n}.
\end{lemma}

\begin{proof}
Let $\ell \in \mathbb{N}$ be given. We use induction on $k$ $(\in \{1, \ldots, \ell\})$ to construct the leading order profiles \eqref{N:UPROFILE_EIGl} of the generalised eigenfunctions $\Psi_k$. 

The base cases\footnote{We take two base cases since $K_1= \mathcal{O}(\eps)$ and not $ \mathcal{O}(\eps^2)$.} $k=1$ and $k=2$ related to the first and second generalised eigenfunctions follow directly from generalising the proof for the three-component system~\eqref{1F2S}, see Appendix A of~\cite{CBvHHR19}. This generalisation goes in a similar fashion as in the proof of Lemma~\ref{L:EIGSF}, hence we omit the details.

Next, assume \eqref{N:UPROFILE_EIGl} holds for the $(k-1)$-th generalised eigenfunction. The $k$-th generalised eigenvalue problem is given by $\mathcal{L} \Psi_k = \eps^2 \Psi_{k-1}$ \eqref{GEIG},
which, by setting $\Psi_k= (u,\vec{v}) = (u,v_1, \ldots, v_N)$ and with slight abuse of notation, can be written as a system of ODEs. In the slow scaling it is given by 
\begin{align}\label{GEN3_SLOW}
\left\{
\begin{aligned}
\eps \partial_y u =& p\\
\eps \partial_y p =& \eps^2 \Psi_{k-1,U} -  u + 3(U_{\rm SF})^2   u + \varepsilon 
\nabla \mathcal{F}^N(\vec{V}_{\rm SF}) \cdot \vec{v} \,,\\
\partial_y v_j =&  q_j\,,\\
\partial_y q_j =&  \dfrac{1}{d_j^2} \left( v_j - u + \tau_j \Psi_{k-1,j}\right), \qquad j = 1, \ldots, N \,, 
\end{aligned}
\right.
\end{align}
and in the fast scaling by
\begin{align}\label{GEN3_FAST}
\left\{
\begin{aligned}
\partial_z u =& p\\
\partial_z p =& \eps^2 \Psi_{k-1,U} -  u + 3(U_{\rm SF})^2   u + \varepsilon 
\nabla \mathcal{F}^N(\vec{V}_{\rm SF}) \cdot \vec{v} \,,\\
\partial_z v_j =&  \eps  q_j\,,\\
\partial_z q_j =&  \dfrac{\eps }{d_j^2} \left( v_j - u + \tau_j \Psi_{k-1,j}\right), \qquad j = 1, \ldots, N \,. 
\end{aligned}
\right.
\end{align}
We proceed 
in a similar fashion as in the proof of Lemma~\ref{L:EIGSF} given in Appendix~\ref{A:E} and we are going to determine the correct asymptotic scaling of $(u, \vec{v})$ by subsequently examining~\eqref{GEN3_FAST} in the fast field $I_f$ and \eqref{GEN3_SLOW} in the slow fields $I_s^\pm$, see Definition~\ref{fields}.\\

{\bf Fast field $I_f$, $\O(1)$}: 
Upon using a regular expansion $(u,v_1, \ldots, v_N)  = (u^0,v_1^{0}, \ldots, v_N^{0})+ \eps (u^1,v^1_{1}, \ldots, v^1_{N}) + \eps^2 (u^2,v^2_{1}, \ldots, v^2_{N})+ \O(\eps^3)$ and recalling the leading order part of $U_{\rm SF}$ from \eqref{N:PROFILES}, we get to leading order in the fast field from \eqref{GEN3_FAST} 
$$
\O(1): \partial_z^2 u^0 = 3 (U_{\rm SF}^0)^2 u^0 - u^0 \implies u^0 = \bar{C} \sech^2{(z/\sqrt2)}\,.
$$
That is, $u^0$ is in the kernel of $L^0$, see Definition~\ref{LL}, and 
without loss of generality we can take $\bar{C}=0$, thus $u^0=0$. Since $\partial_z v_j^0=0$ and $\partial_z q_j^{0}=0$, for $j=1, \ldots, N,$ in $I_f$, the slow components, and their derivatives, are constant to leading order in $I_f$. \\

{\bf Slow fields $I_s^\pm$, $\O(1)$}: 
Since $u^0=0$ in the slow fields $I_s^\pm$, the slow equations \eqref{GEN3_SLOW} in these fields reduce to 
\begin{align*}
\begin{aligned}
\O(1):  d_j^2 \partial_{y}^2 v_j^{0} = v_j^0  + \tau_j \Psi_{k-1,j}
\,, \qquad j=1,2,\ldots, N\,,
\end{aligned}
\end{align*}
where we recall that $\Psi_{k-1,j}$ is the leading order $V_j$-component of the $k-1$-th generalised eigenfunction $\Psi_{k-1}$. Hence, by the inductive assumption we get
\begin{align*}
\begin{aligned}
\O(1):  d_j^2 \partial_{y}^2 v_j^{0} = v_j^0  + \tau_j v_{\pm}^{k-1} \left(\dfrac{y}{d_j}\right)
\,, \qquad y \in I_s^\pm\,,  \qquad j=1,2,\ldots, N\,, 
\end{aligned}
\end{align*}
and, by Lemma~\ref{L:EIGv} and in particular \eqref{PAR} and \eqref{EQ1n}, we get
\begin{align} \label{VV}
v_j^{0}(y) = 
\left\{
\begin{aligned}
v_-^{0}\left(\dfrac{y}{d_j}\right) &=  v_+^{0}\left(-\dfrac{y}{d_j}\right)  \,, \quad y \in I_s^-,\\
v_+^{0}\left(\dfrac{y}{d_j}\right) &= (-1)^{k}\dfrac{(2k-1)!!}{(2k)!!} \frac{\tau_j^{k}}{d_j} e^{y/d_j}  \sum_{i=0}^{k} \dfrac{2^i}{i!}  \dfrac{\tbinom{2k-i}{k}}{  \tbinom{2k}{k}}   \left(\dfrac{y}{d_j}\right)^i
\,, \quad y \in I_s^+\,.
\end{aligned}
\right.
\end{align}
This describes the leading order behaviour of the slow components in the slow fields. Note that we used that the slow components, and their derivatives, are to leading order constant over the fast field and should thus match in the fast field. \\

{\bf Fast field $I_f$, $\O(\eps)$}:
By the asymptotic scaling of $\Psi_{k-1,U}$ \eqref{N:UPROFILE_EIGl} in the fast field,
we get that the $\mathcal{O}(\eps)$-fast equation in the fast field is given by
$$
\O(\eps): \partial_z^2 u^1 = - u^1 +3 (U_{\rm SF}^0)^2 u^1  + 
\nabla \mathcal{F}^N(\vec{V}^0_{\rm SF}) \cdot \vec{v}^0\,.
$$
With Definition~\ref{LL} and \eqref{VV}, this becomes
$$
L^0 u^1 
= (-1)^{k}\dfrac{(2k-1)!!}{(2k)!!} \sum_{j=1}^N  \frac{\alpha_j \tau_j^{k} }{d_j}\,,
$$
where we used that $\partial_j \mathcal{F}^N_{\rm{nl}}(\vec{V}_{\rm SF}^0(y))$ is to leading order $0$ in the fast field.
Hence, we obtain a solvability condition, which, since $k \geq 2$, coincides with \eqref{TAYLOR}. That is, to have a $k$-th generalised eigenfunction we require that $\lambda=0$ is a root of at least order $k$.
The equation for $u^1$ now reduces to
$
L^0 u^1 
= 0
$
and, as before we can, without loss of generality, set $u^1 = 0$ in the fast field, see also \cite{CBvHHR19}. \\

{\bf Slow fields $I_s^\pm$, $\O(\eps)$}:
Finally, to determine the leading order behaviour of the fast component in the slow fields $I_s^\pm$, we look at the equation of the fast component in $I_s^\pm$:
$$
\eps^2 \partial_y^2 u =
 \eps^2 \Psi_{k-1,U} -  u + 3(U_{\rm SF})^2   u + \varepsilon 
\nabla \mathcal{F}^N(\vec{V}_{\rm SF}) \cdot \vec{v} \,.\\
$$
Using that $\Psi_{k-1,U}$ in $I_s^\pm$ is $\O(\eps)$ , $U_{\rm SF}= \pm 1 + \eps^2 U_{\rm SF}^2$, and $u=\eps u^1 + \O(\eps^2)$ we get at $\O(\eps)$
$$
u^1 = -\frac12\left( \sum_{j=1}^N \left( \al_j + \partial_j \mathcal{F}^N_{\rm nl}(\vec{V}_{\rm SF}^0)(y) \right)v_j^{0}(y) \right)
\,,
$$
see \eqref{N:UPROFILE_EIGl}.

By the $\O(\lambda^k)$-equality of \eqref{TAYLOR}, the $\O(\eps)$-expressions for the fast $U$-component of the generalised eigenfunction $\Psi_{k,U}, k=2,3,\ldots, \ell $ in the slow fields as given in \eqref{N:UPROFILE_EIGl} indeed approach (to leading order) zero as $x$ approaches the boundary of the slow fields.   
In particular, to leading order we have
\begin{align*}
\lim_{x \to \pm \eps}\Psi_{k,U} (x) &= 
-\dfrac12 \eps  \sum_{j=1}^N \left( \al_j + \partial_j \mathcal{F}^N_{\rm nl}(\vec{V}_{\rm SF}^0)(0_\pm) \right)\Psi_{k,j}(0_\pm) ) \\&=   (-1)^{k+1}  \eps  \dfrac{(2k-1)!!}{2(2k)!!} 
 \sum_{j=1}^N 
 \dfrac{\alpha_j  \tau_j^{k}}{d_j} =0\,.
\end{align*}
Similarly, for the fast $U$-component of the first generalised eigenfunction $\Psi_{1,U}$ as given in~\eqref{N:UPROFILE_EIGl} we indeed  to leading order have, by the leading order equality of \eqref{TAYLOR}, that
$$
\lim_{x \to \pm \eps} \Psi_{1,U} (x)=
-\dfrac12 \eps  \sum_{j=1}^N \left( \al_j + \partial_j \mathcal{F}^N_{\rm nl}(\vec{V}_{\rm SF}^0)(0_\pm) \right)\Psi_{1,j}(0_\pm) ) 
=  \dfrac14 \eps \sum_{j=1}^N  \dfrac{\alpha_j \tau_j}{d_j}
= \dfrac{\eps}{3 \sqrt2}\,.
$$
This concludes the proof. 
\end{proof}
\begin{remark} 
It can be shown (similar as in the proof given in Appendix~\ref{A:L}) that the polynomials
\begin{align*}
f^j(x) :=\sum_{i=0}^{j} a_j^i   x^i := \sum_{i=0}^{j} \dfrac{2^i}{i!}  \dfrac{\tbinom{2j-i}{j}}{  \tbinom{2j}{j}}   x^i
\end{align*} 
in Lemma~\ref{L:EIGv}, generated from \eqref{a1} and~\eqref{ainew}, are also generated
from $g^0(y)=1, g^1(y)=1+y$ and the recurrence relation 
\begin{align*}
g^{j+1}(y) = g^{j}(y) +  \dfrac{1}{(2j-1)(2j+1)} y^2 g^{j-1}(y)\,, \qquad j=1,2,\ldots\,,
\end{align*}
studied in \cite{WIT}.
\end{remark}

\subsection{Simultaneous degeneracies for existence and zero eigenvalue}
\label{S34}

Next we show that there are parameter combinations such that \eqref{EX:MAXa} of Proposition~\ref{L:EX} and \eqref{N:ZERO} of Proposition~\ref{L:N+1} hold at the same time. That is, for $\mathcal{F}^N(\vec{V})= \mathcal{F}^N_\beta(\vec{V})$ \eqref{NONLnew} there exist parameter combinations such that the existence function $\Gamma_0(c)$ \eqref{E:N} is highly degenerate while simultaneously the Evans function $E_0$ has an $(N+1)$-fold root (which will later on serve as organising center). In particular, \eqref{EX:MAXa} and \eqref{N:ZERO} are both simultaneously true if $\tau_j$ and $d_j$ are chosen such that $\tau_j \neq \tau_k \neq 0$, $\tau_j/d_j \neq \tau_k/d_k >0$, and
\begin{align}
\label{COND1}
\tau_j d_k^2 = \tau_k d_j^2\,, \qquad 1 \leq j < k \leq N\,,
\end{align}
while the remaining parameters obey \eqref{EX:MAXa}, \eqref{EX:MAXb} and \eqref{N:ZERO}.
Condition \eqref{COND1} immediately follows from equating the expressions for $\alpha_j$ from \eqref{EX:MAXa} and \eqref{N:ZERO}
\begin{align*}
 \dfrac{2 \sqrt 2 d_j}{3 \tau_j}  \prod_{k=1,k \neq j}^N \dfrac{\tau_k}{\tau_k-\tau_j}  &=  \dfrac{2 \sqrt 2 d_j}{3 \tau_j}  \prod_{k =1, k \neq j}^N \dfrac{(\tau_k/d_k)^2}{(\tau_k/d_k)^2-(\tau_j/d_j)^2} \implies \\ &
  \prod_{k=1,k \neq j}^N \dfrac{1}{1-\tau_j/\tau_k}  = \prod_{k =1, k \neq j}^N \dfrac{1}{1-(\tau_j d_k)^2/(\tau_k d_j)^2} 
\end{align*}
and this equation is solved by \eqref{COND1}.
Note that $\tau_i \neq \tau_j \neq 0$ together with \eqref{COND1} implies that we necessarily have $d_j \neq d_k \neq 0$.
Furthermore, observe that \eqref{COND1} is a dependent set of linear equations in $\tau_j$
$$
\begin{pmatrix}
d_2^2 & -d_1^2 & 0 &\cdots&\cdots  &0 \\
d_3^3& 0 & -d_1^2 &\ddots & &\vdots \\
\vdots & \vdots & \ddots & \ddots & \ddots & \vdots\\
\vdots & \vdots &  & \ddots & \ddots & 0\\
d_N^2 & 0 & \cdots & \cdots & 0 &-d_1^2 \\
0 & d_3^2 & -d_2^2 & 0 &\cdots  &0 \\
&& \vdots &&\\
0&0&0&0& -d_N^2 & -d_{N-1}^2
\end{pmatrix}
\begin{pmatrix}
\tau_1 \\ \tau_2 \\ \vdots \\ \tau_N
\end{pmatrix}
=0\,,
$$
with rank $N-1$.\footnote{We have $\sum_{i=1}^N d_i^2 C_i=0$, where $C_i$ is the $i$-th column of the matrix.} Hence, one can {\it{freely pick}} one $\tau_j$. So, given a set of $d_j, j=1, \ldots, N,$ and~$\tau_1$, \eqref{COND1} fixes $\tau_j, j=2, \ldots, N$ and, in turn, \eqref{EX:MAXa}/\eqref{N:ZERO} fixes $\alpha_j, j=1, \ldots, N$. In particular, for a given set $d_j, j=1, \ldots, N$ and $\tau_1$, we get
\begin{equation} \label{OPTIMAL}
\tau_j = \tau_1 \dfrac{d_j^2}{d_1^2}\,, \qquad j=2,\ldots, N \,, \quad 
\quad
\alpha_j
= 
\frac{2\sqrt2}{3}\dfrac{d_1^2}{\tau_1 d_j} \prod_{k=1, k \neq j=1}^N \dfrac{d_k^2}{d_k^2-d_j^2} \,, \quad j=1, \ldots, N.
\end{equation}
We emphasize that the above construction is, {\it{a priori}}, only valid in the singular limit $\eps = 0$ and not for $0<\eps\ll1$. For both degeneracies \eqref{EX:MAXa} and~\eqref{N:ZERO} the coefficients $\alpha_j$ are used separately to create an open set of expansions and thus a sufficient condition that each degeneracy can be created for $\eps>0$, respectively. Hence, using this approach in the present situation for $\eps>0$ we would need to show that the parameters also generate an open set for the combined coefficient vector. We do not  pursue this further in this paper.

By combining the results of the previous two sections, in particular Proposition~\ref{L:EX} and Proposition~\ref{L:N+1}, it is now straightforward to find parameter combinations 
for which the existence condition $\Gamma_0$ has a particular singularity structure near $c=0$, while 
the zero root
of the Evans function $E_0$ of the related stationary front solution still has maximum degeneracy. For instance,
if we set $N=3$, $\mathcal{F}^N(\vec{V}) = \mathcal{F}^N_\beta(\vec{V})$ \eqref{NONLnew} and
\begin{equation}\label{e:transpar22}
\begin{aligned}
&\alpha_1=\dfrac{578 \sqrt2}{315}, \; \alpha_2=-\dfrac{289}{90\sqrt{2}},\; \alpha_3=\dfrac{3125}{2142\sqrt{2}},\;
\beta_1=1,\; \beta_2=\beta_3=\gamma=0,\\
&\tau_1=1,\; \tau_2=\left(\frac{3}{2}\right)^2,\; \tau_3=\left(\frac{17}{10}\right)^2,\;
d_1=1,\; d_2=\frac32,\; d_3=\frac{17}{10}\,,
\end{aligned}
\end{equation}
such that $\Gamma_0(c)$~\eqref{EQ:EX_AFF} is $\mathcal{O}(c^2)$
and $E_0(\lambda)$~\eqref{eq:taylor_evansN} is $\mathcal{O}(\lambda^{4})$,
then we have an analytically predicted transcritical bifurcation point $c=0$, see the left panel of Figure~\ref{F:TC}.

In contrast, if we set $N=3$, $\mathcal{F}^N = \mathcal{F}^N_\beta$ \eqref{NONLnew} and
\begin{equation}\label{e:pitchpar}
\begin{aligned}
&\alpha_1=\dfrac{578 \sqrt2}{315}, \; \alpha_2=-\dfrac{2023}{675\sqrt{2}},\; \alpha_3=\dfrac{3125}{2142\sqrt{2}},\;
\beta_1=1,\; \beta_2=-\dfrac{784}{2025},\; \beta_3=\gamma=0,\\
&\tau_1=1,\; \tau_2=\left(\frac{3}{2}\right)^2,\; \tau_3=\left(\frac{17}{10}\right)^2,\;
d_1=1,\; d_2=\frac75,\; d_3=\frac{17}{10}\,,
\end{aligned}
\end{equation}
such that $\Gamma_0(c)$~\eqref{EQ:EX_AFF} is $\mathcal{O}(c^3)$
and $E_0(\lambda)$~\eqref{eq:taylor_evansN} is $\mathcal{O}(\lambda^{4})$,
then we have an analytically predicted pitchfork bifurcation point $c=0$, see the right panel of Figure~\ref{F:TC}.
\begin{figure}
\begin{center}
\includegraphics[width=0.97\textwidth]{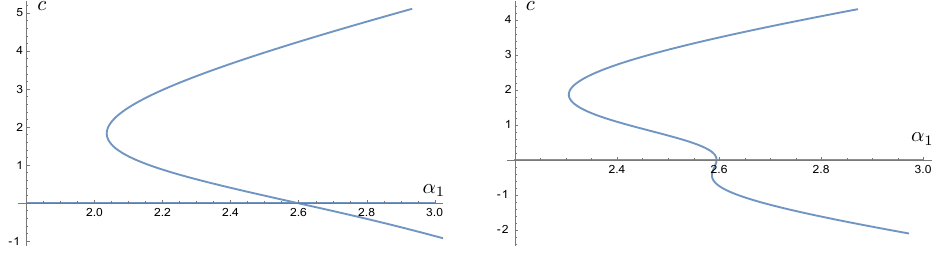}
\caption{Solutions of $\Gamma_0(c)$=0 \eqref{E:N}, with $\mathcal{F}^N(\vec{V}) = \mathcal{F}^3_\beta(\vec{V})$ \eqref{NONLnew} and for varying~$\alpha_1$. The remaining system parameters are as given in \eqref{e:transpar22}, respectively,~\eqref{e:pitchpar}. In the left panel we observe a bifurcation at $c=0$ that will turn out to be transcritical (see Figure~\ref{f:trans_branches}), while in the right panel we have a pitchfork bifurcation at $c=0$.}
\label{F:TC}
   \end{center}
   \end{figure}

From Proposition~\ref{L:EX} we know that we can keep on increasing the singularity structure of the existence function $\Gamma_0(c)$~\eqref{EQ:EX_AFF} up to $\mathcal{O}(c^7)$ (for $N=3$), while ensuring that the zero root of the Evans function $E_0$ is fourfold. In particular, we take $\mathcal{F}^3_\beta(\vec{V})$ to be affine and use~\eqref{OPTIMAL}
to obtain the required parameter combination
\begin{equation}
 \label{OPT}
 \begin{aligned}
&\alpha_1=\dfrac{578 \sqrt2}{315}, \; \alpha_2=-\dfrac{289}{90\sqrt{2}},\; \alpha_3=\dfrac{3125}{2142\sqrt{2}},\;
\beta_1=\beta_2=\beta_3=\gamma=0,\\
&\tau_1=1,\; \tau_2=\left(\frac{3}{2}\right)^2,\; \tau_3=\left(\frac{17}{10}\right)^2\,,\;
d_1=1,\; d_2=\frac32,\; d_3=\frac{17}{10}.
\end{aligned}
\end{equation}

\begin{remark}
    The above parameter sets \eqref{e:transpar22}--\eqref{OPT} display a lot of similarities. This is by construction. The maximally degenerate parameter set \eqref{OPT} is determined by picking $d_1, d_2, d_3$ and $\tau_1$ freely, $\mathcal{F}^3_\beta(\vec{V})$ affine, and the remaining parameters as in \eqref{OPTIMAL}. The other degeneracies can be obtained by tuning the parameters from this optimal parameter set. For instance, changing only $\beta_1$ from \eqref{OPT}, as in \eqref{e:transpar22}, still results in a 
    fourfold zero root of the Evans function $E_0$ (as the Evans function at $c=0$ is independent of $\beta_1$, see \eqref{eq:taylor_evansN}), while it results in an existence condition $\Gamma_0(c)=\mathcal{O}(c^2)$, see~\eqref{EQ:EX_AFF}. To subsequently increase the singularity structure to $\Gamma_0(c)=\mathcal{O}(c^3)$, $\alpha_2,\beta_2$, and $d_2$ are updated to ensure the $\mathcal{O}(c^2)$-term of $\Gamma_0$ vanishes, while keeping the order four of the zero root of $E_0$. 
\end{remark}


\section{Dynamic front solutions}\label{S:DFS}
This section contains several core novelties of the present work. We start by performing center manifold reduction to finally derive the reduced system of ODEs \eqref{FR}- \eqref{eq:main_ODE} for dynamic front solutions. This allows to prove the statements of Main Result~\ref{thm:intro} and the extensions thereof from Remark~\ref{r:impactN}.
\subsection{Center manifold reduction}\label{S:CMR}

Recall from Lemma~\ref{lem:evals} that there are parameters $P^\eps_{\Np}$ such that \eqref{eq:multi-component-RD} possesses a stationary front solution $ Z_{\rm SF} $ with an $(\Np+1)$-fold zero eigenvalue of the linearisation $\calL$ from \eqref{Ldef}. Let $\vec{\Psi}(x)$ denote the matrix with columns $\Psi_j(x)$, $j=1,\ldots,\Np$ of generalised eigenfunctions. Applying the ansatz for center manifold reduction with continuous symmetry, \cite[e.g]{haragus2010local}, to the translation symmetry and generalised kernel gives, with $\eta(t) := x-a(t)$ and $\vec{b}(t)\in\R^\Np$, that 
\begin{align}\label{eq:Zsplit1}
 Z(x,t) = Z_{\rm SF}(\eta(t)) + \vec{\Psi}(\eta(t))\cdot\vec{b}(t) + R(\eta(t), t) \, ,
\end{align}
where $R(\eta(t),t)$ is orthogonal to the generalised kernel of the adjoint of $\calL$. 
It turns out that this uniquely defines $a(t)$, which gives a somewhat implicit but standard notion of the position~$a(t)$,~\cite[e.g.]{MZ09}. This is relative to the stationary front solution $Z_{\rm SF}$ and is not as directly geometric as, {\it e.g.}, the zero of the fast component. 
As shown in \cite[e.g.]{haragus2010local}, using $\eta(t)$ allows to infer the skew-product structure of the reduced ODE \eqref{FR},\eqref{eq:main_ODE} in which $a(t)$ only appears in the equation \eqref{FR} for the symmetry group, and does not enter into the complementary part \eqref{FR}. We note that when expanding $Z_{\rm SF}(x-a(t)) = Z_{\rm SF}(x) - a(t) Z_{\rm SF}'(x) + \calO(a(t)^2)$, the term of order $a(t)$ lies in the kernel, but using $Z(x,t) = Z_{\rm SF}(x) - a(t) Z_{\rm SF}'(x) + \vec{\Psi}(\eta(t))\cdot\vec{b}(t) + R(x, t)$ for the center manifold will not directly reveal the skew-symmetry structure of the reduced ODE. 

\begin{proposition}[Center manifold reduction for nilpotent linear parts]\label{prop:reduced_system}
Consider parameters $P = P^\eps_{\Np} + \check{P}$ from Lemma~\ref{lem:evals} and let $|\cP|$, as well as $\varepsilon>0$, be sufficiently small. Then in a `tubular' neighborhood of the spatial translates of $Z_{\rm SF}$ in $(L_2(\R))^3$,  \eqref{eq:multi-component-RD} possesses an exponentially attracting $(\Np+1)$-dimensional center manifold, wherein all solutions are dynamic front solutions $Z = (U, \vec{V})$ with representation 
 \eqref{eq:Zsplit1}.

The dynamics of the front position $ a = a(t) $ is governed by
\begin{align*}
 \dfrac{d}{dt} a & = \varepsilon^2 c_1 \, ,
\end{align*}
and the dynamics of the front speed $ c_1 = c_1(t) $
by an $\Np$-dimensional system of ODEs 
 \begin{align}\label{SPEED:ODE_compact}
  \dfrac{d}{dt} \vec{c} = \varepsilon^2 \GCM(\vec{c}; \check{P},\eps)\,,
 \end{align}
for $\vec{c}=(c_1,\ldots,c_{\Np})$.  
It is of the $\Np$-th order form
\begin{align}\label{SPEED:ODE}
\left\{
\begin{array}{lcl}
  \dfrac{d}{dt} c_{k} & = & \varepsilon^2 c_{k+1} \,,  \qquad  k = 1, \ldots, \Np-1 \, ,\\[.25cm]
  \dfrac{d}{dt} c_{\Np}   & = & \eps^2 \Geps(\vec{c}) =\varepsilon^2  \Gred(\vec{c}; \check{P})  + \Gerr(\vec{c}; \check{P})\,,
\end{array}\right.
\end{align}
with 
$\Gerr(\vec{c}; \check{P}) =  \Gnot(\vec{c};\cP) + \calO(\eps^2(|\vec{c}|^3+|\vec{c}|^2 |\cP|)) + o(\eps^2)$
where $\Gnot(\vec{c};\cP)=\calO(\sum_{j= 2}^\Np|c_j|^2)$, and the decisive term
\begin{align*}
\Gred(\vec{c}; \check{P}) = a_0(\check{P}) + \sum_{j = 1}^{\Np} a_j(\check{P}) c_j + 
 c_1\sum_{j=1}^\Np a_{1j} c_j \, , 
\end{align*}
where $a_{1j}\in\R$, $a_j$ are real valued and $a_j(0)=0$ for $j=0,1,\ldots,\Np$.
Moreover, in terms of \eqref{eq:Zsplit1} we have that $\vec{b}$ equals $\vec{c}$ up to a near-identity coordinate change and $R = \calO(|\vec{c}|^2 + |\cP|)$.
\end{proposition}

\begin{proof}
See Appendix~\ref{subsec:proof_CM}.
\end{proof}

\begin{remark}\label{r:inaccessible}
We note that $\Gnot$ in $\Gerr$ from \eqref{SPEED:ODE} is the only term that is possibly not of order $\eps^2$. 
However, we still write $\eps^2\Geps$ for later convenience (e.g. $\Gnot=0$ for $\Np=1$), although it may contain terms of order $1$. While it appears natural that the dynamics on the center manifold is of order $\eps^2$, and this in particular holds for $\Gred$, for our Main Result \ref{thm:intro} it is not necessary to verify it for $\Gnot$. 
\end{remark}

\begin{remark}
While the parameter choice $P=P^{\varepsilon}_{\mathrm{org}}$ for creating the organising center is unique for $\eps=0$, cf.\ Proposition~\ref{L:N+1}, we expect this is not the case for $\eps>0$. For $P^{\varepsilon}_{\mathrm{org}}$ we used the most convenient option (by Lemma~\ref{l:linearunfold}) to adjust $\alpha_j$, cf.\ Lemma~\ref{lem:evals}, but other adjustments are likely available in general. We do not pursue this further here.
\end{remark}

{{\small{
\begin{table}[!h]
\begin{tabular}{ p{3cm} p{6.7cm} p{3.4cm}  }
 \hline
  \\ [-1.5ex]
 \multicolumn{3}{c}{\underline{Multi-component reaction-diffusion equation}} \\[2mm]
 $\partial_t Z = \mathcal{G}(Z, P, \varepsilon)\, , $  
 &
\multirow{ 2}{*}{$\left\{\begin{aligned}
 \partial_t U & =   \varepsilon^2 \partial_x^2 U+ U - U^3 - \varepsilon  \mathcal{F}^N(\vec{V})\,, \\
 \tau_j \partial_t V_j & =  \varepsilon^2 d_j^2 \partial_x^2 V_j  + \varepsilon^2( U -  V_j )\, 
 \end{aligned}\right.$}
 & 
 $\mathcal{F}^N(\vec{V}) :=$ \\$Z = (U, \vec{V}),$&& $ \gamma + \sum \alpha_j V_j + \mathcal{F}^N_{\rm{nl}}(\vec{V})$\\[5mm] 
 \hline 
 \\ [-1.5ex]
\multicolumn{3}{l}{$\bullet$ \underline{Uniformly travelling front solution $Z(x,t) = Z_{\rm TF}(x-\varepsilon^2 ct)$:} existence condition
} \\[2mm]
\multicolumn{3}{c}{$0 = \Gamma_0(c;P) = \mathcal{F}^N\left(\vec{V}^*(c)\right) - \dfrac13\sqrt2 c \,,
\quad
V^*_k(c) =  
 \dfrac{c \tau_k}{\sqrt{4 d^2_k+c^2 \tau_k^2}} $} \\[4mm]
 \hline
  \\ [-1.5ex]
 \multicolumn{3}{l}{$\bullet$ \underline{Stationary front solution $Z(x,t) = Z_{\rm SF}(x)
$:} critical point spectrum} \\[2mm]
\multicolumn{3}{c}{$\sigma_{\rm pt, crit} \left( \partial_Z \mathcal{G}(Z_{\rm SF};P)\right) = \Bigg\{ \varepsilon^2 \lambda \in \mathbb{C} ~\left|~ E_0(\lambda; P) =  
\lambda + \dfrac{3\sqrt2}{2}\sum_{k=1}^N 
\dfrac{\alpha_k}{d_k} \left(\dfrac{1}{\sqrt{\lambda \tau_k+1}}-1 \right) = 0 \right\}$
}\\[5mm]
 \hline
  \\ [-1.5ex]\multicolumn{3}{l}{$\bullet$ \underline{Dynamic front solution $Z(x,t) \approx  Z_{\rm SF}(x-\varepsilon^2 a(t))$}: ODE on center manifold at $P = P_{\Np}^\varepsilon + \check{P}$} \\[2mm]
\multicolumn{3}{c}{$\dfrac{d}{dt}a(t) = c_1(t), \quad  \dfrac{d}{dt}\vec{c}(t) = \GCM(\vec{c},\cP,\eps) \;\Leftrightarrow\;
\left\{
\begin{array}{lcl}
  \dfrac{d}{dt} c_{k} =  \varepsilon^2 c_{k+1} \,,  \qquad  k = 1, \ldots, \Np-1,\\[.4cm]
  \dfrac{d}{dt} c_{\Np}  =  \eps^2\Geps(\vec{c}; \check{P}).
\end{array} \right. (\textnormal{Speed ODE})
$} \\
\multicolumn{3}{l}{\hspace{5mm}Speed ODE equilibria: $\vec{c}_*=(c, 0, \ldots, 0)$}\\[2mm]
\multicolumn{3}{c}{
$G_0(\vec{c}_*;\check{P}) = 0  = H_0(c{ ;\cP})\Gamma_0(c;\check{P})$\,, 
\qquad $H_0(c{ ;\cP})\neq 0$}\, \\[2mm]
\multicolumn{3}{l}
{\hspace{5mm}Linearisation around trivial equilibrium 
} \\[2mm]
\multicolumn{3}{c}
 {$
\sigma(D\GCM(\vec{c}_*;\cP,\eps)) =  \Big\{ \varepsilon^2 \lambda \in \mathbb{C} ~\Big|~ \underbrace{(-1)^{\Np}\left( \lambda^{\Np} - \sum_{k = 0}^{\Np-1} \partial_{c_{k+1}}G_0(0, 0, \ldots, 0;\check{P}) \lambda^{k} \right)}_{
 = \widetilde{H}_0(\lambda{; \cP}) E_0(\lambda;P_{\Np}^0+ \check{P})/\lambda} = 0  \Big\}$}
\\[5mm]
 \hline\\
\end{tabular}
\caption{Overview of the results from the previous sections. All are understood to be leading order in $\varepsilon$. 
The function $\Geps(\vec{c}; \check{P})$  
is described in Proposition~\ref{prop:reduced_system} and $H_\eps(c{;\cP}), \widetilde{H}_\eps(\lambda{;\cP})$ 
in Lemma~\ref{lemma:connections}. }
 \label{T:overview}
\end{table}}
}}

\subsection{Analysis of the speed ODE}\label{subsec:recipe}
The conventional path of analysing the speed ODE~\eqref{SPEED:ODE} is to use the adjoint eigenfunctions and projections to compute the coefficients of the Taylor expanded right hand side $\Gred$. Following our strategy from \cite{CBvHHR19}, we choose to circumvent the added computational effort of this approach and instead exploit the results of the previous sections for stationary and uniformly travelling front solutions. In Table~\ref{T:overview} we have assembled an overview of these results that allow to spot an immediate connection between the right hand side $\Gred$ of the speed ODE and the existence function $\Gamma_0$ and Evans function $E_0$. Since $\Gamma_0$ and $E_0$ themselves depend in turn on the system parameters and the coupling function $\calF^N(\vec{V})$, it becomes conceivable how a tuning of these can be performed to create the desired dynamics in the speed ODE as stated in Main Result~\ref{thm:intro}. The next lemma summarises some aspects connecting $\Geps$ with $\Gamma_\eps$ (cf.\ Proposition~\ref{P:EXeps}) and $E_\eps$ also for $\eps>0$.

\begin{lemma}\label{lemma:connections}
Consider $P=P_\Np^\eps+\cP$ for $|\cP|$ and $\eps>0$ sufficiently small, so that Proposition~\ref{prop:reduced_system} applies. Then the following holds true.
\begin{itemize}
\item[(i)] Consider the linearisation $D\GCM(\vec{0};\cP,\eps)$ 
of the speed ODE in the trivial equilibrium $\vec{c}=\vec{0}_\Np$ corresponding to a stationary front solution. There is a 1-to-1 correspondence between the roots of $\underline{E}_\eps(\lambda) := E_\eps(\lambda)/\lambda$ and the non-zero critical roots of the characteristic polynomial of $D\GCM(\vec{0};\cP,\eps)$ 
given by
\begin{align*}
  \underline{E}_\eps^{\rm red}(\lambda; \check{P}) :=(-1)^{\Np}\left( \lambda^{\Np} - \sum_{k = 0}^{\Np-1} \partial_{k+1}\Geps(\vec{0}_\Np; \check{P}) \lambda^{k} \right) = 0 \, .
\end{align*}
In particular, there exists a smooth, non-vanishing function $\widetilde{H}_\eps(\lambda;\cP)$, 
for $\lambda$ in a neighbourhood of the origin, such that for sufficiently small 
$\check{P}$ we have
\begin{align*}
 \underline{E}_\eps^{\rm red}(\lambda; \check{P})  = \widetilde{H}_\eps(\lambda;\cP) \underline{E}_\eps(\lambda; \check{P}) \, .
\end{align*}
\item[(ii)] Equilibria of the speed ODE are of the form 
$\vec{c}=(c, 0, \ldots,0)$ 
and are in 1-to-1 correspondence with uniformly travelling front solutions $Z_{\rm TF}$, that is, there is a 1-to-1 correspondence of roots of $\Geps$ from \eqref{SPEED:ODE} and the existence function $\Gamma_\eps$ from Proposition~\ref{P:EXeps}. In particular, there exists a smooth, non-vanishing function $H_\eps(c{ ;\cP})$ such that for sufficiently small $c$ and $\check{P}$ we have
\begin{align*}
    \Geps(c, 0, \ldots, 0; \check{P}) = H_\eps(c{ ;\cP}) \Gamma_\eps(c;\check{P}) \,,
\end{align*}
\end{itemize}
\end{lemma}

\begin{proof}
Concerning (i), since eigenvalues are invariant under coordinate changes, the eigenvalues of the linearisation of \eqref{SPEED:ODE} in equilibrium points coincide with the roots of the Evans function for front solutions. The Weierstrass preparation theorem as in \S\ref{s:fulllinear} then proves the claim. Concerning (ii), equilibria of the speed ODE are of the form $\vec{c}_*(c)=(c,0\ldots,0)$, and the solutions of 
$\Geps(\vec{c}_*(c); \check{P})  = 0. $
Since the center manifold reduction contains all equilibria near zero, for $\check{P}, c$ sufficiently near zero, these $\vec{c}_*(c)$ are in 1-to-1 correspondence with the solutions of 
$\Gamma_\eps({c}; \check{P})  = 0.$  
Also, the multiplicities of these roots coincide since any multiplicity for critical roots of $\Gamma_\eps$ (which is at most $2N$) can be unfolded with $\cP$ (cf.\ Proposition~\ref{P:EXeps}), and simple roots are hyperbolic: This can be readily seen when comparing $E_0(0)$ from Proposition~\ref{L:S} with $\partial_c \Gamma_0$ (for $c=0$ this is also reflected in item (i)) and using that simple roots remain simple for $0<\eps\ll1$. Hence, $H_\eps({c};\cP) := \Geps(\vec{c}_*(c);\cP)/\Gamma_\eps(c;\cP)$ is smooth with $H_\eps(0;0)\neq 0$ as claimed. 
\end{proof}
\subsubsection{Scaling transformation}\label{subsec:cmr_ode_terms_general}
As usual in center manifold reduction, we cannot expect to be able to actually compute all terms (not even to leading order) in the speed ODE~\eqref{SPEED:ODE}. However, by normal form theory, in order to demonstrate the occurrence of specific dynamics, typically only a few terms are relevant. For our particular case where the linearisation is nilpotent with a full Jordan chain, we know from~\cite{BarrientosBook,IR05} that {for a relevant nonlinearity $c_1^2$} it is insightful to use the scaling transformation
\begin{align*}
 a_0 = \delta^{2\Np} \nu_0 \, , \quad a_k = \delta^{\Np-k+1} \nu_k \, ,  \quad c_k = \delta^{\Np{+}k{-}1} z_k \, , \quad k = 1, \ldots, \Np \, , 
\end{align*}
with $\vec{\nu} = (\nu_1, \ldots, \nu_{\Np}) \in \mathbb{R}^{\Np}, \|\vec{\nu}\| = 1$. {Applying this to \eqref{SPEED:ODE},} and changing to slow time $T = \varepsilon^2 {\delta} t$,  converts (the $\varepsilon$-leading order of)  \eqref{SPEED:ODE} into the simplified form
\begin{align}\label{e:zscaled}
\left\{
\begin{aligned}
  \dfrac{d}{dT} z_{k} & =   z_{k+1} \,,  \quad \qquad  k = 1, \ldots, \Np-1 \, ,\\[.2cm]
  \dfrac{d}{dT} z_{\Np}   & =  \nu_0 + 
   \vec{\nu}\cdot\vec{z}
+ {a_{11}}
   \,  z_1^2 +   
{a_{12}} 
   \, z_1 z_2 \delta 
+ \mathcal{O}(\delta^2) \,.
\end{aligned} 
\right.
\end{align}
By Lemma~\ref{l:linearunfold} we know that the linear coefficients $\nu_k$ can be fully controlled by the parameters of~\eqref{eq:multi-component-RD} with \eqref{N:NONL22}, {\emph{i.e.}}, on a linear level the unfolding  is complete. Hence,
it remains to determine $a_{11}$ and possibly $a_{12}$, in terms of the system parameters, depending on the type of dynamics of interest. This observation is used in \S\ref{S:1+3} for the case $N=\Np=3$ in the proof of Lemma~\ref{lem:chaos}  for proving chaotic dynamics following \cite{IR05}.
\subsection{Generating arbitrary singularity structures} 
\label{subsec:sing}
\label{S:sing}

By Proposition~\ref{prop:reduced_system} the speed ODE \eqref{SPEED:ODE} is scalar in the case $\Np = 1$, {\emph{i.e.}},
\begin{align}\label{e:N1CMF}
\begin{array}{lcl}
  \dfrac{d}{dt} c_1   & = & \varepsilon^2  \Geps(c_1; \check{P})   \,.
\end{array}
\end{align}
This can happen for $N = 1$, but also for $N > 1$ (which was already demonstrated in \cite{CBDvHR15} for $N = 2$). The fact that the ODE is scalar allows particularly strong control over the possible dynamics of front solutions via the coupling function $\calF^N$ of the original multi-component system \eqref{eq:multi-component-RD}. In particular, beyond Proposition~\ref{P:EXeps}, one can generate arbitrary scalar singularities in the speed ODE. This is the main result of the following theorem, as announced in Main Result~\ref{thm:intro} for $N=1$, and for $N>1$ with reducible coupling as in Remark~\ref{r:impactN}. We briefly elaborate on the latter.

If $\calF^N(\vec{V})$ is independent of $V_m$ for some $m$, then the $V_m$-component in~\eqref{eq:multi-component-RD} satisfies the linear equation $\tau_m\partial_t V_m =\eps( d_m \partial_x^2 V_m + U-V_m)$ into which the $U$-component enters as an inhomogeneity. Since $V_m$ does not appear anywhere else in \eqref{eq:multi-component-RD}, the $V_m$-component, and thus the parameters $\tau_m, d_m$, do not impact the existence of, {\it e.g.}, uniformly travelling or stationary front solutions. In addition, there is also no impact on the critical stability properties: 
Recall the linearization yields the block matrix operator $M_1$ from the proof of Lemma~\ref{L:SE}. 
Since $V_m$ only appears in the linear equation for $V_m$, the impact of $V_m$ on its spectrum is only through the linearisation of the equation for $V_m$ with respect to $V_m$. This reads $\tau_m\partial_t V_m = \eps^2(d_m\partial_{xx}V_m-V_m)$, and the spectrum of the associated linear operator is $(-\infty,-\eps^2/\tau_m]$, which is part of the strictly stable essential spectrum for $\eps>0$ and thus not critical. This observation directly extends to any number of decoupled components.

\begin{theorem}[Generating arbitrary scalar singularities for $N\geq 1$]\label{thm:one_slow_sing}
Assume the coupling function $\calF^N$ depends only on one component as in Remark~\ref{r:impactN} for $k=1$. Consider $P=P_\Np^\eps+\cP$ with $\Np = 1$ and $|\cP|$, $\varepsilon>0$ sufficiently small. 
Then for any $M \in \mathbb{N}$, any Taylor polynomial $\mathcal{T}_{\Geps}^M$  in~\eqref{e:N1CMF} can be realised by suitable choice of $\tau_1$ and~$\mathcal{T}_{\calF^N}^M$. In particular, any scalar singularity can be embedded and unfolded in the reduced system on the center manifold. In addition, if all roots of~$\Gamma_0$ from Lemma~\ref{L:N} are simple, then the speed ODE \eqref{e:N1CMF} is topologically equivalent to the scalar ODE
\[
\dot z = \eps^2 \Gamma_0(z).
\]
\end{theorem}

We actually prove that for $\tau_1/(2\sqrt{d_1}) \neq \sqrt{2}/3$ there are neighbourhoods  of zero in $\R^M$ wherein the coefficient vectors of $\calT_{\Geps}^M$ are in 1-to-1 correspondence with those of $\calT_{\calF^1}^M$; without loss of generality $V_1$ is the relevant component. In case $\tau_1/(2\sqrt{d_1})= \sqrt{2}/3$, we additionally need to invoke $\tau_1$, so that the relation is no longer injective.
\begin{proof}
By discussion of decoupled components before the theorem it suffices to consider $N=1$. For a general smooth coupling function $\calF^1=:\calF$, the existence function from \eqref{E:N} is given (to leading order in $\varepsilon$) by
\begin{align*}
    \Gamma_0(c) = \calF(v(c)) - \frac {\sqrt{2}} 3  c \,, \qquad v(c):=\vb=\frac{c \tau_1}{\sqrt{c^2 \tau_1^2 +4 d_1^2}} \, ,
\end{align*}
where we suppressed the $P$-dependence $ \Gamma_0 =  \Gamma_0(c;P)$ for notational convenience. For the Taylor polynomial of degree $M \in \mathbb{N}$ of nested functions we have by the chain rule that
\begin{align*}
    &\Gamma_0(v(0)) +  \dfrac{d}{dc}\Gamma_0(v(0))c + \frac12\dfrac{d^2}{dc^2}\Gamma_0(v(0))c^2 + \ldots + \frac{1}{M!}\dfrac{d^M}{dc^M}\Gamma_0(v(0))c^M\\[.2cm]
     & \quad =  \calF(v(0)) +
    \left(
    \begin{array}{cccc}
        v'(0) - \dfrac {\sqrt{2}} 3 &   0   & \cdots & 0 \\
          *   & v'(0) & \ddots & \vdots \\
          \vdots   & \ddots &  \ddots & 0 \\
          *   & \cdots &  * & v'(0) \\
    \end{array}
    \right)
    \left(
    \begin{array}{c}
         \dfrac{d}{dV_1} \calF(v(0)) c\\
         \vdots \\
          \dfrac{1}{M!} \dfrac{d^{M}}{dV_1^{M}} \calF(v(0)) c^M\\
    \end{array}
    \right) \\[.2cm]
     & \qquad \quad  =  \calF(0) +
    \left(
    \begin{array}{cccc}
        \dfrac{\tau_1}{2\sqrt{d_1}} - \dfrac {\sqrt{2}} 3 &   0   & \cdots & 0 \\
          *   & \dfrac{\tau_1}{2\sqrt{d_1}}& \ddots & \vdots \\
          \vdots   & \ddots &  \ddots & 0 \\
          *   & \cdots &  * & \dfrac{\tau_1}{2\sqrt{d_1}} \\
    \end{array}
    \right)
    \left(
    \begin{array}{c}
         \dfrac{d}{dV_1} \calF(0) c\\
         \vdots \\
          \dfrac{1}{M!} \dfrac{d^{M}}{dV_1^{M}} \calF(0) c^M\\
    \end{array}
    \right) \, .
\end{align*}
{Hence, for a triangular matrix $S$ with diagonal $\left(1,\frac{\tau_1}{2\sqrt{d_1}}-\frac{\sqrt{2}}3, \frac{\tau_1}{(2\sqrt{d_1}})\mathds{1}_{M-1}\right)$ we have
\[
\left(\frac{d^j}{dc^j}\Gamma_0(0)\right)_{j=0}^M = S \left(\frac{d^j}{dc^j}\calF(0)\right)_{j=0}^M.
\] 
}
Since $\tau_1/(\sqrt{d_1})\neq 0$ the inverse function theorem gives that,
if $\tau_1/(2\sqrt{d_1}) - \sqrt{2}/3 \neq 0 $, there is a 1-to-1 correspondence between the Taylor coefficients of $\Gamma_0 $ and those of $\calF$ in a neighbourhood of $c = 0$. 
In case $\tau_1/(2\sqrt{d_1}) = \sqrt{2}/ 3$, the right hand side gives a surjection
when additionally invoking $\tau_1$. 
As in the proof of Proposition~\ref{P:EXeps}, this extends to $\eps>0$ sufficiently small, and the claim about $\Geps$ follows by Lemma~\ref{lemma:connections}. 

By Proposition~\ref{prop:reduced_system} the speed ODE \eqref{SPEED:ODE} is scalar and given by \eqref{e:N1CMF}. If all of the finitely many roots of $\Gamma_0$ near zero are simple they remain simple under perturbation, and thus are in 1-to-1 correspondence with simple roots of $\Geps$ from \eqref{SPEED:ODE} for $0<\eps\ll1$. Hence, the dynamics is strictly monotone between and away from equilibria, and thus topologically equivalent as claimed. 
\end{proof}

Theorem~\ref{thm:one_slow_sing} means that in a generic case with $N=1$ (albeit near a bifurcation point), the dynamics is directly related to that generated by $\Gamma_0$. In a non-generic case with higher multiplicity of roots, the realisation of such a degeneracy on the center manifold requires a choice of $\calF$ from the proof of Theorem~\ref{thm:one_slow_sing} that is not known explicitly. However, such $\calF$ will be a perturbation from a choice of $\calF$ that realises the degeneracy in $\Gamma_0$. To illustrate basic examples let
\begin{align}\label{e:N1nonlin}
\calF^1(V) =  \gamma+ \alpha V +  \beta V^2 + \kappa V^3 + \rho V^4,
\end{align}
so that
\begin{align}\label{e:1slow_ex}
\Gamma_0(c) = \gamma+ \left(\frac{\alpha\tau_1}{2\sqrt{d_1}}- \frac{\sqrt{2}}{3}\right)c 
+ \frac{\beta\tau_1^2}{4d_1}c^2
+\frac{(2\kappa-\alpha)\tau_1^3}{16 d_1^{3/2}} c^3
+ \frac{(\rho\tau_1^2-\beta)\tau_1^2}{16 d_1^2} c^4 + \calO(|c|^5).
\end{align}
\textbf{Cusp.} From \eqref{e:1slow_ex} we see that a cusp singularity occurs in $\Gamma_0$, {\emph{i.e.}}, $\Gamma_0=\calO(|c|^3)$, precisely for $\alpha = 2\sqrt{2 d_1}/(3\tau_1)$, $\beta=0$, and $\kappa\neq \sqrt{2 d_1}/(3\tau_1)$, and can be unfolded by $\gamma$ and $\alpha$. In Figures~\ref{fig:cusp} and~\ref{fig:butter}(a) we plot numerical examples for the resulting bifurcations and dynamic front solutions.\\
\begin{figure}[!t]
\begin{center}
\begin{tabular}{ccc}
\includegraphics[width=0.3\textwidth]{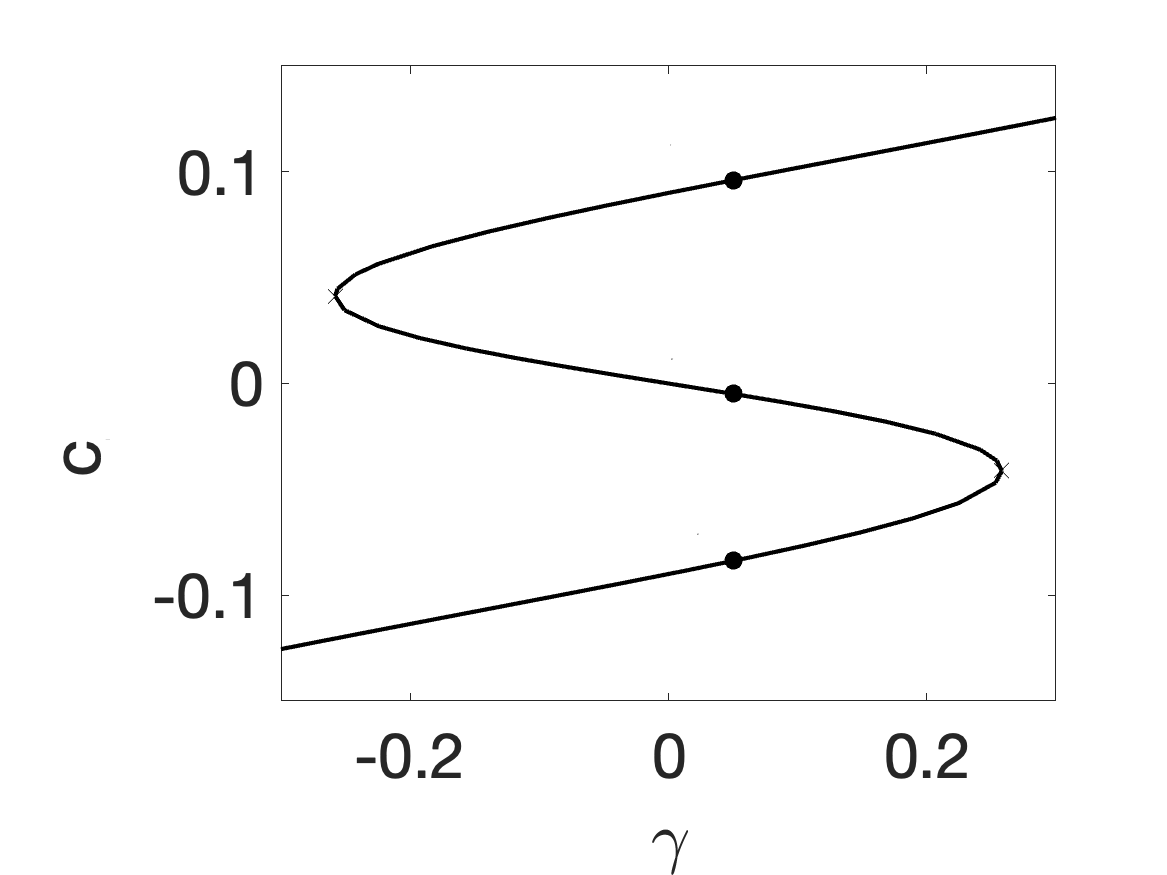}
&\includegraphics[width=0.3\textwidth]{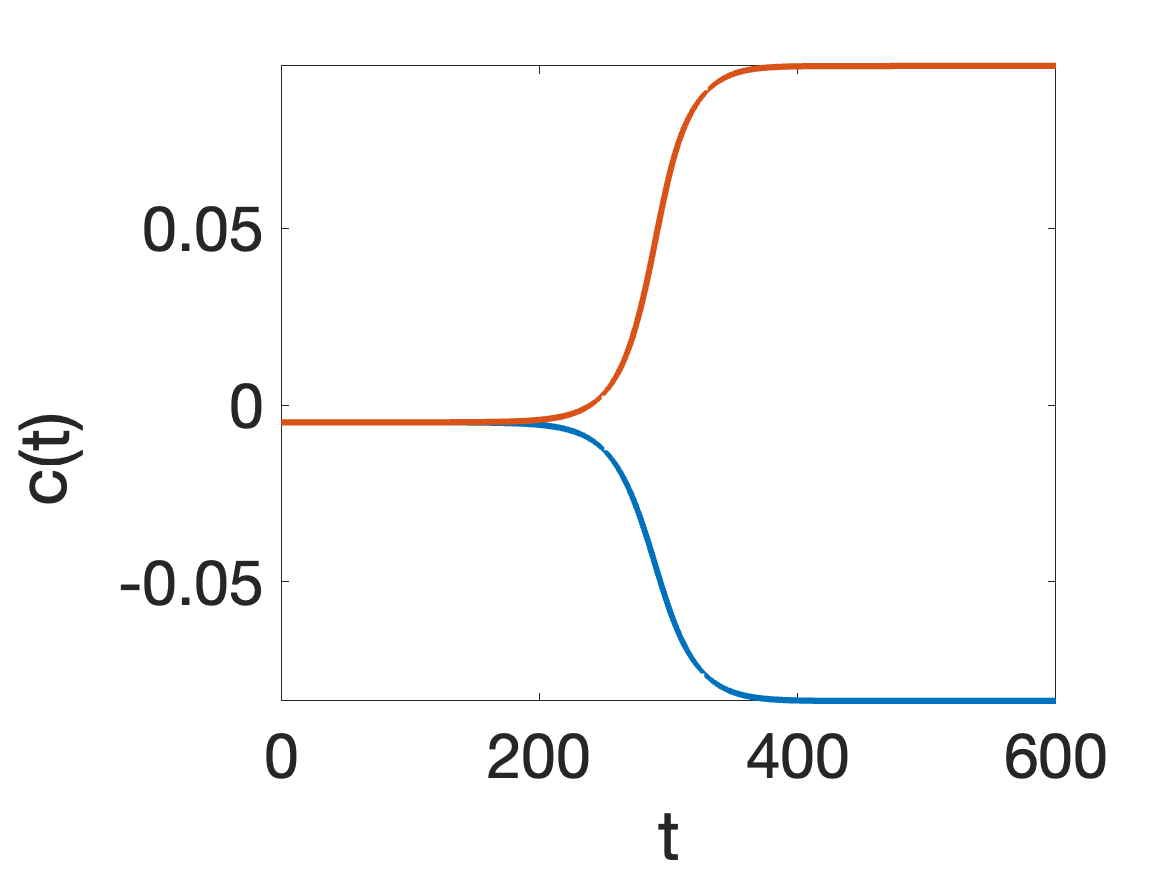}
&\includegraphics[width=0.3\textwidth]{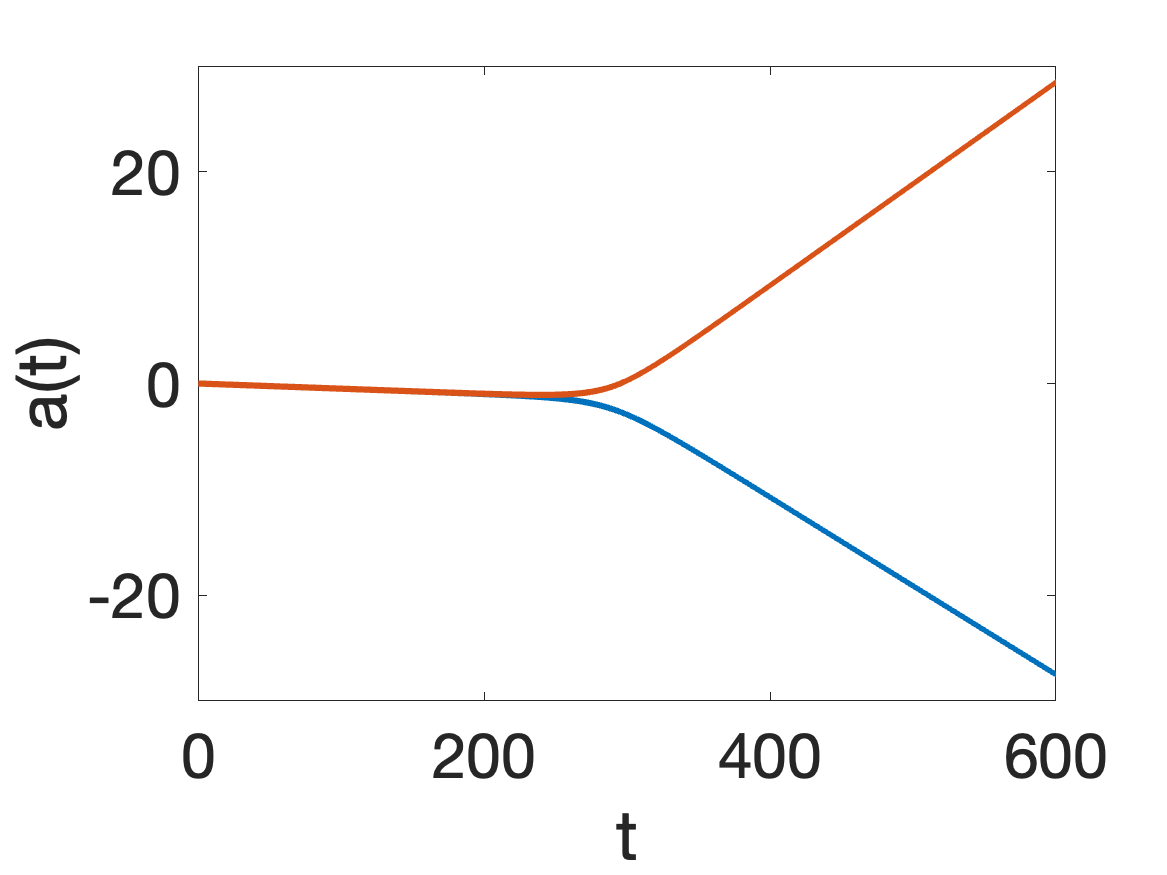}\\
(a) & (b) & (c)
\end{tabular}
\caption{(a) Numerical bifurcation diagram for \eqref{eq:multi-component-RD} with \eqref{e:N1nonlin} and $d_1=\tau_1=1$, $\kappa=-1$, $\rho=0$, $\eps=0.2$, $\beta=0$ and $\alpha=2$. (b,c) Numerical computation of dynamic front solutions upon perturbing from the solution with $c\approx 0$ marked with a bullet in panel (a); (b): The speed approaches the other equilibria marked with bullets in (a); (c): spatial position without co-moving frame. See \S\ref{s:num} for some details on the numerical method. 
\label{fig:cusp}}
   \end{center}
   \end{figure}
\\
\noindent\textbf{Butterfly.} Again we read from \eqref{e:1slow_ex} that a butterfly singularity occurs in $\Gamma_0$, {\it i.e.}, $\Gamma_0=\calO(|c|^4)$, precisely for $\alpha = 2\sqrt{2 d_1}/(3\tau_1)$, $\beta=0$, $\kappa= \sqrt{2 d_1}/(3\tau_1)$, and $\rho\neq0$. It can be unfolded by $\gamma, \alpha, \beta$ and $\kappa$. In Figure~\ref{fig:butter}(b) we plot a numerical example for a signature bifurcation diagram from the unfolding of a butterfly singularity. Notably, the quadratic coefficient $\beta=0$, but the odd symmetry is broken due to the non-zero quartic coefficient~$\rho$. 
\begin{figure}[!t]
\begin{center}
\begin{tabular}{cc}
\includegraphics[width=0.32\textwidth]{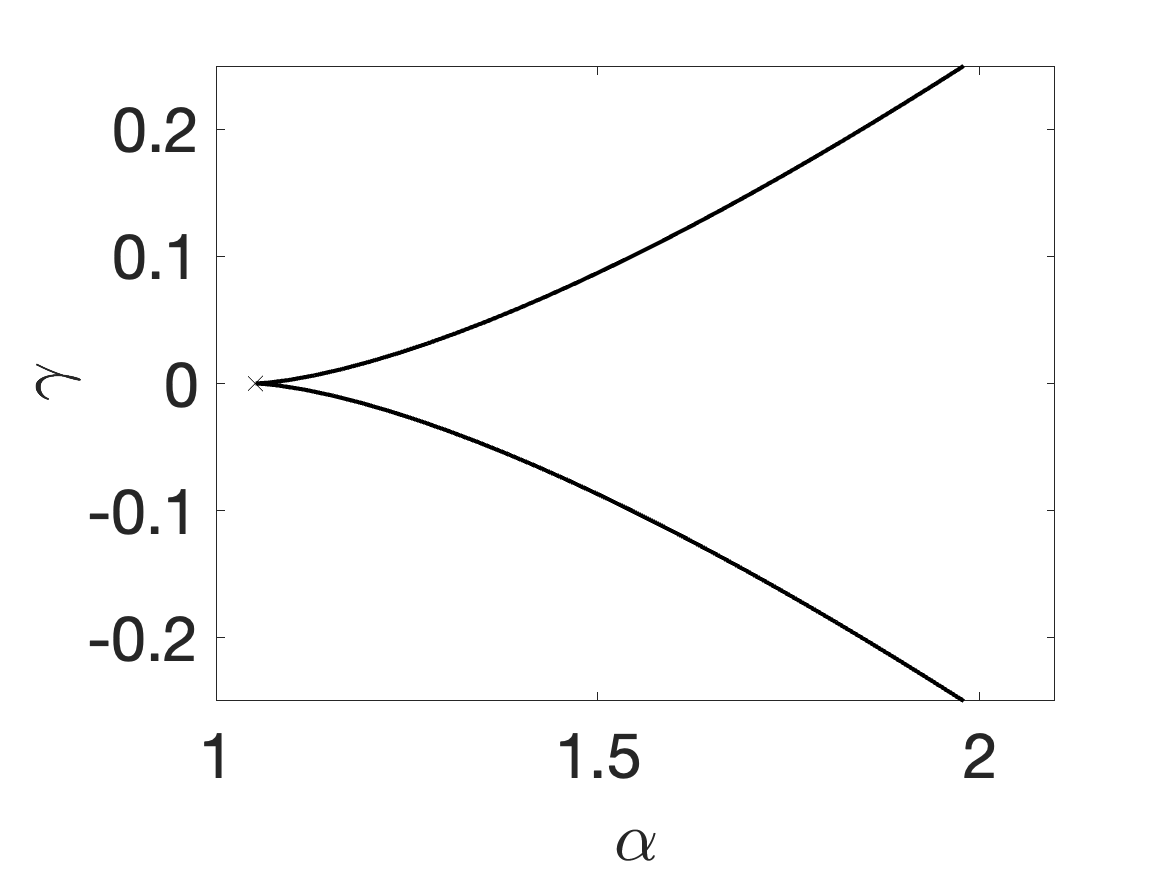}
 &\includegraphics[width=0.32\textwidth]{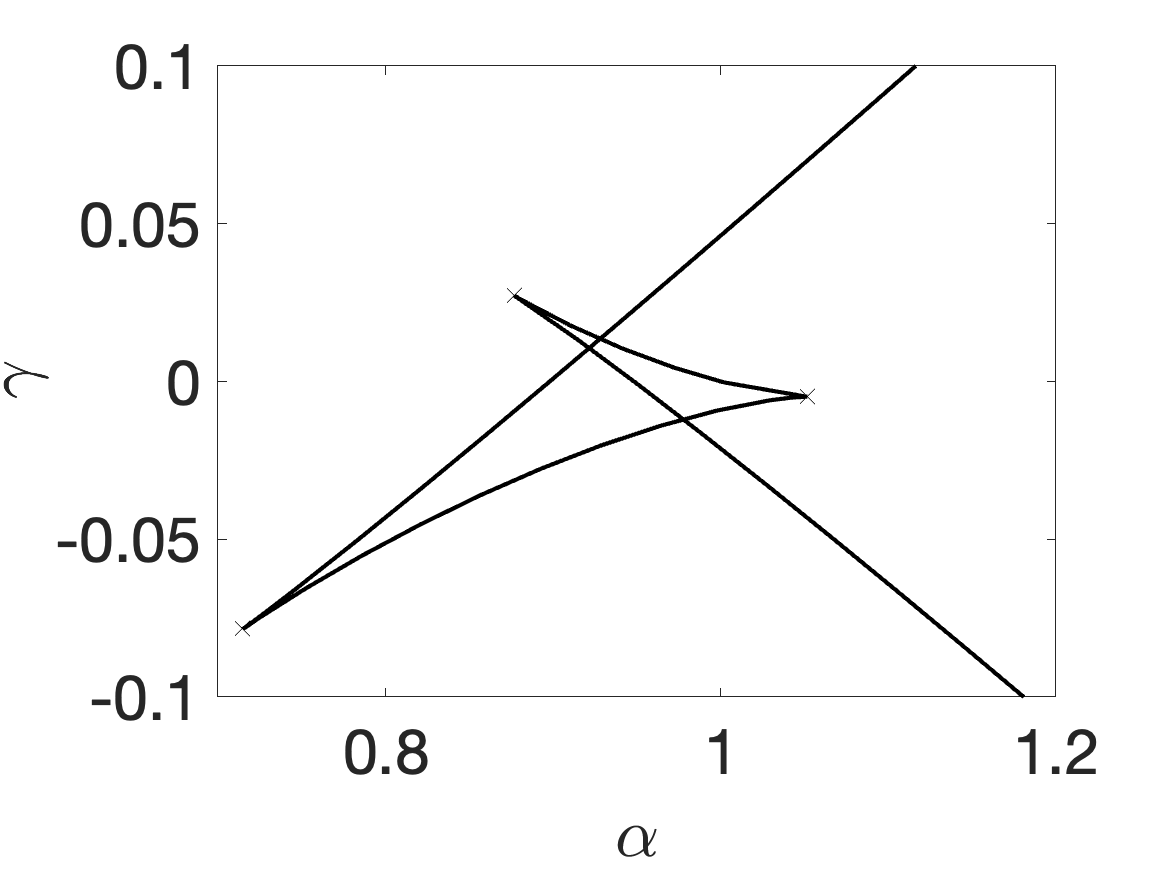}\\
(a) & (b) 
\end{tabular}
\caption{Numerical bifurcation diagrams of fold points for \eqref{eq:multi-component-RD} with \eqref{e:N1nonlin} and $d_1=\tau_1=1$, $\beta=0$, $\eps=0.2$, and (a): $\kappa=-1$, $\rho=0$ as in Figure~\eqref{fig:cusp}, 
(b): $\kappa=2$, $\rho=-0.5$. The numerical setup is as in Figure~\eqref{fig:cusp}. \label{fig:butter}}
   \end{center}
\end{figure} 
\begin{remark}[General coupling function $\calF^N$ and $N \geq 1$]\label{R:GEN}
Following up on Remark~\ref{r:impactN}, we comment on coupling functions that depend on more that one slow variable.
\begin{itemize}
\item[(a)]
From Proposition~\ref{L:EX}, if $\tau_j/d_j$ are pairwise distinct, for $\alpha_j$ as in \eqref{EX:MAXa} we have $\calT_{\Gamma_0}(c) = \mathcal{O}(c^{2N})$, which in particular means a steady state bifurcation so that the $\calO(\lambda)$-equation in \eqref{TAYLOR} holds. A necessary condition for one-dimensional center manifold $\Np=1$ is that the $\calO(\lambda^2)$-equation in \eqref{TAYLOR} does not hold. Indeed, this is generic: the summands of the $\calO(\lambda^2)$-equation are those of the $\calO(\lambda)$-equation with an extra factor $\tau_j$. Since $\tau_j,d_j>0$ we can write $d_j=\tau_j r_j$, with $r_j>0$, which makes the summands of the $\calO(\lambda)$-equation independent of $\tau_j$. Thus, for generic $\tau_j$, the $\calO(\lambda^2)$-equation in \eqref{TAYLOR} does not hold if that for $\calO(\lambda)$ holds. However, for $N>2$, $\Np=1$ requires to rule out eigenvalues on the imaginary axis (and with positive real parts for a stable center manifold). While this can be checked using the Evans function in specific cases, general statements seems challenging. For $N=2$ this is possible since only one real eigenvalue needs to be controlled, see \cite{CBDvHR15}.
\item[(b)] Suppose the coupling function $\calF^N$ for $N>1$ and $\Np=1$ depends on more that one slow variable, e.g., 
\begin{align*}
0= \Gamma_0(c) = \mathcal{F}^N(V^*_1(c), V^*_2(c)) - \frac13\sqrt2 c \,,
\quad
V^*_k(c) =  
\frac{c \tau_k}{\sqrt{4 d^2_k+c^2 \tau_k^2}}\,,  k = 1,2\,.
\end{align*}
In such a case, following the proof of Theorem~\ref{thm:one_slow_sing}, would require a Taylor expansion of nested multi-variable functions. While it is tedious to pursue this procedure in full generality, it is a possible strategy for specific cases. In particular, a 1-to-1 correspondence can still hold. Indeed, this was done for $\calF^N=\calF^N_\beta$, $N=2$ and $\Np=1$ in \cite{CBDvHR15}.
\end{itemize}
\end{remark}
\subsection{Chaotic dynamics for the front speed}
\label{S:1+3}
The goal in this section is to find chaotic dynamics of a single front solution in the sense that the speed $c(t)$ of the dynamic front solution satisfies the speed ODE {\eqref{SPEED:ODE}} with chaotic dynamics. This requires an at least three-dimensional ODE, and 
hence we set $N=3$, so $\vec{V}=(V_1,V_2,V_3)$, and we choose parameters 
$P=P_3^\eps+\cP$ with $|\cP|$ and $\eps>0$ sufficiently small so that  
Proposition~\ref{prop:reduced_system} holds for $\Np=3$. In order to apply known results from ODE theory, it turns out that also the symmetry of the original system with affine coupling term needs to be broken, cf.\ Remark~\ref{ADV}, which we do here by quadratic terms in $\calF^N$ of \eqref{eq:multi-component-RD} (see also \eqref{NONLnew}) so that 
\begin{align}\label{eq:F3}
\calF^3(\vec{V}) = \gamma+\sum_{j=1}^3 \alpha_j V_j + \sum_{j=1}^3 \beta_j V_j^2 + \calO(\|\vec{V}\|^3)  \,, \qquad \alpha_j, \beta_j, \gamma \in \mathbb{R} \, .
\end{align}
The existence condition for uniformly travelling front solutions and the corresponding Evans function for stationary front solutions $Z_{\rm SF}$ read for $N=3$, cf.\ \eqref{E:N} and \eqref{EV}, 
\begin{align}
 \Gamma_0(c) &=  \mathcal{F}^3
 (\vec{V}^*) - \frac13\sqrt2 c = 0 \quad, 
 V_j^* =  \frac{c \tau_j}{\sqrt{4 d^2_j+c^2 \tau_j^2}}\,,  \qquad j = 1, 2, 3 \,,
 \label{3:E}\\
E_0(\lambda) = &
\lambda + 3\sqrt2\sum_{i=1}^3  
\partial_j \mathcal{F}^3
(\vec{V}^*) \, \left(\frac{1}{\sqrt{c^2 \tau_i^2+ 4 d_i^2 (\lambda \tau_i+1)}}-\frac{1}{\sqrt{4 d_i^2+ c^2 \tau_i^2}} \right)  \, .\label{3:EV}
\end{align}
In particular, $\gamma=0$ always admits the solution $c=0$ to $\Gamma_0(c)=0$, which corresponds to a stationary front solution of \eqref{eq:multi-component-RD} for $N=3$. The system parameters are in this setting
 \[
P = \big(\gamma, (\alpha_j,\beta_j,\tau_j, d_j)_{j = 1, 2, 3}\big)
\]
and we consider $P = P_{\rm org}^{\varepsilon} + \check{P}$ with $P_{\rm org}^\eps=P_3^\eps$ 
-- in particular we assume $\tau_1,\tau_2,\tau_3$ are pairwise distinct. Since we will invoke quadratic terms, the resulting organising center for $\eps=0$, $P_{\rm org}^0$, is a transcritical bifurcation point of $Z_{\rm SF}$  for which $E_0$ possesses a quadruple root, i.e.,
\[
\left.\frac{\mathrm{d}^i }{\mathrm{d} \lambda^i} E_0(0)\right|_{P=P_{\rm org}^{0}} = 0, \; \qquad i=0,1,2,3, 
\]
with the explicit set of parameters, cf.\ Proposition~\ref{L:N+1} and \eqref{e:Porg}, given by
\begin{align*}
 \alpha_j  =  \dfrac{2 \sqrt 2 d_j}{3 \tau_j} \prod_{k =1, k \neq j}^3 \dfrac{\tau_k}{\tau_k-\tau_j}
\,, \qquad j=1, 2, 3\,.
\end{align*}
By Proposition \ref{prop:reduced_system}, for any  sufficiently small $|\cP|$ there is a corresponding center manifold reduction to dynamic front solutions \eqref{eq:Zsplit1}. Here $c_1,c_2,c_3$ satisfy the three-dimensional ODE 
\begin{align}\label{SPEED:ODE3}
 \begin{pmatrix}
  \dot{c}_1 \\ \dot{c}_2 \\ \dot{c}_3 
 \end{pmatrix}
 = \varepsilon^2 
 \begin{pmatrix}
   0   & 1 & 0 \\
   0   & 0 & 1 \\
   0   & 0  & 0
  \end{pmatrix}
 \begin{pmatrix}
  {c}_1 \\ {c}_2 \\ {c}_3 
 \end{pmatrix}
 +
 \begin{pmatrix}
  0 \\ 0 \\ \varepsilon^2 \Gred(c_1, c_2, c_3; \check{P})  + \Gerr(c_1, c_2, c_3; \check{P})
  \end{pmatrix}  \, ,
\end{align}
with $\Gerr(c_1, c_2, c_3; \check{P})=\calO({ |c_2^2|+|c_3^2|+}\eps^2(|\vec{c}|^3 + |\vec{c}|^2|\cP|)) + o(\eps^2)$ and
\begin{align}\label{eq:NF_G3}
\Gred(c_1, c_2, c_3; \check{P}) = a_0(\check{P}) + \sum_{j=1}^3 a_j(\check{P}) c_j + 
{ c_1 \sum_{j=1}^3 a_{1j}c_j}
\, , \qquad a_j(\cP), a_{kj} \in \mathbb{R} .
\end{align}
In particular,  $ (c_1, c_2, c_3) = (0,0,0) $ is an equilibrium at $ \check{P} = 0 $ with nilpotent linearisation featuring a full Jordan chain. 
Based on this and the results obtained so far, our proof for chaotic dynamics relies on the results of \cite{IR05} through the following Lemma.
\begin{lemma}[Shilnikov homoclinic orbit]\label{lem:chaos}
Consider $N=3$ and parameters $P=P_{\rm org}^\eps+\cP$ with $|\cP|$, $\eps>0$ sufficiently small. For $a_0, a_2, a_3, a_{11}$ from~\eqref{eq:NF_G3} assume  
\begin{equation}\label{e:chaos_cond}
a_{11}\neq 0, \quad
\mathrm{rank}((\nabla a_j(0))_{j=0, 2,3})=3.
\end{equation}
Then there exist parameter perturbations $\check{P}$ with arbitrarily small $|\cP|$ such that \eqref{SPEED:ODE3} possesses a Shil'nikov homoclinic orbit to an equilibrium that can be transversally unfolded by $\check{P}$. Such $\check{P}$ satisfy $a_0(\check{P}) a_{11}<0$ and $a_2(\check{P})<0$ and the unfolding occurs by the resulting changes in $a_3(\check{P})$. 
\end{lemma}
\begin{proof}
We prove that the assumptions imply the sufficient conditions from \cite[Theorem 4.1]{IR05} for the statements. 
{Applying the scaling transformation from \S\ref{subsec:cmr_ode_terms_general}} gives \eqref{e:zscaled} for $\Np=3$. 
The assumption $a_{11}\neq 0$ admits the coordinate shift $\tilde c_1=c_1-a_1(\check{P})/(2a_{11})$ that removes the term $a_1(\cP) c_1$ from $\Gred$ in \eqref{eq:NF_G3}, and analogously, for $\eps>0$ sufficiently small, the term that is linear in $c_1$ from \eqref{SPEED:ODE3}.
Hence, we obtain \begin{align}\label{NF_IBAN}
\frac{d}{dT} \left( 
 \begin{array}{c}
  {z}_1 \\ {z}_2 \\ {z}_3 
 \end{array}
 \right)
 = 
\left( 
 \begin{array}{ccc}
   0   & 1 & 0 \\
   0   & 0 & 1 \\
   0   & 0  & 0
  \end{array}
 \right) 
 \left( 
 \begin{array}{c}
  z_1 \\ z_2 \\ z_3 
 \end{array}
 \right)
 +
 \left( 
 \begin{array}{c}
  0 \\ 0 \\   
  \check{G}(z_1, z_2, z_3;  \bar{\lambda}, \bar{\mu} ,\bar{\nu})
  \end{array}
 \right)  \, ,
\end{align}
where {$(\bar{\lambda}, \bar{\mu} ,\bar{\nu})=\vec{\nu}+ o_\eps(1)$ from \S\ref{subsec:cmr_ode_terms_general},} $  \bar{\lambda}^2 + \bar{\mu}^2 + \bar{\nu}^2 = 1 $ 
\begin{align*}
  \check{G}(z_1, z_2, z_3;  \bar{\lambda}, \bar{\mu} ,\bar{\nu}) = \bar{\lambda} + \bar{\mu} z_{ 2} + \bar{\nu} z_{ 3} +  (a_{11}{ +o_\eps(1)}) z_1^2 + \mathcal{O}(\delta) \, .
\end{align*}
This system has been studied in \cite{IR05} for $0<\delta\ll1$, cf.\ \cite[eq.\ (1)]{IR05}. It is known from the results of \cite{DIK01} that the only case for possibly chaotic dynamics is $(a_{11}{ +o_\eps(1)})\bar{\lambda}<0$, $\bar{\mu}<0$ and for suitable choices of such $(\bar{\lambda}, \bar{\mu} ,\bar{\nu})$ structurally stable dynamics as claimed was established in \cite[Theorem 4.1]{IR05}. Hence, a sufficient condition to realise such parameters is that $\{(a_j(\check{P}))_{j=0,2,3}: |\check{P}|<r\} +o_\eps(1)\subset\R^3$ contains an open ball around the origin{, where $r>0$ is sufficiently small so Proposition~\ref{prop:reduced_system} applies}. This is in particular the case if \eqref{e:chaos_cond} holds, and $\eps>0$ is sufficiently small. For such a fixed $\eps$ we can apply \cite[Theorem 4.1]{IR05} and it follows that $a_3$ is an unfolding parameter.
\end{proof}
\begin{remark}
From \cite{IR05} one can extract more detailed information. The Shil'nikov homoclinic orbit is actually a result of unfolding a T-point heteroclinic cycle between two equilibria with one transverse and one codim-two heteroclinic connection with leading oscillatory dynamics near the equilibria. The T-point occurs in \eqref{NF_IBAN} at $\delta=\varepsilon=0$ for $\bar{\nu}=0$ and a specific choice of $\bar{\lambda}, \bar{\mu}$, and is unfolded by $\bar{\nu}$ and $\bar{\lambda}$ or $\bar{\mu}$.
Hence, this also occurs in \eqref{SPEED:ODE} and thus for dynamic front solutions. \end{remark}

In order to verify the possibility to generate chaotic dynamics for front solutions in \eqref{eq:multi-component-RD} in the given setting, it remains to show \eqref{e:chaos_cond}. For this, we use the information provided by the leading order (in $\varepsilon$) existence function and Evans function.
\begin{theorem}[Chaotic dynamics for $N=\Np=3$]\label{thm:chaos}
Let the conditions of Lemma~\ref{lem:chaos} be fulfilled so that Proposition \ref{prop:reduced_system} holds with $N=\Np=3$ and assume for $\calF^3(\vec{V})$ from \eqref{eq:F3} that
\begin{align}\label{e:quad_cond}
\sum_{j=1}^3\frac{\beta_j \tau_j^2}{ d_j^2} \neq 0, \quad \tau_j\neq\tau_k>0, \; j\neq k,\; j,k=1,2,3. 
\end{align}
Then for any sufficiently small $\eps>0$ there are arbitrarily small parameter perturbations $\check{P}$ to $P=P_{\rm org}^\eps+\check{P}$ such that the dynamic fronts move chaotically in the following sense. The speed ODE~\eqref{SPEED:ODE3} for dynamic fronts of \eqref{eq:multi-component-RD} with $\calF^N(\vec{V})=\calF^3(\vec{V})$  
possesses a Shil'nikov homoclinic orbit that is transversally unfolded by modifying $\cP$. Specifically, such $\cP$ necessarily satisfy $a_{11}a_0(\check{P})<0$ and $a_2(\check{P})<0$, and can be realized by suitable choices of $\gamma, \alpha_1, \alpha_2$. The unfolding can be realised by $\alpha_1, \alpha_2$ through their control of $a_3$. The statement trivially extends for any $N> 3$ by adjusting the coupling function as in Remark~\ref{r:impactN}.
\end{theorem}
\begin{proof}
Taylor expanding the existence condition \eqref{3:E}  gives
\begin{align}
0= \calT_{\Gamma_0}(c;\cP) &= \gamma 
 -\frac{3\sqrt{2}}{2} e_0\, c 
+ \frac14 \left(\sum_{j=1}^3 \frac{\beta_j \tau_j^2}{ d_j^2} \right)c^2 
-  \frac{1}{16} \left( \sum_{j=1}^3 \alpha_j \dfrac{\tau_j^3}{d_j^3}\right) c^3 + \mathcal{O}(c^4) \,, \nonumber
\end{align}
with
$$e_0 := 1- \frac{3\sqrt{2}}{4}\sum_{j=1}^3  \frac{\alpha_j \tau_j }{d_j} \, .$$
Taylor expanding the Evans function \eqref{3:EV}, cf.\ \eqref{eq:taylor_evansN},  around $c = c_1=0$ gives
\begin{align*}
\calT_{E_0}
(\lambda;\cP)) &=  
e_0 \lambda 
+ \frac{3\sqrt{2}}{2}\sum_{k=2}^\infty \left( (-1)^k
\frac{(2k-1)!!}{(2k)!!}
\sum_{j=1}^3\ \frac{\alpha_j \tau_j ^k}{d_j} \right) \lambda^k \,.
\end{align*}
By Lemma \ref{lemma:connections}, for any $K\in\N_0$, the Taylor polynomials of $\Gred$, $\Gamma_0$, cf.\ Definition~\ref{TAYDEF}, satisfy
\begin{align*}
\calT^K_{\Gred(c,0,0;\cP)}(c,\cP) = \calT^K_{H_0(c,\cP) \Gamma_0(c,\cP)}(c,\cP).
\end{align*} 
In particular, since $H_0(0;0)\neq 0$, for any parameter $p$ from $P$ and also for $p=c$ we have
\begin{align}\label{e:Gderivative1}
\partial_p \Gamma_0(0)\neq 0 \Rightarrow \partial_p \Gred(0) \neq 0.
\end{align} 
Since $\partial_\lambda E_0(0)=e_0=0$ at $P=P_{\rm org}^0$ we have $\partial_c \Gamma_0(0)= 0$
so $\partial_c^2 \Gred(0)=2a_{11}$ implies
\begin{align*}
\partial_c^2 \Gamma_0(0) \neq 0 \Rightarrow a_{11} \neq 0.
\end{align*}  
Hence, since $\partial_c^2 \Gamma_0(0) \neq 0$ by assumption \eqref{e:quad_cond}, the first condition in \eqref{e:chaos_cond} holds. 

Towards the rank condition in \eqref{e:chaos_cond}, we first observe that $\partial_\gamma \Gamma_0=1$ and this implies, by \eqref{e:Gderivative1} and \eqref{eq:NF_G3}, that $\nabla_{\gamma} a_0(0)\neq 0$.  Since $a_2, a_3$ are independent of $\gamma$,  it therefore remains to show
\begin{equation}\label{e:rank_cond2}
\mathrm{rank}((\nabla a_j(0))_{j=2,3})=2,
\end{equation}
for which we use the Evans function \eqref{3:E} as in \S\ref{s:organising}. The characteristic equation of the linearization of \eqref{SPEED:ODE3} in  $c_1=c_2=c_3=0$ is given by
\begin{align*}
\lambda^3 - a_3(\cP) \lambda^2 - a_2(\cP) \lambda - a_1(\cP) = 0 
\end{align*}
so that the product structure in \eqref{e:evansWeierN} and Lemma~\ref{lemma:connections} imply for any $K\in\N$ that 
\begin{align*}
 \mathcal{T}_{\ta_j}^K(\cP) = \mathcal{T}_{a_j}^K(\cP), \; j=1,2,3,
 \end{align*}
which allows us to replace $a_j$ in \eqref{e:rank_cond2} by $\ta_j$, $j=2,3$. 
Therefore, Lemma~\ref{l:linearunfold} implies that~\eqref{e:rank_cond2} holds. Hence, \eqref{e:chaos_cond} holds and thus the claim follows from Lemma~\ref{lem:chaos}.
\end{proof}
\begin{remark}
The proof of Theorem \ref{thm:chaos} is simplified by the fact that it only requires qualitative information of normal form coefficients on the center manifold that are directly accessible through the existence function (coefficients $\gamma,a_{11})$ and the Evans function (coefficients $\ta_2,\ta_3$). In general, gaining information on the normal form coefficients requires the adjoint eigenfunctions in order to compute the center manifold reduction. By way of the existence function and Evans function we can (only!) obtain  information about coefficients of terms $c_jc_1^k$ in $\calT_{\Gred}^{k+1}$ for $j=1,\ldots,N$, $k\in\N_0$, i.e., linear in $c_j$. This approach has been used before in \cite{CBvHHR19} for $N=2$. 
\end{remark}

\subsection{Numerical results}\label{s:num}
\begin{figure}
\begin{tabular}{cc}
   \includegraphics[width=0.4\textwidth]{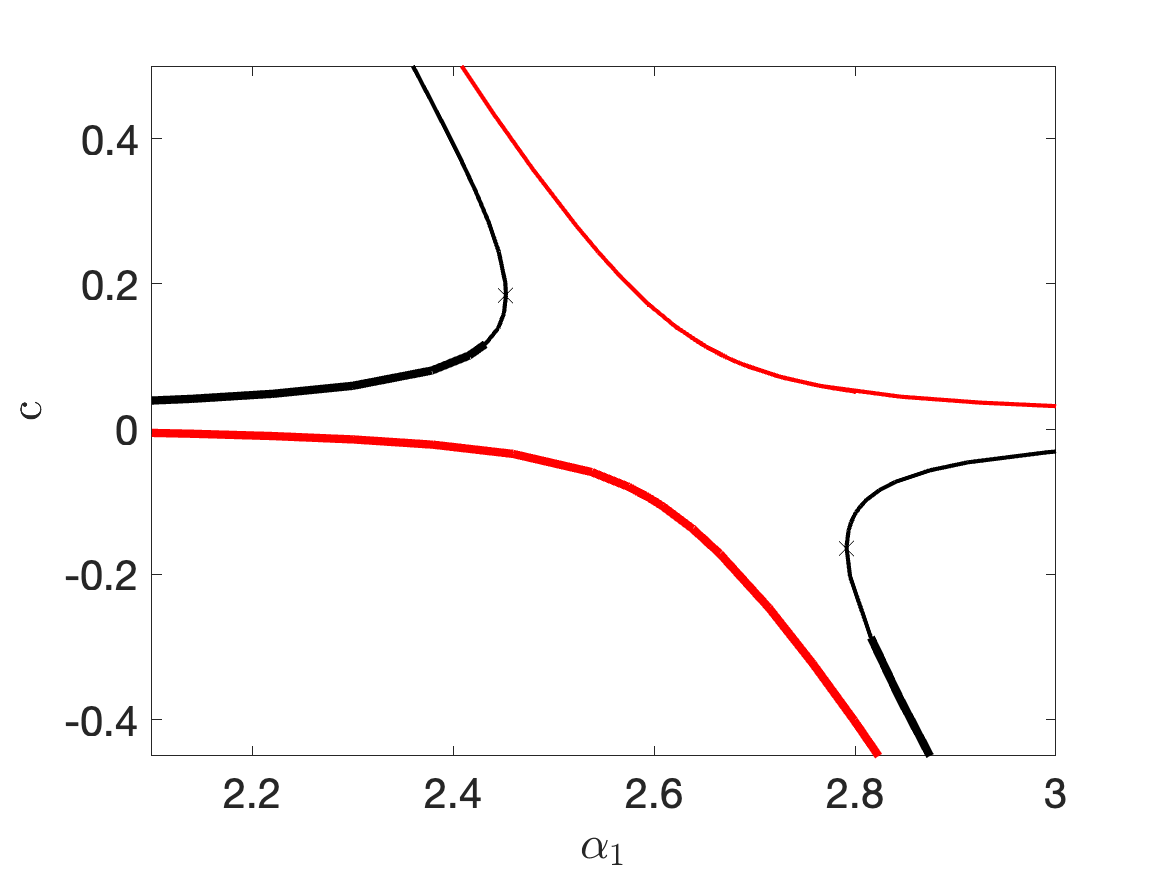}
   & \includegraphics[width=0.4\textwidth]{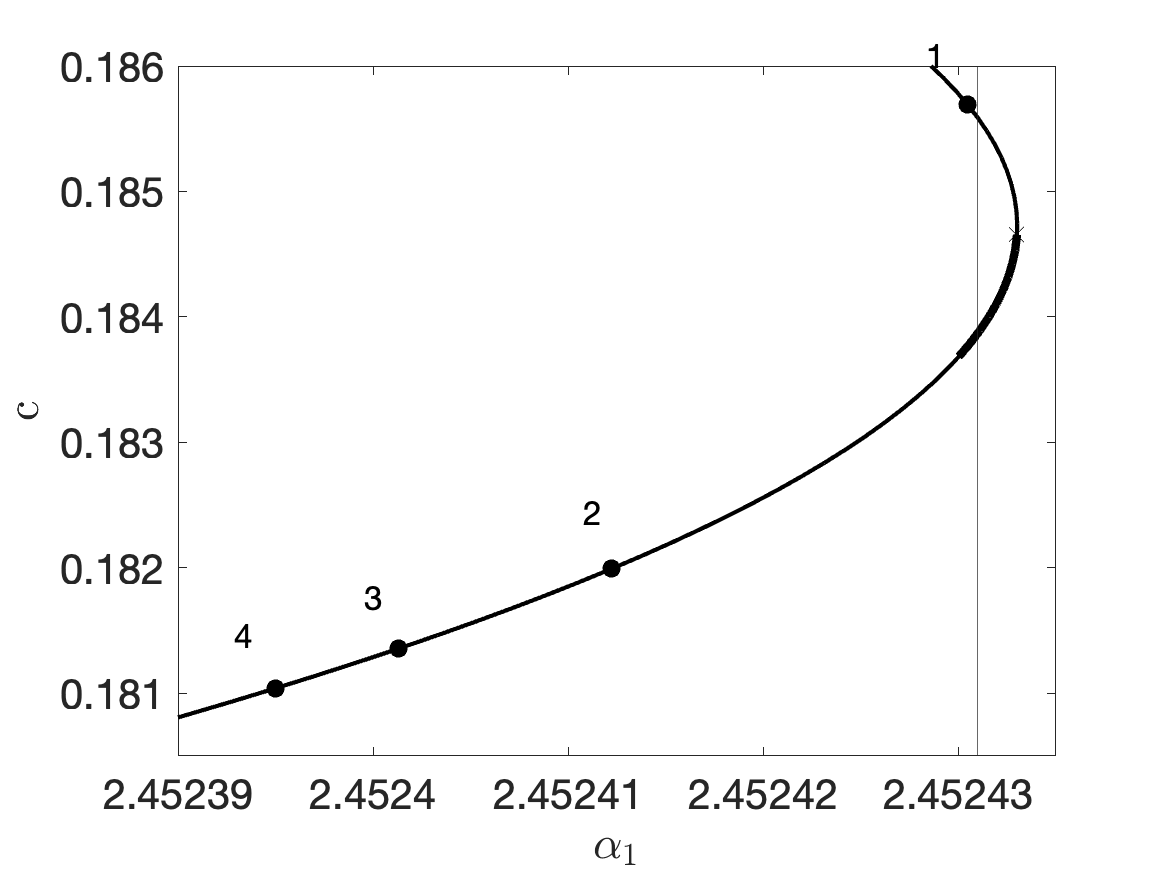}\\
   (a) & (b)  
   \end{tabular}
   \caption{(a) Branches of travelling front solutions as described in the text with parameters $\eps=0.03$ and \eqref{e:transpar22} with $\gamma=0$ (red) and $\gamma=0.11$ (black); thick lines indicate stable parts of a branch. (b) Selected magnification of (a) very near the left fold point;  
the vertical line marks the $\alpha_1$-value of solution \#1, which co-exists with a slower stable front solution. }
   \label{f:trans_branches}
\end{figure}

We present numerical computations that illustrate aspects of the preceding analysis 
for \eqref{eq:multi-component-RD} with $N=3$ and $\mathcal{F}^N(\vec{V})= \mathcal{F}^3_\beta(\vec{V})$ from \eqref{NONLnew}. 
Unfortunately, the chaotic motion arising from Shil'nikov homoclinics that we proved to occur is challenging to directly observe numerically since the attractors have very small basins of attraction and the characteristic dynamics occurs over long time scales. A discussion of such a phenomenon and the relation to normal forms can be found in  \cite[e.g.]{WITTENBERG19971}. However, we can seek signature bifurcations of ODEs in this context, in particular the period doublings found in \cite{ACST85} and the Hopf-zero bifurcation near a nilpotent singularity \cite{BarrientosBook}. The additional difficulty is that we work with the PDE \eqref{eq:multi-component-RD} so that the reduced ODE is only indirectly accessible. 

Since Theorem~\ref{thm:chaos} requires quadratic terms, we start with parameters near the analytically predicted transcritical bifurcation point with four small eigenvalues from \eqref{e:transpar22}, and we fix $\eps=0.03$ for the numerical computations discussed next. All of these are done with the \textsc{Matlab} package \textsc{pde2path} \cite{p2pbook,p2p}. This software was also used for the computations of Figure~\ref{fig:cusp}, where -- to illustrate the robustness -- the domain is short $[-10,10]$ with homogeneous Neumann boundary conditions and coarse with $95$ grid-points that lie denser near the interface.
For the results discussed next, the domain was $[-20,20]$ with $1077$ grid-points. 

In Figure~\ref{f:trans_branches}(a) we plot the branches resulting from this setting, with $\gamma=0$ in red. That the red branches are disjoint as in a generic unfolding of a transcritical bifurcation can be understood from the fact that $\eps>0$ so that the parameters are not precisely at a transcritical bifurcation point. From such an unfolding, we expect that nearby parameter values produce branches which are conversely connected. Indeed, we can find this for suitable $\gamma>0$, and we plot such branches for $\gamma=0.011$ in black. 

\begin{figure}
\begin{tabular}{ccc}
   \includegraphics[width=0.3\textwidth]{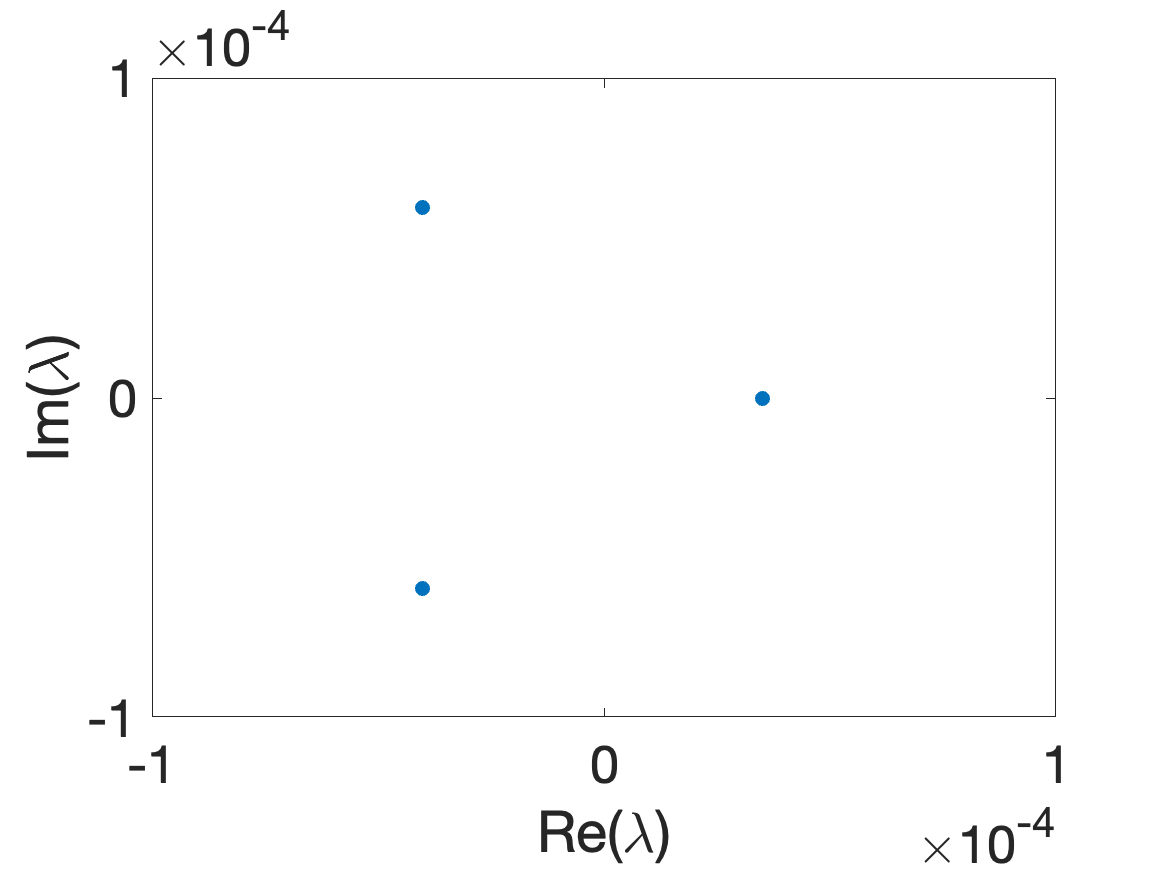}
   & \includegraphics[width=0.3\textwidth]{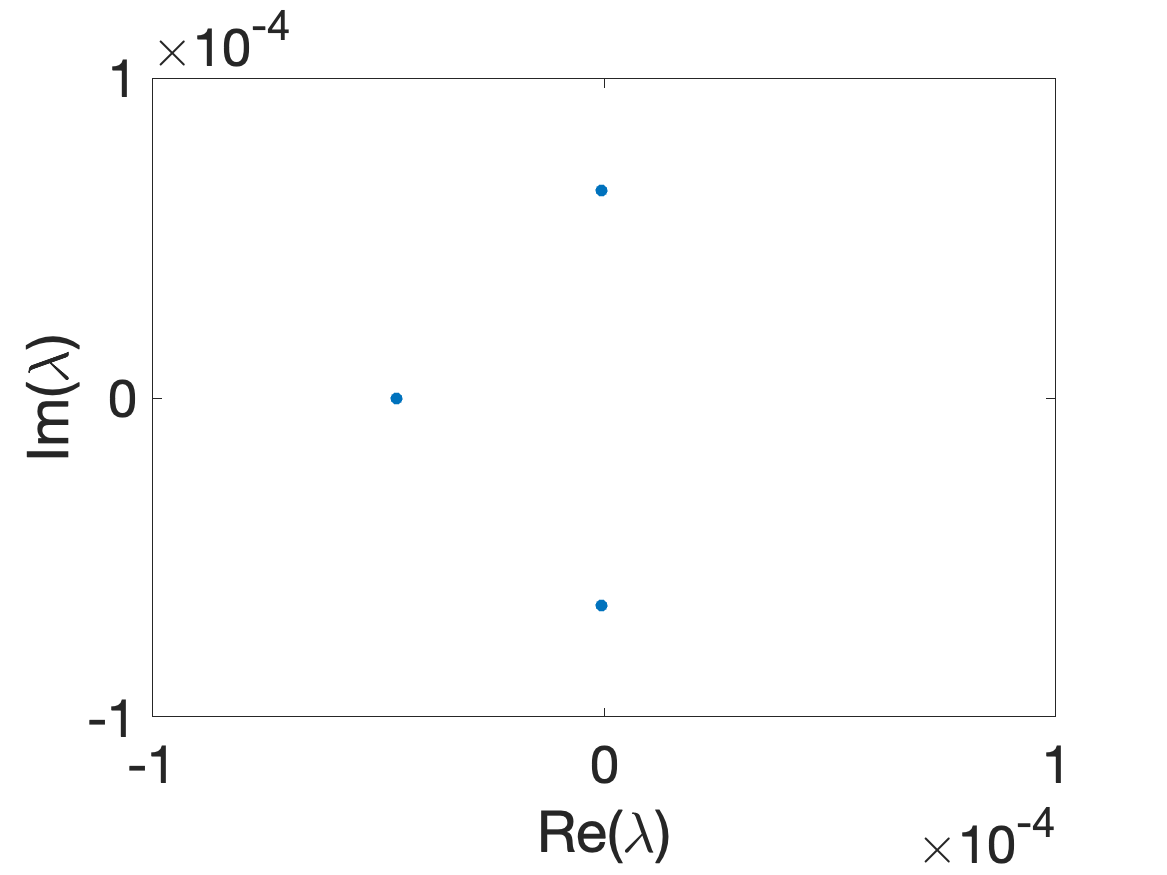} 
   &\includegraphics[width=0.3\textwidth]{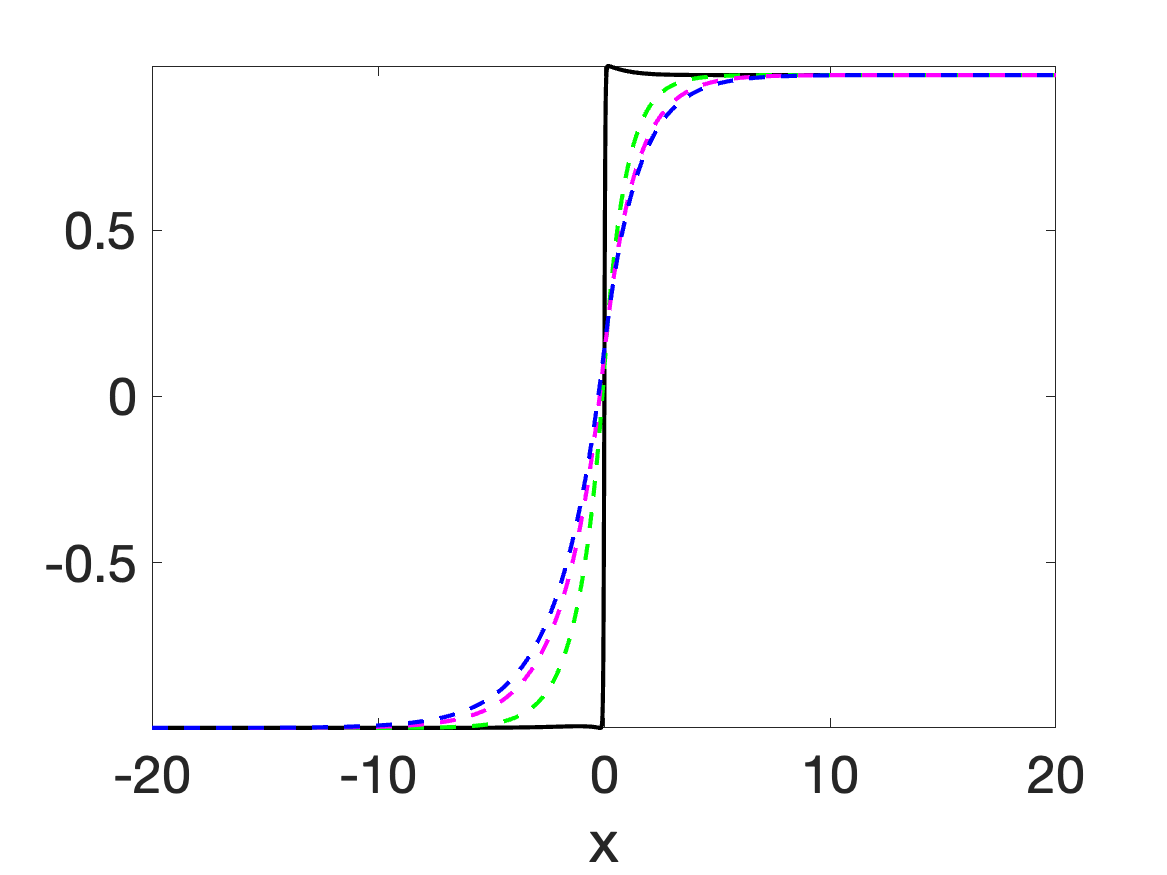}\\
   (a) & (b) & (c)
   \end{tabular}
   \caption{(a,b) Critical spectrum (translation mode removed) at unstable solution $\#1$ in Figure~\ref{f:trans_branches}(b), and at stable solution with the same value of $\alpha_1$ on lower part of the branch in Figure~\ref{f:trans_branches}(b). (c) Solution profile at stable solution. }
\label{f:trans_branches_spec}
\end{figure}

The magnification in Figure~\ref{f:trans_branches}(b) highlights that the fold point stabilises the underlying front solution, which then loses stability in a Hopf bifurcation on the lower part of this branch. At point~$\#1$ the unstable front solution indeed possesses one real unstable eigenvalue and, as expected from the Hopf bifurcation beyond the fold, the leading stable eigenvalues are a complex conjugate pair, see Figure~\ref{f:trans_branches_spec}(a). At the same value of $\alpha_1$ on the lower part of the branch the front solution is stable with a leading complex conjugate pair of eigenvalues (the real eigenvalue from the fold has become more stable than these), see Figure~\ref{f:trans_branches_spec}(b). To give an impression of the shape of the underlying front solutions we plot the profile of the latter solution in Figure~\ref{f:trans_branches_spec}(c). The proximity of the fold point and Hopf bifurcation suggests a nearby fold-Hopf (or Hopf-zero) point, which indeed appears in the unfolding of the nilpotent singularity, cf.\ \cite[e.g. Fig. 5.8]{BarrientosBook}, and we expect to find dynamics related to its unfolding. 

\begin{figure}
\begin{tabular}{cc}
   \includegraphics[width=0.4\textwidth]{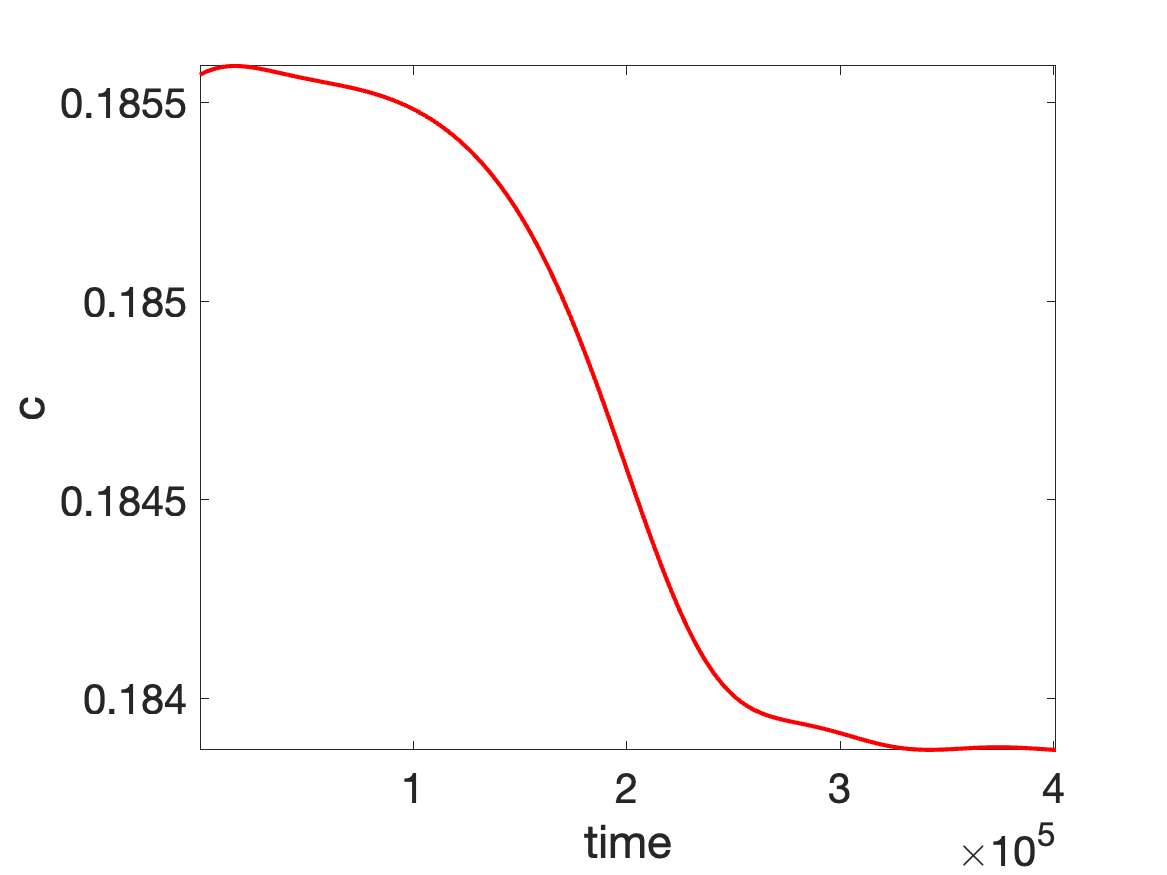}
&   \includegraphics[width=0.4\textwidth]{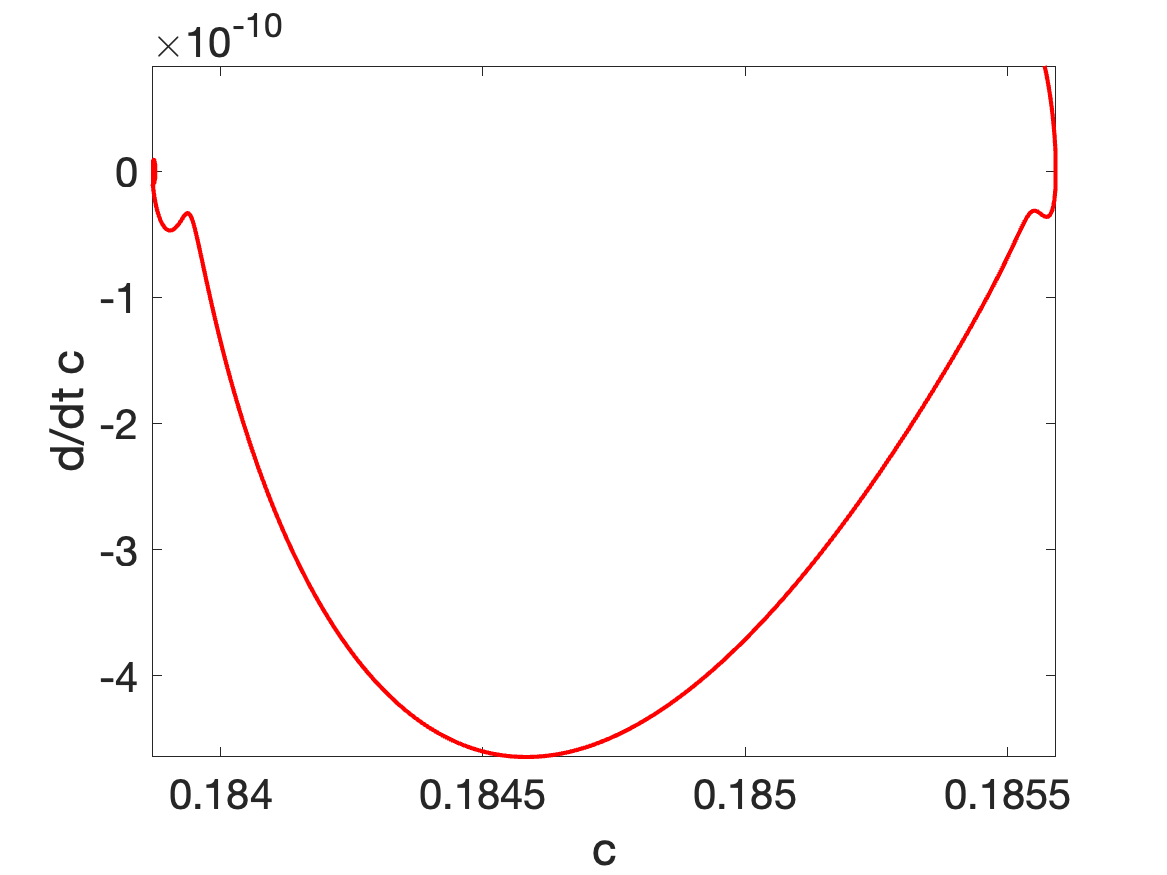} \\
(a) & (b) 
\end{tabular}
   \caption{Time simulation from perturbation off unstable solution $\#1$ in Figure~\ref{f:trans_branches}(b). In panel (a) we plot time versus speed and in (b) local speed versus its finite difference derivative, highlighting the oscillations near the steady states.
   }
   \label{f:trans_branches_sim}
\end{figure}

Upon perturbing the solution $\#1$ in Figure~\ref{f:trans_branches}(b) in the direction of the leading stable eigenspace we expect initial oscillations followed by a drift along the real unstable eigenspace. The nearby fold further suggests this unstable manifold connects to the stable front solution on the lower branch at this value of $\alpha_1$. The convergence to this should then be oscillatory, along the leading stable eigenspace. Numerical simulations indeed confirm all this as plotted in Figure~\ref{f:trans_branches_sim}. Here the time simulation implements `freezing', which determines an instantaneous speed that corresponds to $c$ by a projection of the time step predictor, see \cite[e.g.]{p2pbook}. 

\begin{figure}
\begin{tabular}{cccc}
   \includegraphics[width=0.22\textwidth]{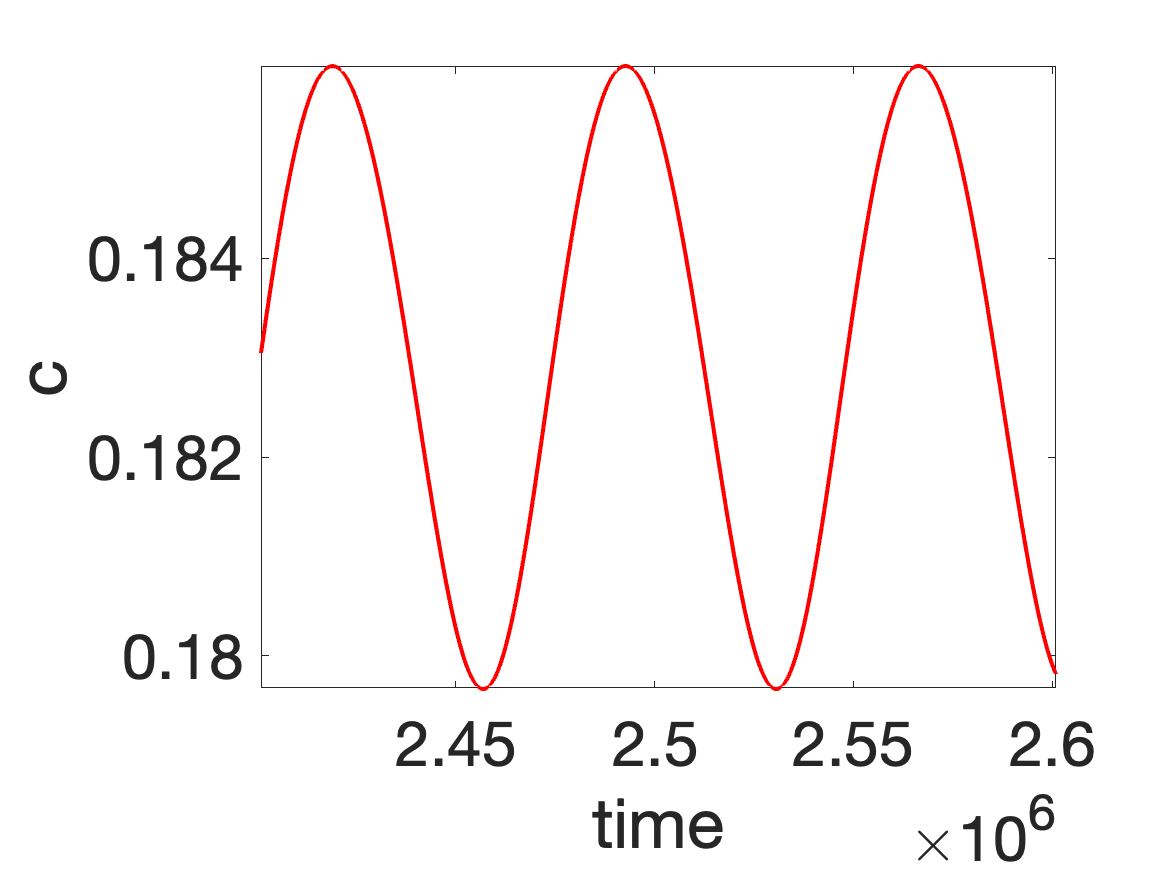} 
&   \includegraphics[width=0.22\textwidth]{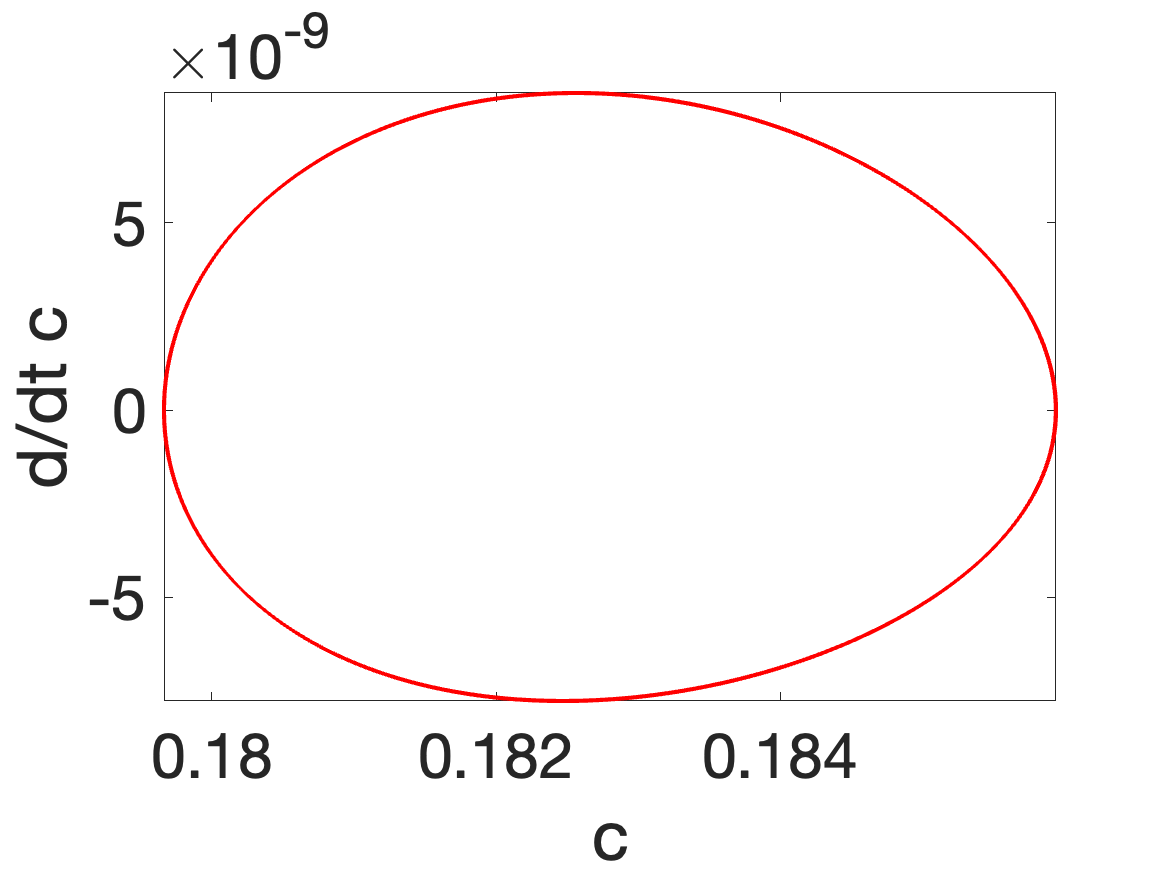}
&   \includegraphics[width=0.22\textwidth]{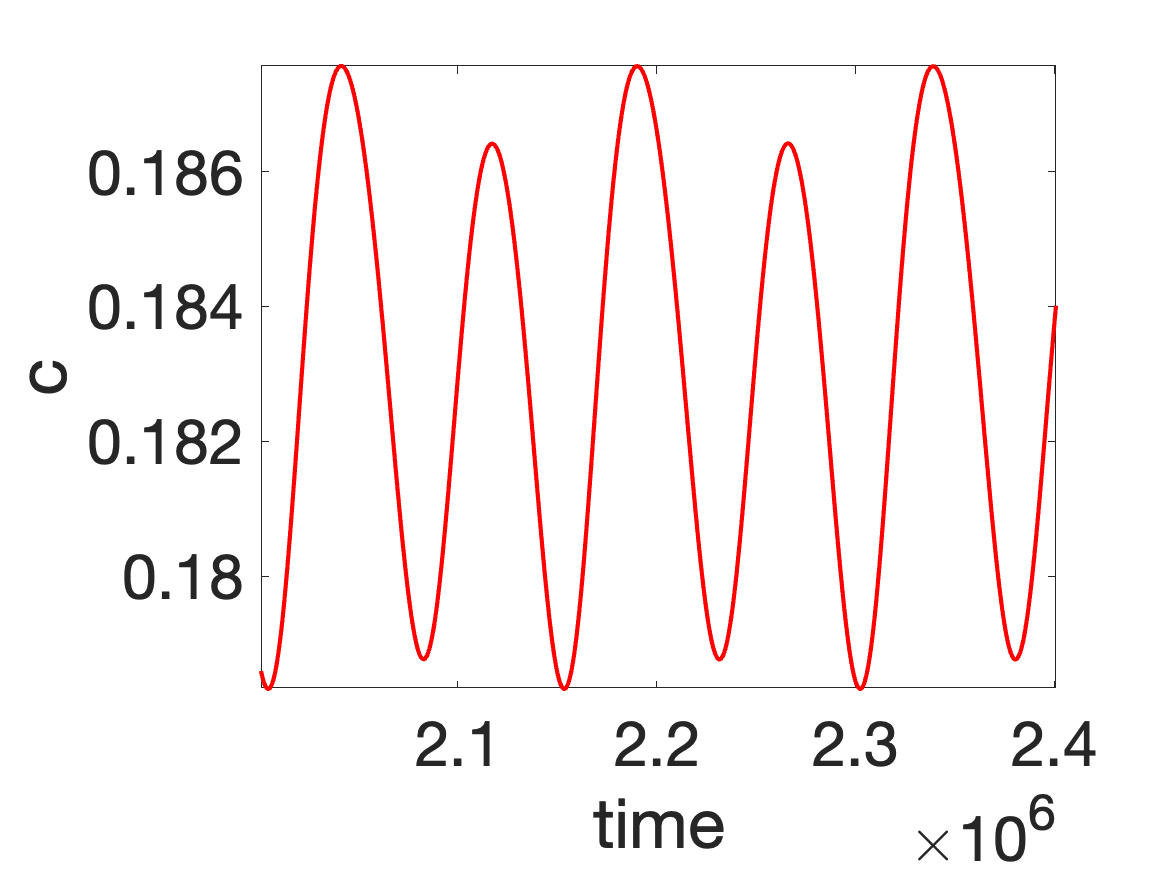} 
&   \includegraphics[width=0.22\textwidth]{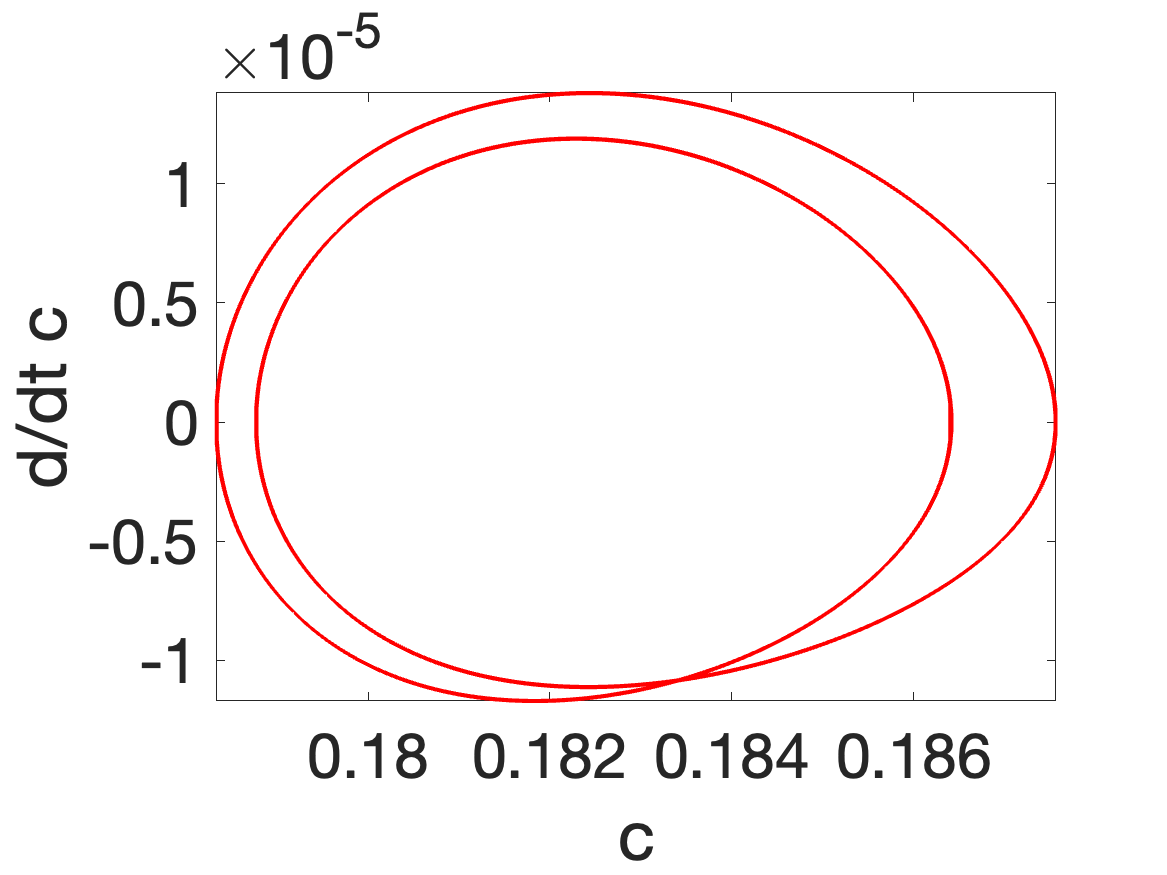}\\
(a) & (b) & (c) & (d)
\end{tabular}
   \caption{Time simulation after a transient from perturbation off unstable solution $\#2$ in panels (a,b) and solution $\#3$ in panels (c,d), cf.\ Figure~\ref{f:trans_branches}(b). 
   }
   \label{f:trans_branches_sim_per}
\end{figure}

In Figures~\ref{f:trans_branches_sim_per}, \ref{f:trans_branches_sim_chaos} we present simulation results from perturbing the unstable front solutions $\#2, 3$ and $4$ from Figure~\ref{f:trans_branches}, which lie in the 2D unstable region beyond the Hopf bifurcation. Closest to the Hopf bifurcation is solution $\#2$ and the simulation upon perturbing from it appears to converge to a period orbit, thus suggesting that the Hopf bifurcation is supercritical. Further along the branch, perturbing from $\#3$ yields a period  whose nature suggests a period doubling bifurcation along the corresponding branch of periodic solutions. Finally, perturbing from $\#4$ gives an apparently chaotic solution, which suggest a torus bifurcation. 
This scenario is a powerful indication for complex dynamics and for the link of the ODE and front dynamics, and is consistent with the unfolding of a fold-Hopf bifurcation as described in~\cite[Ch. 8.5]{KuznetsovBook}. It also bears similarities with the numerical results for parameters of ODE with Shil'nikov homoclinic orbits beyond reach of rigorous analysis in \cite{ACST85}.

\begin{figure}
\begin{tabular}{ccc}
   \includegraphics[width=0.3\textwidth]{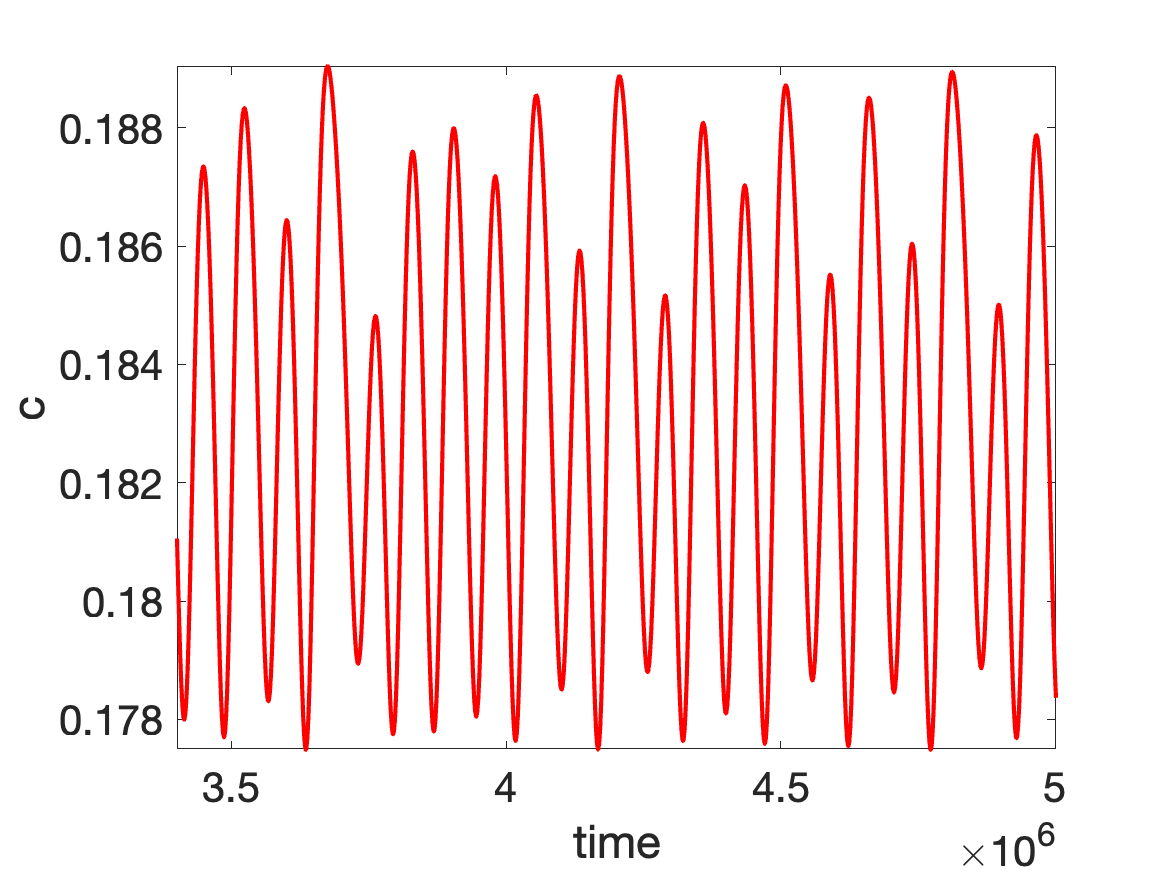} 
&   \includegraphics[width=0.3\textwidth]{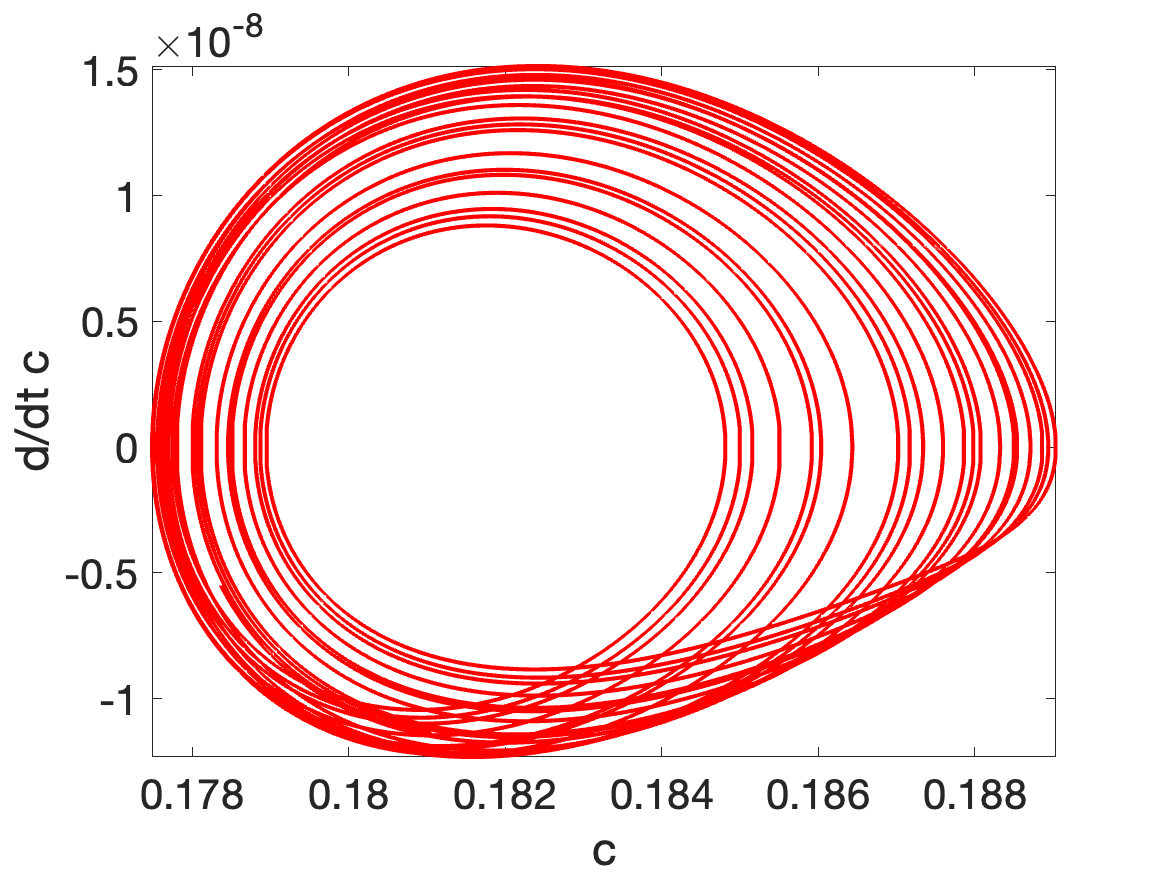}
&   \includegraphics[width=0.3\textwidth]{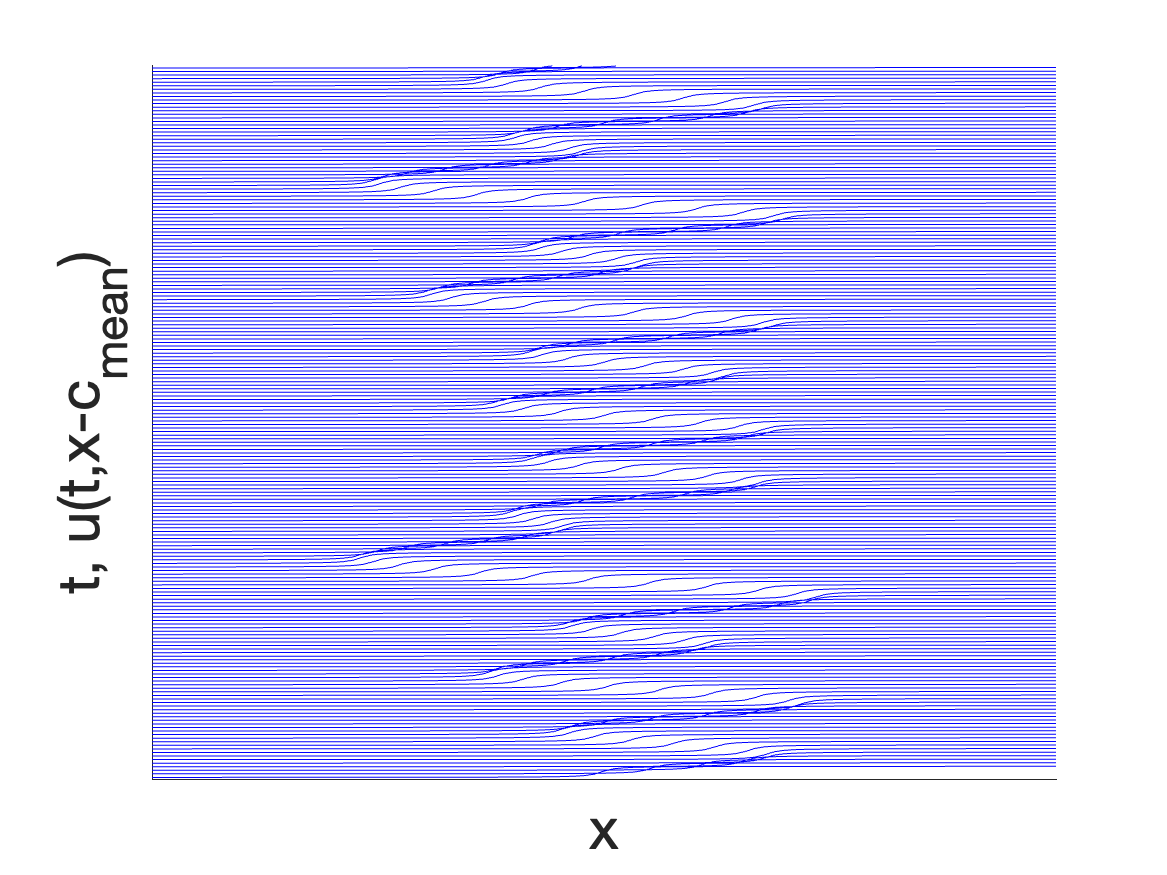}\\
(a) & (b) & (c)
\end{tabular}
   \caption{Time simulation after a transient from perturbation off unstable solution $\#4$ in Figure~\ref{f:trans_branches}(b) analogous to Figure~\ref{f:trans_branches_sim_per}. Panel (c) gives an impression of the corresponding front motion in a frame with the mean velocity over the time window.
   }
   \label{f:trans_branches_sim_chaos}
\end{figure}


\section{Discussion}
\label{S:DIS}
In the present work, we have examined the dynamics of front solutions of the multi-component reaction-diffusion system \eqref{eq:multi-component-RD} with one fast variable $U$ governed by an Allen-Cahn-type equation coupled to $N$ linear PDEs for the slow variables $V_j, j = 1, \ldots, N$. Here, we will first give a brief summary of our main results to then discuss possible further research directions in the outlook.

\subsection{Summary}
We began by observing that \eqref{eq:multi-component-RD} possesses a global attractor that, in particular, attracts all dynamic front solutions, see Lemma~\ref{lem:trap}. For the rest of the work, we concentrated on the dynamics of stationary, uniformly travelling and dynamic front solutions. We gave a range of intermediate results - each of interest in their own right - in order to demonstrate the complicated (even chaotic) behaviour that the time evolution of front speeds can exhibit.

Using spatial dynamics and GSPT, we converted the question of the existence of uniformly travelling front solutions with speed $\varepsilon^2 c $ into examining the roots of the ($\varepsilon$-leading order) existence function $\Gamma_0(c)= 0$, with $\Gamma_0$ as in \eqref{E:N}, see Lemma~\ref{L:N}. With this explicit algebraic equation at hand, we determined the highest possible degree of degeneracy by computing the highest multiplicity of the root $c = 0$ for a generic nonlinear coupling function $\calF^N(\vec{V}) =\calF^N_\beta(\vec{V})$, see Proposition~\ref{L:EX}. The structure of \eqref{eq:multi-component-RD} -- and, in particular, the imposed coupling through $\mathcal{F}^N(\vec{V})$ -- caused the underlying computations to involve Vandermonde matrices which are technically convenient, see Lemma~\ref{L:VDM}.
 
Linearisation around the so obtained uniformly travelling front solutions revealed that only certain point spectrum -- described by roots of the so-called Evans function -- can lead to instabilities. Hence, there are no essential instabilities in the parameter regime of interest, even though the essential spectrum approaches the imaginary axis in the limit $\varepsilon \rightarrow 0$, {and the eigenvalues can be viewed as emanating from it.} See Proposition~\ref{L:S}, Lemma~\ref{L:SE} and Lemma~\ref{L:SP}. 

In order to prepare the center manifold reduction describing dynamic front solutions that bifurcate from stationary ones, we then focused on the linear analysis around stationary front solutions. We determined the highest possible degeneracy of the Evans function $E_0$, see Proposition~\ref{L:N+1}, which is then used as organising center for unfolding bifurcations to dynamics front solutions. Before executing the reduction, we demonstrated that a full unfolding is possible on a linear level, see Lemma~\ref{l:linearunfold}, and also discuss the number of critical eigenvalues to be expected in general, see Hypothesis~\ref{hyp:criticalroots} and Lemma~\ref{lem:evals}. For $N=1,2,3$ we gave the precise number, see Lemma~\ref{lem:criticalroots_N_leq_2}, Lemma~\ref{lem:criticalroots_N_eq_3}.

Decisive for the center manifold reduction is our general result that the geometric multiplicity of the $\lambda = 0$ eigenvalues is always one, leading to Jordan chains for algebraic multiplicities larger than one, see Lemma~\ref{L:N+1_J}. We closed this section by giving an approximation of the corresponding (generalised) eigenfunctions, see Lemma~\ref{L:EIGSF} and Lemma~\ref{L:GEIGSG}, which heavily rely on the results presented in Lemma~\ref{L:EIGv}. This sets the scene for the unfolding around a nilpotent organising center.

Next, an important intermezzo was given in \S\ref{S34}, where we determined an explicit parameter set of maximal degeneracy both for the existence function and Evans function in the case of $N=3$ and $\calF^N(\vec{V}) =\calF^N_\beta(\vec{V})$. This explicit parameter set directly paved the way for finding chaotic behaviour numerically in \S\ref{s:num}.

In \S\ref{S:DFS} we moved to dynamic front solutions and used various results from the previous sections to prove Main Result~\ref{thm:intro}. Using a combination of center manifold reduction and bifurcation theory we derived an $\Np$-dimensional system of ODEs, see~\eqref{SPEED:ODE} of Proposition~\ref{prop:reduced_system} (see also \eqref{eq:main_ODE} of \S\ref{S:intro}) governing the speed of the dynamic front. Based on the linear analysis around a stationary front solution, the linearisation of the right hand side is nilpotent, hence, allowing a bridge to research on chaotic dynamics unfolded in ODE systems with nilpotent linear parts.

In preparation of the main results, we briefly described the procedure of analysing the dynamics of the reduced ODE for the speed, which, in particular, involved a convenient use of the existence function and Evans function from previous sections and the Weierstrass Preparation Theorem.

All these findings paved the way for finally showing how complicated dynamics can be caused by (possibly a combination of) two different mechanisms:
\begin{itemize}
    \item The coupling function $\mathcal{F}^N(\vec{V})$ is directly linked to the existence function function $\Gamma_0$, which, in turn, determines the singularity structure of the reduced ODE for the front speed \eqref{SPEED:ODE} (all in the sense of Taylor expansions). This gives a means of generating arbitrary singularity structures for the front motion, see Theorem~\ref{thm:one_slow_sing} (and the first part of Main Result~\ref{thm:intro}).
    \item The number of slow components $N \in \mathbb{N}$ can be used to control the dimension of the reduced ODE for the front speed \eqref{SPEED:ODE}, which can enable or inhibit complicated dynamics. A paradigm example of this is the occurrence of a Shil'nikov bifurcation, see Lemma~\ref{lem:chaos}, which is then confirmed to appear in \eqref{eq:multi-component-RD} for $N=3$ and a specific coupling function $\calF^N(\vec{V})$ in Theorem~\ref{thm:chaos} (see also the second part of Main Result~\ref{thm:intro}).
\end{itemize}
The type of chaotic motion that we encounter here is hard to verify numerically. However, in \S\ref{s:num} we numerically observed  some signature bifurcations of the dynamics around a zero-Hopf bifurcation and the occurrence of Shil'nikov homoclinics (as, {\it{e.g.}}, in \cite{ACST85}).

\subsection{Outlook} 
We view the results presented in this paper as starting point for embarking on a quest to further study chaotic dynamics of localised structures in the context of PDEs of class \eqref{eq:multi-component-RD} and generalisations thereof. As alluded to, we have chosen but two specific scenarios to illustrate the novel insights that the procedure introduced in this work allows. 

An immediate extension would be to keep the structure of our multi-component system~\eqref{eq:multi-component-RD} and use the framework presented here to embed (via $\mathcal{F}^N{(\vec{V}})$) other singularity structures (see \cite[e.g.]{WP74}) and prove chaotic dynamics by associating, for instance, results on chaos around Takens-Bogdanov points (as treated in \cite{AGS89, Ma99, RK17}). Furthermore, we chose $\Np = 3$ here for conciseness and a proof of concept. There are also more general results on chaotic dynamics available for nilpotent singularities of order $\Np > 3$,  see~\cite[e.g.]{BIR11, DIR07, DIR08}. Further connections to these could be explored.  In particular, in the case of even $\Np$ there is a perturbed Hamiltonian structure in the speed ODE at onset, which could be an interesting direction to pursue (see \cite{BIR11}, Theorem 3.3, for details of this aspect). Another direction that we left unexplored is to introduce new terms in \eqref{eq:multi-component-RD} to get other kinds of symmetry-breaking as, {\it{e.g.}}, described in Remark~\ref{ADV}.

Moving away from the nilpotent singularities with full Jordan chains, one could also attempt to change the system of PDEs such that the geometric multiplicity of the zero eigenvalue is two or larger, from which one would expect chaotic motion as in the Lorenz system (see the original work \cite{Lor63} and, {\it{e.g.},} \cite{OT17} for a normal form implying the existence of a Lorenz attractor). This would be particularly interesting since the Lorenz system features various routes to observable chaotic motion, including a homoclinic explosion and period-doubling.

One interpretation of \eqref{eq:multi-component-RD} is  an Allen-Cahn-type PDE for the fast variable $U$ that is ``forced" by a symmetry breaking term $\gamma(x,t)$. Recall that for constant $\gamma$, {\it i.e.} \eqref{AC},  fronts move with a speed proportional to $\gamma$. One could attempt to develop this further by using spatially periodic {$\gamma$ analogous to} \cite{NTYU07} or chaotic forcing as, {\emph{e.g.}}, done in \cite{ASCFK20} for the bistable Swift-Hohenberg equation with multiplicative forcing by the Kuramoto-Sivashinky equation. 
While one expects that arbitrary front motion over finite time scales can be created {by suitable $\gamma(t)$}, it is not clear that long time properties can be realised in this way.

Similar to the already mentioned related results on interacting localised structures which lead to chaos, one could attempt to study dynamic pulse, or $2$-front, solutions in \eqref{eq:multi-component-RD} (obtained from a combination of a front and a back solution as described in \cite{DvHK09,NS21, vHDK08, vHDKP10} for the three-component model \eqref{1F2S}). Since the number of critical eigenvalues (potentially) doubles by doubling the number of interfaces, one anticipates from \cite{vHDKP10} that, in certain parameter regions, the dimension and structure of the reduced ODEs for the interface positions is such that chaotic motion can occur.

Another interesting research direction to pursue is the influence of heterogeneity of different kinds on the dynamics of front solutions. As already mentioned, numerical discretisations where found in \cite{VD13} to cause chaotic motion for pulse solutions in a system of reaction-diffusion systems. Since the corresponding reduced ODE for the pulse motion in this case was planar, one can anticipate that the numerical discretisation causes the translational eigenvalue to move through the imaginary axis, thereby increasing the dimension of the reduced ODE such that it can potentially support chaotic behaviour. A similar effect is to be expected upon adding spatially varying coefficients to \eqref{eq:multi-component-RD} in a similar fashion as in \cite{bastiaansen2020pulse}. This also removes the translation invariance and, in particular, could lead to a three-dimensional reduced ODE for the front motion, even for $N = 2$. A case of the latter was formally analyzed in \cite{NTYU07}, which led to numerical evidence for chaos.

\section*{Acknowledgements}
The authors thank D. Turaev and I. Ovsyannikov for fruitful discussions. They also thank V. Ramirez Garcia Luna for finding the closed formula for the recurrence relations~\eqref{a1} and  \eqref{ainew} needed to prove Lemma~\ref{L:EIGv}. PvH and JR acknowledge support from the {\emph{Alexander von Humboldt Foundation}} under the 
{\emph{Research Fellowship for Experienced Researchers}} program. PvH thanks the University of Bremen for their hospitality.
JR notes this paper is a contribution to projects M2 and M7 of the Collaborative Research Centre TRR 181 `Energy Transfer in Atmosphere and Ocean' funded by the Deutsche Forschungsgemeinschaft (DFG, German Research Foundation) -- Project Number 274762653. Part of this research has been carried out during the Lorentz Center Workshop "Coherent Structures: Current Developments and Future Challenges".

\bibliographystyle{siam}
\bibliography{Chaotic_Fronts_arXiv}

\begin{thebibliography}{10}

\bibitem{AGJ90}
{\sc J.~Alexander, R.~A. Gardner, and C.~K. R.~T. Jones}, {\em A topological
  invariant arising in the stability analysis of travelling waves}, J. Reine
  Angew. Math., 410 (1990), pp.~167--212.

\bibitem{Alfaro2008}
{\sc M.~Alfaro, D.~Hilhorst, and H.~Matano}, {\em The singular limit of the
  {A}llen-{C}ahn equation and the {F}itz{H}ugh-{N}agumo system}, J. Differ.
  Equations, 245 (2008), pp.~505--565.

\bibitem{ASCFK20}
{\sc A.~J. Alvarez-Socorro, M.~G. Clerc, M.~Ferr\'e, and E.~Knobloch}, {\em
  Chaotic motion of localized structures}, Phys. Rev. E, 101 (2020), p.~042212.

\bibitem{AGS89}
{\sc D.~Armbruster, J.~Guckenheimer, and S.~Kim}, {\em Chaotic dynamics in
  systems with square symmetry}, Phys. Lett. A, 140 (1989), pp.~416--420.

\bibitem{ACST85}
{\sc A.~Arneodo, P.~H. Coullet, E.~A. Spiegel, and C.~Tresser}, {\em Asymptotic
  chaos}, Physica D, 14 (1985), pp.~327--347.

\bibitem{ACT82}
{\sc A.~Arneodo, P.~H. Coullet, and C.~Tresser}, {\em Oscillators with chaotic
  behavior: A illustration of a theorem by {S}hil'nikov}, J. Stat. Phys., 27
  (1982), pp.~171--182.

\bibitem{barkley2015rise}
{\sc D.~Barkley, B.~Song, V.~Mukund, G.~Lemoult, M.~Avila, and B.~Hof}, {\em
  The rise of fully turbulent flow}, Nature, 526 (2015), pp.~550--553.

\bibitem{BarrientosBook}
{\sc P.~G. Barrientos, S.~Ib\'{a}\~{n}ez, A.~A. Rodrigues, and J.~A.
  Rodr\'{\i}guez}, {\em Emergence of chaotic dynamics from singularities}, 32
  $^{\rm o}$ Col\'{o}quio Brasileiro de Matem\'{a}tica, Instituto Nacional de
  Matem\'{a}tica Pura e Aplicada (IMPA), Rio de Janeiro, 2019.

\bibitem{BIR11}
{\sc P.~G. Barrientos, S.~Ib\'{a}\~{n}ez, and J.~A. Rodr\'{\i}guez}, {\em
  Heteroclinic cycles arising in generic unfoldings of nilpotent
  singularities}, J. Dyn. Differ. Equ., 23 (2011), pp.~999--1028.

\bibitem{bastiaansen2020pulse}
{\sc R.~Bastiaansen, M.~Chirilus-Bruckner, and A.~Doelman}, {\em Pulse
  solutions for an extended {K}lausmeier model with spatially varying
  coefficients}, SIAM J. Appl. Dyn. Syst., 19 (2020), pp.~1--57.

\bibitem{P02}
{\sc M.~Bode, A.~W. Liehr, C.~P. Schenk, and H.-G. Purwins}, {\em Interaction
  of dissipative solitons: particle-like behaviour of localized structures in a
  three-component reaction-diffusion system}, Physica D, 161 (2002),
  pp.~45--66.

\bibitem{brown2023analysing}
{\sc C.~Brown, G.~Derks, P.~van Heijster, and D.~J.~B. Lloyd}, {\em Analysing
  transitions from a {T}uring instability to large periodic patterns in a
  reaction-diffusion system}, Nonlinearity, 36 (2023), p.~6839.

\bibitem{carter2016stability}
{\sc P.~Carter, B.~de~Rijk, and B.~Sandstede}, {\em Stability of traveling
  pulses with oscillatory tails in the {F}itz{H}ugh--{N}agumo system}, J. Nonl.
  Science, 26 (2016), pp.~1369--1444.

\bibitem{CBDvHR15}
{\sc M.~Chirilus-Bruckner, A.~Doelman, P.~van Heijster, and J.~D.~M.
  Rademacher}, {\em Butterfly catastrophe for fronts in a three-component
  reaction--diffusion system}, J. Nonl. Science, 25 (2015), pp.~87--129.

\bibitem{CBvHHR19}
{\sc M.~Chirilus-Bruckner, P.~van Heijster, H.~Ikeda, and J.~D. Rademacher},
  {\em Unfolding symmetric {B}ogdanov--{T}akens bifurcations for front dynamics
  in a reaction--diffusion system}, J. Nonl. Science, 29 (2019),
  pp.~2911--2953.

\bibitem{DKT87}
{\sc A.~E. Deane, E.~Knobloch, and J.~Toomre}, {\em Traveling waves and chaos
  in thermosolutal convection}, Phys. Rev. A, 36 (1987), pp.~2862--2869.

\bibitem{D98}
{\sc A.~Doelman, R.~A. Gardner, and T.~J. Kaper}, {\em Stability analysis of
  singular patterns in the 1-{D} {G}ray-{S}cott model: a matched asymptotics
  approach}, Physica D, 122 (1998), pp.~1--36.

\bibitem{D01}
\leavevmode\vrule height 2pt depth -1.6pt width 23pt, {\em Large stable pulse
  solutions in reaction-diffusion equations}, Indiana Univ. Math. J.,  (2001),
  pp.~443--507.

\bibitem{D02}
\leavevmode\vrule height 2pt depth -1.6pt width 23pt, {\em A stability index
  analysis of 1-{D} patterns of the {G}ray-{S}cott model}, American
  Mathematical Society, 2002.

\bibitem{DSZ15}
{\sc A.~Doelman, L.~Sewalt, and A.~Zagaris}, {\em The effect of slow spatial
  processes on emerging spatiotemporal patterns}, Chaos, 25 (2015), p.~036408.

\bibitem{DvHK09}
{\sc A.~Doelman, P.~van Heijster, and T.~J. Kaper}, {\em Pulse dynamics in a
  three-component system: existence analysis}, J. Dyn. Differ. Equ., 21 (2009),
  pp.~73--115.

\bibitem{DIR07}
{\sc F.~Drubi, S.~Ibanez, and J.~A. Rodriguez}, {\em Coupling leads to chaos},
  J. Differ. Equations, 239 (2007), pp.~371--385.

\bibitem{DIR08}
{\sc F.~Drubi, S.~Ib{\'a}nez, and J.~{\'A}. Rodr{\'i}guez}, {\em {Singularities
  and chaos in coupled systems}}, B. Belg. Math. Soc.-Sim., 15 (2008), pp.~797
  -- 808.

\bibitem{DIK01}
{\sc F.~Dumortier, S.~Ib{\'a}{\~n}ez, and H.~Kokubu}, {\em New aspects in the
  unfolding of the nilpotent singularity of codimension three}, Dyn. Syst. Int.
  J., 16 (2001), pp.~63--95.

\bibitem{engel2021traveling}
{\sc M.~Engel, C.~Kuehn, and B.~de~Rijk}, {\em A traveling wave bifurcation
  analysis of turbulent pipe flow}, Nonlinearity, 35 (2022), p.~5903.

\bibitem{F79}
{\sc N.~Fenichel}, {\em Geometric singular perturbation theory for ordinary
  differential equations}, J. Differ. Equations, 31 (1979), pp.~53--98.

\bibitem{Garcke1998}
{\sc H.~Garcke, B.~Nestler, and B.~Stoth}, {\em On anisotropic order parameter
  models for multi-phase systems and their sharp interface limits}, Physica D,
  115 (1998), pp.~87--108.

\bibitem{GIUNTA}
{\sc V.~Giunta, M.~C. Lombardo, and M.~Sammartino}, {\em Pattern formation and
  transition to chaos in a chemotaxis model of acute inflammation}, SIAM J.
  Appl. Dyn. Syst., 20 (2021), pp.~1844--1881.

\bibitem{GAP06}
{\sc S.~V. Gurevich, S.~Amiranashvili, and H.-G. Purwins}, {\em Breathing
  dissipative solitons in three-component reaction-diffusion system}, Phys.
  Rev. E, 74 (2006), pp.~066201, 7.

\bibitem{haragus2010local}
{\sc M.~Haragus and G.~Iooss}, {\em Local bifurcations, center manifolds, and
  normal forms in infinite-dimensional dynamical systems}, Springer Science \&
  Business Media, 2010.

\bibitem{harville1998matrix}
{\sc D.~A. Harville}, {\em Matrix algebra from a statistician's perspective},
  Taylor \& Francis, 1998.

\bibitem{henry2006geometric}
{\sc D.~Henry}, {\em Geometric theory of semilinear parabolic equations},
  Springer, 1981.

\bibitem{HodgkinHuxley}
{\sc A.~L. Hodgkin and A.~F. Huxley}, {\em A quantitative description of
  membrane current and its application to conduction and excitation in nerve},
  J. Physiol., 117 (1952), pp.~500--544.

\bibitem{HS10}
{\sc A.~J. Homburg and B.~Sandstede}, {\em Chapter 8 - Homoclinic and
  Heteroclinic Bifurcations in Vector Fields}, vol.~3 of Handbook of Dynamical
  Systems, Elsevier Science, 2010.

\bibitem{HKTP10}
{\sc S.~Houghton, E.~Knobloch, S.~Tobias, and M.~Proctor}, {\em Transient
  spatio-temporal chaos in the complex {G}inzburg-{L}andau equation on long
  domains}, Phys. Lett. A, 374 (2010), pp.~2030--2034.

\bibitem{IR05}
{\sc S.~Ib{\'a}{\~n}ez and J.~Rodr{\'i}guez}, {\em Shil'nikov configurations in
  any generic unfolding of the nilpotent singularity of codimension three on
  $\mathbb{R}^3$}, J. Differ. Equations, 208 (2005), pp.~147--175.

\bibitem{J95}
{\sc C.~K. R.~T. Jones}, {\em Geometric singular perturbation theory}, in
  Dynamical systems ({M}ontecatini {T}erme, 1994), vol.~1609 of Lecture Notes
  in Math., Springer, Berlin, 1995, pp.~44--118.

\bibitem{K99}
{\sc T.~J. Kaper}, {\em An introduction to geometric methods and dynamical
  systems theory for singular perturbation problems}, in Analyzing multiscale
  phenomena using singular perturbation methods ({B}altimore, {MD}, 1998),
  vol.~56 of Proc. Sympos. Appl. Math., Amer. Math. Soc., Providence, RI, 1999,
  pp.~85--131.

\bibitem{KLM99}
{\sc E.~Knobloch, A.~S. Landsberg, and J.~Moehlis}, {\em Chaotic
  direction-reversing waves}, Phys. Lett. A, 255 (1999), pp.~287--293.

\bibitem{KMTW86}
{\sc E.~Knobloch, D.~R. Moore, J.~Toomre, and N.~O. Weiss}, {\em Transitions to
  chaos in two-dimensional double-diffusive convection}, J. Fluid Mech., 166
  (1986), pp.~409--448.

\bibitem{KW87}
{\sc E.~Knobloch and J.~B. Weiss}, {\em Chaotic advection by modulated
  traveling waves}, Phys. Rev. A, 36 (1987), pp.~1522--1524.

\bibitem{KuznetsovBook}
{\sc Y.~A. Kuznetsov}, {\em Elements of applied bifurcation theory}, vol.~112
  of Applied Mathematical Sciences, Springer-Verlag, New York, third~ed., 2004.

\bibitem{Lor63}
{\sc E.~N. Lorenz}, {\em Deterministic nonperiodic flow}, Journal of
  Atmospheric Sciences, 20 (1963), pp.~130 -- 141.

\bibitem{MS58}
{\sc N.~Macon and A.~Spitzbart}, {\em Inverses of {V}andermonde matrices}, Am.
  Math. Mon., 65 (1958), pp.~95--100.

\bibitem{Ma99}
{\sc K.~Matthies}, {\em A subshift of finite type in the {T}akens--{B}ogdanov
  bifurcation with {D}3 symmetry}, Doc. Math., 4 (1999), pp.~463--485.

\bibitem{MTKW83}
{\sc D.~Moore, J.~Toomre, E.~Knobloch, and N.~Weiss}, {\em Period doubling and
  chaos in partial differential equations for thermosolutal convection},
  Nature, 303 (1983), pp.~663--667.

\bibitem{MV93}
{\sc L.~Mora and M.~Viana}, {\em Abundance of strange attractors.}, Acta
  Mathematica, 171 (1993), pp.~1--71.

\bibitem{nishiura1987stability}
{\sc Y.~Nishiura and H.~Fujii}, {\em Stability of singularly perturbed
  solutions to systems of reaction-diffusion equations}, SIAM J. Math. Anal.,
  18 (1987), pp.~1726--1770.

\bibitem{nishiura1990singular}
{\sc Y.~Nishiura, M.~Mimura, H.~Ikeda, and H.~Fujii}, {\em Singular limit
  analysis of stability of traveling wave solutions in bistable
  reaction-diffusion systems}, SIAM J. Math. Anal., 21 (1990), pp.~85--122.

\bibitem{NS21}
{\sc Y.~Nishiura and H.~Suzuki}, {\em Matched asymptotic expansion approach to
  pulse dynamics for a three-component reaction-diffusion system}, J. Differ.
  Equations, 303 (2021), pp.~482--546.

\bibitem{NTU03}
{\sc Y.~Nishiura, T.~Teramoto, and K.-I. Ueda}, {\em Dynamic transitions
  through scattors in dissipative systems}, Chaos, 13 (2003), pp.~962--972.

\bibitem{NTYU07}
{\sc Y.~Nishiura, T.~Teramoto, X.~Yuan, and K.-I. Ueda}, {\em Dynamics of
  traveling pulses in heterogeneous media}, Chaos, 17 (2007), p.~037104.

\bibitem{NU01}
{\sc Y.~Nishiura and D.~Ueyama}, {\em Spatio-temporal chaos for the
  {G}ray--{S}cott model}, Physica D, 150 (2001), pp.~137--162.

\bibitem{P98}
{\sc M.~Or-Guil, M.~Bode, C.~P. Schenk, and H.-G. Purwins}, {\em Spot
  bifurcations in three-component reaction-diffusion systems: The onset of
  propagation}, Phys. Rev. E, 57 (1998), pp.~6432--6437.

\bibitem{OT17}
{\sc I.~I. Ovsyannikov and D.~V. Turaev}, {\em Analytic proof of the existence
  of the {L}orenz attractor in the extended {L}orenz model}, Nonlinearity, 30
  (2016), p.~115.

\bibitem{PUR14}
{\sc H.-G. Purwins and L.~Stollenwerk}, {\em Synergetic aspects of
  gas-discharge: lateral patterns in dc systems with a high ohmic barrier},
  Plasma Phys. Contr. F., 56 (2014), p.~123001.

\bibitem{RK17}
{\sc A.~M. Rucklidge and E.~Knobloch}, {\em Chaos in the takens-bogdanov
  bifurcation with {O}(2) symmetry}, Dynamical Systems, 32 (2017),
  pp.~354--373.

\bibitem{sandstede1998stability}
{\sc B.~Sandstede}, {\em Stability of {N}-fronts bifurcating from a twisted
  heteroclinic loop and an application to the {F}itz{H}ugh--{N}agumo equation},
  SIAM J. Math. Anal., 29 (1998), pp.~183--207.

\bibitem{sandstede2002stability}
\leavevmode\vrule height 2pt depth -1.6pt width 23pt, {\em Stability of
  travelling waves}, in Handbook of dynamical systems, vol.~2, Elsevier, 2002,
  pp.~983--1055.

\bibitem{P97}
{\sc C.~P. Schenk, M.~Or-Guil, M.~Bode, and H.-G. Purwins}, {\em Interacting
  pulses in three-component reaction-diffusion systems on two-dimensional
  domains}, Phys. Rev. Lett., 78 (1997), pp.~3781--3784.

\bibitem{SCHWEISGUTH2019659}
{\sc F.~Schweisguth and F.~Corson}, {\em Self-organization in pattern
  formation}, Developmental Cell, 49 (2019), pp.~659--677.

\bibitem{Shi65}
{\sc L.~P. Shilnikov}, {\em A case of the existence of a denumerable set of
  periodic motions.}, Sov. Math. Dokl., 6 (1965), pp.~163--166.

\bibitem{Smoller}
{\sc J.~Smoller}, {\em Shock waves and reaction-diffusion equations}, vol.~258
  of Grundlehren der mathematischen Wissenschaften [Fundamental Principles of
  Mathematical Sciences], Springer-Verlag, New York, second~ed., 1994.

\bibitem{Steinbach2023}
{\sc I.~Steinbach and H.~Salama}, {\em Multi-phase-field approach}, in Lectures
  on Phase Field, Springer Nature Switzerland, Cham, 2023, pp.~61--68.

\bibitem{teramoto2021traveling}
{\sc T.~Teramoto and P.~Van~Heijster}, {\em Traveling pulse solutions in a
  three-component {F}itz{H}ugh--{N}agumo model}, SIAM J. Appl. Dyn. Syst., 20
  (2021), pp.~371--402.

\bibitem{Tr81}
{\sc C.~Tresser}, {\em Modules Simples de Transitions vers la turbulence}, PhD
  thesis, 1981.

\bibitem{TZ10}
{\sc D.~Turaev and S.~Zelik}, {\em Analytical proof of space-time chaos in
  {G}inzburg-{L}andau equations}, Discret. Comtin. Dyn. S., 28 (2010),
  pp.~1713--1751.

\bibitem{p2pbook}
{\sc H.~Uecker}, {\em Numerical continuation and bifurcation in nonlinear
  {PDE}s}, Society for Industrial and Applied Mathematics (SIAM), Philadelphia,
  PA, 2021.

\bibitem{p2p}
{\sc H.~Uecker, D.~Wetzel, and J.~D.~M. Rademacher}, {\em pde2path - a {Matlab}
  package for continuation and bifurcation in {2D} elliptic systems}, Numer.
  Math. Theory Methods Appl., 7 (2014), pp.~58--106.

\bibitem{vHCNT16}
{\sc P.~van Heijster, C.-N. Chen, Y.~Nishiura, and T.~Teramoto}, {\em Localized
  patterns in a three-component {F}itz{H}ugh--{N}agumo model revisited via an
  action functional,}, J. Dyn. Differ. Equ.,  (2016).

\bibitem{van2019pinned}
{\sc P.~van Heijster, C.-N. Chen, Y.~Nishiura, and T.~Teramoto}, {\em Pinned
  solutions in a heterogeneous three-component {F}itz{H}ugh--{N}agumo model},
  J. Dyn. Differ. Equ., 31 (2019), pp.~153--203.

\bibitem{vHDK08}
{\sc P.~van Heijster, A.~Doelman, and T.~J. Kaper}, {\em Pulse dynamics in a
  three-component system: stability and bifurcations}, Physica D, 237 (2008),
  pp.~3335--3368.

\bibitem{vHDKNU11}
{\sc P.~van Heijster, A.~Doelman, T.~J. Kaper, Y.~Nishiura, and K.-I. Ueda},
  {\em Pinned fronts in heterogeneous media of jump type}, Nonlinearity, 24
  (2011), pp.~127--157.

\bibitem{vHDKP10}
{\sc P.~van Heijster, A.~Doelman, T.~J. Kaper, and K.~Promislow}, {\em Front
  interactions in a three-component system}, SIAM J. Appl. Dyn. Syst., 9
  (2010), pp.~292--332.

\bibitem{vHS11}
{\sc P.~van Heijster and B.~Sandstede}, {\em Planar radial spots in a
  three-component {F}itz{H}ugh-{N}agumo system}, J. Nonl. Science, 21 (2011),
  pp.~705--745.

\bibitem{vHS12}
\leavevmode\vrule height 2pt depth -1.6pt width 23pt, {\em Coexistence of
  stable spots and fronts in a three-component {F}itz{H}ugh-{N}agumo system},
  RIMS Kokyuroku Bessatsu, B31 (2012), pp.~135--155.

\bibitem{vHS14}
\leavevmode\vrule height 2pt depth -1.6pt width 23pt, {\em Bifurcations to
  travelling planar spots in a three-component {F}itz{H}ugh-{N}agumo system},
  Physica D, 275 (2014), pp.~19--34.

\bibitem{VE07}
{\sc V.~K. Vanag and I.~R. Epstein}, {\em Localized patterns in
  reaction-diffusion systems}, Chaos, 17 (2007), p.~037110.

\bibitem{VD13}
{\sc F.~Veerman and A.~Doelman}, {\em Pulses in a {G}ierer--{M}einhardt
  equation with a slow nonlinearity}, J. Appl. Dyn. Syst., 12 (2013),
  pp.~28--60.

\bibitem{We21}
{\sc M.~Westdickenberg}, {\em On the metastability of the 1-d {A}llen--{C}ahn
  equation.}, J. Dyn.Y Differ. Equ., 33 (2021), pp.~1853--1879.

\bibitem{WheelerMcFadden}
{\sc A.~A. Wheeler and G.~B. McFadden}, {\em On the notion of a $\xi$--vector
  and a stress tensor for a general class of anisotropic diffuse interface
  models}, P. Roy. Soc. Lond. A Mat., 453 (1997), pp.~1611--1630.

\bibitem{WIT}
{\sc V.~Witkovsk{\`y}}, {\em Exact distribution of positive linear combinations
  of inverted chi-square random variables with odd degrees of freedom}, Stat.
  Probab. Lett., 56 (2002), pp.~45--50.

\bibitem{WITTENBERG19971}
{\sc R.~W. Wittenberg and P.~Holmes}, {\em The limited effectiveness of normal
  forms: {A} critical review and extension of local bifurcation studies of the
  {B}russelator {P}{D}{E}}, Physica D, 100 (1997), pp.~1--40.

\bibitem{Woesler1996}
{\sc R.~Woesler, P.~Sch{\"u}tz, M.~Bode, M.~Or-Guil, and H.-G. Purwins}, {\em
  Oscillations of fronts and front pairs in two- and three-component
  reaction-diffusion systems}, Physica D, 91 (1996), pp.~376--405.

\bibitem{WP74}
{\sc A.~Woodcock and T.~Poston}, {\em A Geometrical Study of the Elementary
  Catastrophes}, Lecture Notes in Mathematics, Springer Berlin, Heidelberg,
  1974.

\bibitem{YTN07}
{\sc X.~Yuan, T.~Teramoto, and Y.~Nishiura}, {\em Heterogeneity-induced defect
  bifurcation and pulse dynamics for a three-component reaction-diffusion
  system}, Phys. Rev. E, 75 (2007), p.~036220.

\bibitem{MZ09}
{\sc S.~Zelik and A.~Mielke}, {\em Multi-pulse evolution and space-time chaos
  in dissipative systems}, Mem. Am. Math. Soc., 198 (2009), pp.~VI -- 95.

\end{thebibliography}

\begin{appendix}
\section{Proof of Lemma~\ref{lem:criticalroots_N_eq_3}.}\label{app:numberroots}
We study the Evans function $E_0$ corresponding to a stationary front solution $Z_{\rm SF}$ for $N=3$, see~\eqref{EV0}, for $\alpha_j$ near those from Proposition~\ref{L:N+1}.
To ease notation, we write $E_0$ as
\[
E_0(\lambda) = \lambda + \sum_{j=1}^3 a_j \left( \frac{1}{\sqrt{\lambda \tau_j +1}}-1 \right)\;, \quad
a_j:=\frac{3\sqrt{2}}{2}\frac{\alpha_j}{d_j}. 
\]
Since the values of $\tau_j$ are assumed to be mutually distinct for a stationary front solution (cf.\ Proposition~\ref{L:N+1}), we can order them as $0<\tau_1<\tau_2<\tau_3$, which means that the singularities of $E_0$ satisfy $-1/\tau_1< -1/\tau_2< -1/\tau_3<0$.
The values of $a_j$ from Proposition~\ref{L:N+1} read 
\begin{align}\label{e:app4fold}
a_1 = \frac{2}{\tau_1}\frac{\tau_2}{\tau_2-\tau_1}\frac{\tau_3}{\tau_3-\tau_1}\;,\quad
a_2 = \frac{2}{\tau_2} \frac{\tau_1}{\tau_1-\tau_2}\frac{\tau_3}{\tau_3-\tau_2}\;,\quad
a_3 = \frac{2}{\tau_3}\frac{\tau_1}{\tau_1-\tau_3}\frac{\tau_2}{\tau_2-\tau_3}, 
\end{align}
so that by the above ordering we have 
\begin{align*}
a_1>0\;,\quad a_2<0\;,\quad a_3>0.
\end{align*}

We next perform a homotopy along $\tau_3$ from its given initial value $\tau_3=\tau_3^0$ to $\tau_3=\tau_2$. At the endpoint of this homotopy $E_0$ is equivalent to the Evans function for $N=2$, for which we know from \cite{CBDvHR15} that it possesses at most three roots. It therefore suffices to show that the number of root decreases by one during the homotopy. 

Along the homotopy we choose $a_j$ as in \eqref{e:app4fold} except the values of $a_2$ and $a_3$ are slightly perturbed to ensure they remain bounded. In particular, in the formulas for $a_2, a_3$ from~\eqref{e:app4fold} (but not elsewhere) we replace $\tau_2$ by $\tau_2+g(\tau_3)$, where $g$ is continuous and monotone with $g(\tau_3^0)=0$ and $g(\tau_2)\in(0,\tau_3-\tau_2)$, which will be chosen sufficiently small in the following. Notably, this does not change the locations of the singularities of $E_0$.

More generally, for any $a_j$ we can bound the location of the roots of $E_0$. 
We estimate 
\[
|E_0(\lambda)| \geq |\lambda| - \sum_{j=1}^3 | a_j| \left(\left|\frac{1}{\sqrt{\tau_j \lambda+1}}\right| + 1\right) 
\geq |\lambda| -2 \sum_{j=1}^3|a_j|,
\]
where the last inequality holds for $1\leq |\tau_j\lambda+1|$, $j=1,2,3$. 
Since $|\tau_j\lambda+1|\geq |\tau_j \lambda|-1$ the latter holds for $|\lambda|\geq 1/\tau_1$ and we infer that there are no roots with
\[
|\lambda| > 2 \max\left\{\sum_{j=1}^3 |a_j|, \;\frac{1}{\tau_1}
\right\}.
\]
Therefore, during the homotopy roots remain bounded and thus the number of roots can only change due to roots on the branch cut $(-\infty, -1/\tau_3]$.  Specifically, since there can be no root near the fixed singularity  $-1/\tau_1$ there are four cases to consider for roots $\lambda$ of $E_0$: 
\begin{tabular}{rlrl}
(i): & $\lambda \in (-\infty, -1/\tau_1)$, &
(ii): & $\lambda \in (-1/\tau_1,-1/\tau_2)$, \\
(iii): & $\lambda \in (-1/\tau_2,-1/\tau_3)$, & and 
(iv):& $\lambda\to -1/\tau_2$ as $\tau_3\to \tau_2$. 
\end{tabular}

\medskip
\textbf{Case (i):}  $\lambda<-1/\tau_1$ means for a root that $\Re(E(\lambda)) = \lambda - a_1-a_2-a_3=0$, {\emph{i.e.}}, $\lambda = a_1+a_2+a_3$. But $\lambda<-1/\tau_1$ requires $a_1+a_2+a_3< -1/\tau_1$ and from \eqref{e:app4fold} for $g=0$ we have $$a_1+a_2+a_3 = \frac{2}{\tau_1}+\frac{2}{\tau_2}+\frac{2}{\tau_3}>0>-\frac{1}{\tau_1}.$$ Hence, there is no such root for $g>0$ sufficiently small and case (i) does not occur \textbf{Case (ii):}
$\lambda\in(-1/\tau_1,-1/\tau_2)$ means for a root that 
$\Im(E(\lambda)) = \dfrac{a_1}{\sqrt{-\tau_1\lambda-1}} + \dfrac{a_2}{\sqrt{-\tau_2\lambda-1}} = 0$, 
which means $a_1a_2<0$ and $a_1^2(\tau_2\lambda+1) = a_2^2(\tau_1\lambda+1)$ from which we obtain 
$\lambda = \lambda_2:= \dfrac{a_2^2-a_1^2}{a_1^2\tau_2-a_2^2\tau_1}$. Now $\lambda_2\in(-1/\tau_1,-1/\tau_2)$ requires $a_1\neq a_2$ and we compute that 
\[
\mathrm{sgn}\; \partial_{\tau_1} \lambda_2 = \mathrm{sgn} (a_1^2-a_2^2)\;, \quad
\mathrm{sgn}\; \partial_{\tau_2} \lambda_2 = \mathrm{sgn} (a_2^2-a_1^2). 
\]
At $\tau_1=\tau_2$ we have $\lambda_2 = -1/\tau_1 =  -1/\tau_2$ so that the signs of the derivatives mean for $\tau_1<\tau_2$: 
\begin{itemize}
\item by increasing $\tau_2$ we have $\lambda_2< -1/\tau_1$ if $a_1^2>a_2^2$, and
\item by decreasing $\tau_1$ we have $\lambda_2> -1/\tau_2$ if $a_2^2>a_1^2$.
\end{itemize}
In both cases $\lambda_2\notin(-1/\tau_1,-1/\tau_2)$, which means case (ii) does not occur for a root. \textbf{Case (iii):}  $\lambda\in(-1/\tau_2,-1/\tau_3)$ means that 
$\Im(E(\lambda)) = \dfrac{a_3}{\sqrt{-\tau_3\lambda-1}}$, which is non-zero for $\lambda\neq 0$. Hence, case (iii) cannot occur for a root. \textbf{Case (iv):} It suffices to study whether $\lambda\approx -1/\tau_3$ for a root when $\tau_3 = \tau_2+\delta$ with homotopy parameter $\delta\approx 0$. Writing $\lambda= \nu -1/\tau_2$ with $\nu\in\mathbb{C}$ yields 
\begin{align*}
E_0 =& \nu -\frac{1}{\tau_2} 
+ a_1\left(\frac{1}{\sqrt{\tau_1\nu-\dfrac{\tau_1}{\tau_2}+1}}-1\right) 
\\&+ \frac{1}{\sqrt{\tau_2\nu}}\left( 
a_2(1-\sqrt{\tau_2\lambda})
+ a_3\left(\frac{1}{\sqrt{1+\dfrac{\delta}{\tau_2\nu}(\nu-\dfrac{1}{\tau_2})}}-\sqrt{\tau_2\nu}\right)\right),
\end{align*}
and multiplication with $\sqrt{\tau_2\nu}$ gives that its roots are those of 
\begin{align*}
F:= &
\sqrt{\tau_2\nu}\left(\nu -\frac{1}{\tau_2} 
+ a_1\left(\frac{1}{\sqrt{\tau_1\nu-\dfrac{\tau_1}{\tau_2}+1}}-1\right)  
- (a_2+a_3)\right)
\\&+ a_2
+ \frac{a_3}{\sqrt{1+\dfrac{\delta}{\tau_2\nu}(\nu-\dfrac{1}{\tau_2})}}.
\end{align*}
Since $a_2, a_3\neq0$ it follows that roots near $\nu=0$ can only occur for bounded $\delta/\nu$ as $\delta\to 0$.   
Defining $\rho := \delta/(\nu\tau_2)$ we find that $F=0$ at $\nu=0$ requires $a_2+ a_2(1-\rho/\tau_2)^{-1/2}=0$.  Since $a_2a_3<0$ we infer that the occurrence of roots requires $\rho\to\rho^*:= \tau_2(1-(a_3/a_2)^2)$ as $\delta\to 0$. From \eqref{e:app4fold} (i.e. for $g=0$)
and the ordering of $\tau_j$ we have 
\[
\frac{a_3}{a_2} = -\frac{\tau_1-\tau_2}{\tau_1-\tau_3}\frac{\tau_2^2}{\tau_3^2} \in(-1,0),
\]
which implies $\rho^*>0$ for sufficiently small $g>0$. 

With $\rho$ substituting $\delta/(\nu\tau_2)$ we have that $F$ is a smooth function of $w:=\sqrt{\nu}$ and $\rho$ for $\nu\approx0$. On the one hand, we compute for $g=0$ that
\[
\partial_w F|_{w=0,\rho=\rho^*} = \sqrt{\tau_2}A, \;\; {\rm{with}} \;\; 
A= -\frac{1}{\tau_2}+a_1\sqrt{\frac{\tau_2}{\tau_2-\tau_1}} - (a_1+a_2+a_3)\in \R.
\]
For generic modification of $\tau_2$ in $a_2, a_3$ by $g>0$ we thus have real $\partial_w F|_{w=0,\rho=\rho^*}\neq 0$. 

On the other hand, we directly compute, since $a_2<0$, that 
\[
\partial_\rho F|_{w=0,\rho=\rho^*} = \frac{a_3}{2\tau_2}\left(\frac{a_3}{a_2}\right)^{-3} 
= \frac{a_2^3}{2\tau_2 a_3^2}<0.
\]
Hence, there is a unique real root $w(\rho)$ of $F$ and thus a unique $\nu(\rho)>0$ for $\rho\approx 0$, and this also holds for sufficiently small $g>0$. 
This means that at the endpoint of the homotopy, where $\tau_2=\tau_3$ the number of roots of $E_0$ has decreased by one, which proves the claim.$\hfill \square$
 \section{Proof of Lemma~\ref{L:EIGSF}}
\label{A:E}
Before we display the proof we need to state a few more details on the uniformly travelling front solution $Z_{\rm TF}$ which can be proved in the same way as in \cite{CBDvHR15}. As before, we therefore omit the proof of these statements.
\begin{lemma}[More details on front profiles]
Let the conditions of Lemma~\ref{L:N} be fulfilled. Then there is an expansion of the uniformly travelling front solution $Z_{\rm TF}= 
Z_{\rm TF}^0 + \eps Z_{\rm TF}^1 + \eps^2 Z_{\rm TF}^2 + \mathcal{O}(\eps^3)= 
(U_{\rm TF}^0,\vec{V}^{0}_{\rm TF}) + 
\eps (U_{\rm TF}^1,\vec{V}^{1}_{\rm TF}) + \eps^2
(U_{\rm TF}^2,\vec{V}^{2}_{\rm TF}) 
+ \mathcal{O}(\eps^3)$
 such that in $I_f$, see Definition~\ref{fields}, and with $z := y/\eps$ we have
\begin{enumerate}
\item $U_{\rm TF}^1$ is an even function and obeys
\begin{align*}
L^0 \left(\partial_z U_{\rm TF}^1\right) = -\frac12 \sqrt2 c \,\partial_z \rho(z) +3\sqrt2 U_{\rm TF}^0 U_{\rm TF}^1 \rho(z)\,,
\end{align*}
where
$
L^0$
is given in Definition~\ref{LL}
and
$
\rho(z) := \sech^2{(z/\sqrt2)} \,.
$
\item $U_{\rm TF}^2$ obeys
\begin{align} 
\label{UTF2}
\begin{aligned}
&c \int \partial_z^2(U_{\rm TF}^1) \rho \, dz - \dfrac{3\sqrt2}{2} \int (U_{\rm TF}^1)^2 \rho^2 \,dz - 6 \int U_{\rm TF}^0 U_{\rm TF}^1 \partial_z(U_{\rm TF}^1) \rho \,dz \\ &- 3 \sqrt2 \int U_{\rm TF}^0 U_{\rm TF}^2  \rho^2 \, dz
= 4\sqrt2 \left( \sum_{j=1}^N \dfrac{ \alpha_j}{\sqrt{4 d_j^2 + c^2 \tau_j^2}}
+ \dfrac{\partial_j \mathcal{F}^N_{\rm nl}(\vec{V}^*)}{\sqrt{4 d_j^2 + c^2 \tau_j^2}}  \right)\,,
\end{aligned}
\end{align}
where we used the short hand notation
\begin{align*}
\partial_j \mathcal{F}^N_{\rm nl}(\vec{V}^*) :=\left.\frac{\partial \mathcal{F}_{\rm nl}^N(\vec{V})}{\partial V_j}\right|_{\vec{V} = \vec{V}^*} \,,
\end{align*}
and we recall that $\mathcal{F}_{\rm nl}^N(\vec{V})$ is given in \eqref{N:NONL22}.
\end{enumerate}
\end{lemma}
\begin{proof}[Proof of Lemma~\ref{L:EIGSF}]
For notational convenience we write $\Phi = (u,\vec{v})= (u,v_1,\ldots,v_N)$.
Then the eigenvalue problem for $c=0$ and for an $\O(\eps^2)$-eigenvalue with real part larger than $- \eps^2\chi$ (cf.\ Lemma~\ref{L:SE}) can be deduced from \eqref{EQ:1FNS_EVPW} and is given by 
\begin{align}\label{EVPW}
\left\{
\begin{aligned}
\eps^2 \lambda u &= \left(\partial_{z}^2  + \ 1 - 3U_{\rm SF}^2  \right) u - \varepsilon 
\nabla \mathcal{F}^N(\vec{V}_{\rm SF}) \cdot \vec{v}\,,\\
\eps^2 \tau_j \lambda v_j &=  \left(d_j^2 \partial_{z}^2 - \eps^2\right) v_j + \eps^2 u, \qquad \qquad  j = 1, \ldots, N \,. 
\end{aligned}
\right.
\end{align}
As a first order system \eqref{EVPW} becomes
\begin{align} \label{EVPW_SYS}
\left\{
\begin{aligned}
\partial_z u =& p\\
\partial_z p =& \eps^2 \lambda u -  u + 3U_{\rm SF}^2   u + \varepsilon 
\nabla \mathcal{F}^N(\vec{V}_{\rm SF}) \cdot \vec{v}\,,\\
\partial_z v_j =& \eps q^j\,,\\
\partial_z q_j =&  \dfrac{\eps}{d_j^2} \left((1+ \tau_j \lambda) v_j - u \right), \qquad j = 1, \ldots, N \,, 
\end{aligned}
\right.
\end{align}
with the corresponding equivalent slow system
\begin{align}\label{EVPW_SYS2}
\left\{
\begin{aligned}
\eps \partial_y u =& p\\
\eps \partial_y p =& \eps^2 \lambda u -  u + 3U_{\rm SF}^2   u + \varepsilon 
\nabla \mathcal{F}^N(\vec{V}_{\rm SF}) \cdot \vec{v}\,,\\
\partial_y v_j =&  q_j\,,\\
\partial_y q_j =&  \dfrac{1}{d_j^2} \left((1+ \tau_j \lambda) v_j - u \right), \qquad j = 1, \ldots, N \,. 
\end{aligned}
\right.
\end{align}
Since a stationary front has speed zero, we have, similarly to \cite{CBDvHR15, CBvHHR19, DvHK09}, that the first order correction term of $U_{\rm SF}$ in the fast field equals zero. That is,
$$
U_{\rm SF}(z) = U^0_{\rm SF}(z) + \eps^2 U^2_{\rm SF}(z) + \mathcal{O}(\eps^3) =
\tanh{\left(\dfrac{z}{\sqrt2}\right)} + \eps^2 U^2_{\rm SF}(z) + \mathcal{O}(\eps^3)
 \,,\qquad z \in I_f,
$$ 
see \eqref{N:PROFILES} of Lemma~\ref{L:N}. Furthermore, the solvability condition for the second order correct term  of a travelling front $U_{\rm TF}^2$ \eqref{UTF2} simplifies to
\begin{align} \label{SOLVSF}
 - 3 \sqrt2 \int U_{\rm SF}^0 U_{\rm SF}^2  \rho^2 d\,z 
= 2\sqrt2 \left( \sum_{j=1}^N \dfrac{ \alpha_j}{ d_j}
 \right)\,,
\end{align}
where we observe that the term coming from the nonlinearity of the coupling function $\mathcal{F}^N(\vec{V}^*)$ drops out since $V_j^*=0$ at $c=0$, see \eqref{VSTARN}. That is, $\partial_j \mathcal{F}_{\rm nl}^N(\vec{0})=0$.

We are now going to determine the correct asymptotic scaling of $(u,\vec{v})$ by subsequently examining \eqref{EVPW_SYS} in the fast field $I_f$ and \eqref{EVPW_SYS2} in the slow fields $I_s^\pm$, see Definition~\ref{fields}.
\\

{\bf Fast field $I_f$, $\O(1)$}: 
Upon using a regular expansion\footnote{Note that $(u^0,v_1^0, \ldots, v_N^{0}) = (\varphi_U, \varphi_1, \ldots, \varphi_N)$ in the notation of Lemma~\ref{L:EIGSF}.} 
$(u,\vec{v}) = (u^0,\vec{v}^0)+ \eps(u^1,\vec{v}^1)+\eps^2(u^2,\vec{v}^2) + \mathcal{O}(\eps^3)  = (u^0,v_1^0, \ldots, v_N^{0})+ \eps (u^1,v_1^{1}, \ldots, v_N^{1}) + \eps^2 (u^2,v_1^{2}, \ldots, v_N^{2})+ \O(\eps^3)$
we get to leading order in the fast field from~\eqref{EVPW_SYS} 
$$
\O(1): \partial_z^2 u^0 = 3 (U_{\rm SF}^0)^2 u^0 - u^0 \implies u^0 = C \sech^2{(z/\sqrt2)} = C \rho\,,
$$
see again Lemma~\ref{L:N}. \\

{\bf Slow fields $I_s^\pm$, $\O(1)$}: In turn, this reduces the slow equations in the slow fields $I_s^\pm$, since $u^0=0$ in these fields, to
$$
\O(1): d_j^2 \partial_y^2 v_j^{0} = (1+ \tau_j \lambda) v_j^{0} , \qquad j = 1, \ldots, N\,. 
$$ 
These are $N$ linear decoupled equations that can be solved explicitly in terms of exponentials. However, they have to be bounded at $\pm \infty$ (recall $\Re(\lambda)>-\chi$) and match at zero (since the slow components do not change over the fast field, from \eqref{EVPW_SYS}: $\partial_z v_j^{0} = 0 = \partial_z q_j^{0}$). Thus we get that $v_j^{0} = 0 = q_j^{0}$. \\

{\bf Fast field $I_f$, $\O(\eps)$}: Since $U_{\rm SF}^1 = 0$, we get 
 $$
\O(\eps): \partial_z^2 u^1= 3 (U_{\rm SF}^0)^2 u^1 - u^1 \implies u^1 = \hat{C} \sech^2{(z/\sqrt2)}\,.
$$
However, without loss of generality, we can set $\hat{C} =0$ and hence $u^1=0$ \cite{CBvHHR19}. \\

{\bf Slow fields $I_s^\pm$, $\O(\eps)$}: For the next order correction terms in the slow fields, we still get from \eqref{EVPW_SYS2}
$$
\O(\eps): d_j^2 \partial_y^2 v_j^{1} = (1+ \tau_j \lambda) v_j^{1} , \qquad j = 1, \ldots, N\,. 
$$ 
This is solved by 
\begin{align*}
v_j^{1} = 
\left\{
\begin{aligned}
A_j^- e^{\Theta_j^+ y}\,, \quad y \in I_s^-,\\
A_j^+ e^{\Theta_j^- y}\,, \quad y \in I_s^+,
\end{aligned}
\right.
\end{align*}
with
$
\Theta^\pm :=  \pm \dfrac{1}{ d_j} \sqrt{1+\tau_j \lambda},
$
and where we already forced that the solutions stay bounded as $y \to \pm \infty$.

At the $\mathcal{O}(1)$-level the profiles $v_j^{0}$ and their derivatives $q_j^{0}$ needed to match over the fast field. This is no longer true here. While the change of the slow components over the fast field is still zero $\partial_z v_j^{1} = 0$, see \eqref{EVPW_SYS},
the change of the derivative of the slow components over the fast fields is no longer zero. Again, from \eqref{EVPW_SYS} we get $$\partial_z q_j^{1}=-u^0/d_j^2,$$ which implies that, to leading order,
$$
\Delta q_j^{1} = q_j^{1}(0^+) - q_j^{1}(0^-)
=
\int_{-\infty}^{\infty} \partial_z q_j^{1} d z = -\dfrac{1}{d_j^2} \int_{-\infty}^{\infty} C \sech^2{(z/\sqrt2)} d z
=
-2 \sqrt2 \dfrac{C}{d_j^2}\,.
$$
Hence, matching the $v_j^{1}$ profiles over the fast field gives
$
A_j^-= A_j^+:= A_j\,,
$ while the derivatives give
$$
\Theta^+ A_j -2 \sqrt2 \dfrac{C}{d_j^2} = \Theta^- A_j\implies 
A_j =  \dfrac{2 \sqrt2 C}{d_j^2(\Theta^+-\Theta^-)} 
= \dfrac{ \sqrt2  C}{ d_j\sqrt{ 1+\tau_j \lambda}}.
$$
Thus, we get $v_j^{1} = \sqrt2 C \varphi_j(y)$ \eqref{N:UPROFILE_EIG}. \\

{\bf Fast field $I_f$, $\O(\eps^2)$}: Finally, at the $\O(\eps^2)$-level we get for the fast component in the fast scaling in the fast field
 $$
\O(\eps^2): \partial_z^2 u^2 = 3 (U_{\rm SF}^0)^2 u^2 - u^2  + 6 U_{\rm SF}^0 U_{\rm SF}^2 u^0 + 
\nabla \mathcal{F}^N(\vec{V}^0_{{\rm SF}}) \cdot  \vec{v}^1
+ \lambda u^0\,.$$
Recalling that for the coupling function $\mathcal{F}^N(\vec{V})$ \eqref{N:NONL22} we get in the fast field -- since $V^0_{{\rm SF},j} =0$ in the fast field, see \eqref{N:PROFILES} -- that 
$
\nabla \mathcal{F}^N(\vec{V}^0_{\rm SF}) \cdot  \vec{v}^1 = \sum_{j=1}^N\alpha_jv_j^{1}\,.
$
Applying a solvability condition, {\it i.e.}, testing the above expression with $u^0$, and using \eqref{SOLVSF}, yields
\begin{align*}
\begin{aligned}
0&=6 \int_{-\infty}^\infty U_{\rm SF}^0 U_{\rm SF}^2 (u^0)^2 dz + 
\nabla \mathcal{F}^N(\vec{0}) \cdot  \vec{v}^1(0) \int_{-\infty}^\infty  u^0 dz+ \lambda \int_{-\infty}^\infty (u^0)^2  dz \\
&=  4C^2 \left(- \sum_{j=1}^N \dfrac{ \alpha_j}{ d_j}+\sum_{j=1}^N   \dfrac{ \alpha_j}{d_j \sqrt{\tau_j \lambda+1}} + \dfrac{ \sqrt 2 \lambda }{3}\right)  \,,
\end{aligned} 
\end{align*}
which is equivalent to the Evans function \eqref{EV0}.

To conclude, we 
set $C = 1/(\sqrt{2} \eps)$ to get \eqref{N:UPROFILE_EIG} and this completes the proof of Lemma~\ref{L:EIGSF}. 
\end{proof}
\section{Proof of Lemma~\ref{L:EIGv}}
\label{A:L}
\begin{proof}
The general solutions to \eqref{SET} are given by
$$
v^j_{-}(x) = A^j_{-} e^{x} + B^j_{-} e^{-x} +\dfrac{\tau}{2} e^{x} \int_{0}^x e^{-s} v^{j-1}_-(s) ds -
\dfrac{\tau}{2} e^{-x} \int_{-\infty}^x e^{s} v^{j-1}_-(s) ds \,,
$$
and
$$
v^j_{+}(x) = A^j_{+} e^{x} + B^j_{+} e^{-x} +\dfrac{\tau}{2} e^{x} \int_\infty^x e^{-s} v^{j-1}_{+}(s) ds -
\dfrac{\tau}{2} e^{-x} \int_0^x e^{s} v^{j-1}_{+}(s) ds \,.
$$
Implementing the boundary conditions at $\pm \infty$, see \eqref{BC}, gives $B^j_{-} = A^j_{+}=0$. Next, we match at $0$ to obtain
\begin{align*}
\begin{aligned}
v^j_{-}(0) &= v^j_{+}(0) &&\implies  A^j_{-} - \dfrac{\tau}{2} \int_{-\infty}^0 e^{s} v^{j-1}_{-}(s) ds = 
B^j_{+} +  \dfrac{\tau}{2} \int_\infty^0 e^{-s} v^{j-1}_{+}(s) ds \\
(v^j_{-})'(0) &= (v^j_{+})'(0) &&\implies
A^j_{-}   +  \dfrac{\tau}{2} \int_{-\infty}^0 e^{s} v^{j-1}_{-}(s) ds = - B^j_{+} + \dfrac{\tau}{2} \int_\infty^0 e^{-s} v^{j-1}_{+}(s) ds\,.
\end{aligned}
\end{align*}
This yields
$$
A^j_{-}  =  \dfrac{\tau}{2} \int_\infty^0 e^{-s} v^{j-1}_{+}(s) ds\,, \qquad 
B^j_{+} = -\dfrac{\tau}{2} \int_{-\infty}^0 e^{s} v^{j-1}_{-} ds \,,
$$
and \eqref{SET} with \eqref{BC} are thus solved by
\begin{align}
\label{EQ0-}
v^j_{-}(x) = \dfrac{\tau}{2} e^{x} \int_\infty^0 e^{-s} v^{j-1}_{+}(s) ds  +\dfrac{\tau}{2} e^{x} \int_{0}^x e^{-s} v^{j-1}_{-}(s) ds -
\dfrac{\tau}{2} e^{-x} \int_{-\infty}^x e^{s} v^{j-1}_{-}(s) ds \,,\\
\label{EQ0+}
v^j_{+}(x) =  -\dfrac{\tau}{2} e^{-x}\int_{-\infty}^0 e^{s} v^{j-1}_{-}(s) ds  +\dfrac{\tau}{2} e^{x} \int_\infty^x e^{-s} v^{j-1}_{+}(s) ds -
\dfrac{\tau}{2} e^{-x} \int_0^x e^{s} v^{j-1}_{+}(s) ds \,.
\end{align}

We use induction to show that \eqref{PAR} holds if $v^0_{\pm}(x) = e^{\mp x}/d$.
For the base case $m=0$ the results hold true by assumption since $v^0_{+}(-x) =  e^{-(-x)}/d = v^0_{-}(x).$
For the inductive step, we assume that \eqref{PAR} holds for $m=j-1$. Then, for $m=j$ we get from \eqref{EQ0-}, \eqref{EQ0+} and the inductive assumption that
\begin{align*}
\begin{aligned}
v^j_{+}(-x) &=  -\dfrac{\tau}{2} e^{x}\int_{-\infty}^0 e^{s} v^{j-1}_{-}(s) ds  +\dfrac{\tau}{2} e^{-x} \int_\infty^{-x} e^{-s} v^{j-1}_{+}(s) ds -
\dfrac{\tau}{2} e^{x} \int_0^{-x} e^{s} v^{j-1}_{+}(s) ds \\
&=
 \dfrac{\tau}{2} e^{x}\int_{\infty}^0 e^{-z} v^{j-1}_{-}(-z) dz  -\dfrac{\tau}{2} e^{-x} \int_{-\infty}^{x} e^{z} v^{j-1}_{+}(-z) dz +
\dfrac{\tau}{2} e^{x} \int_0^{x} e^{-z} v^{j-1}_{+}(-z) dz
 \\
&=
 \dfrac{\tau}{2} e^{x}\int_{\infty}^0 e^{-z} v^{j-1}_{+}(z) dz  -\dfrac{\tau}{2} e^{-x} \int_{-\infty}^{x} e^{z} v^{j-1}_{-}(z) dz +
\dfrac{\tau}{2} e^{x} \int_0^{x} e^{-z} v^{j-1}_{-}(z) dz 
\\&= v^j_{-}(x).
\end{aligned}
\end{align*}
This completes the proof of \eqref{PAR}. 

Next, we prove the final statement \eqref{EQ1n} of Lemma~\ref{L:EIGv}, which we repeat for convenience,
\begin{align}
\label{aij_A}
v^{j}_{+}(x) = C_{j}  v^0_{+}(x)  \sum_{i=0}^{j} a_j^i   x^i :=
\underbrace{(-1)^{j}\dfrac{(2j-1)!!}{(2j)!!} \tau^{j}}_{:= \,\,C_{j}} v^0_{+}(x)  \underbrace{\sum_{i=0}^{j} \dfrac{2^i}{i!}  \dfrac{\tbinom{2j-i}{j}}{  \tbinom{2j}{j}}   x^i}_{:= f^j(x)}
 = C_jv_{+}^0(x)f^j(x)
\,,
\end{align} 
for $j=1,2,\ldots,k$. The symmetry of \eqref{PAR} implies, since $v_+^0(0)=1/d$, that it must hold that
$$(v^{j}_{+})'(0)=0 \implies C_{j} \left(-\dfrac{a_{j}^0}{d}+\dfrac{a_{j}^1}{d}\right) =0 \implies a_{j}^0 = a_{j}^1\,,$$
which is indeed true
$$
a_{j}^0 = \dfrac{2^0}{0!}  \dfrac{\tbinom{2j}{j}}{  \tbinom{2j}{j}} = 1
=  \dfrac{2^1}{1!}  \dfrac{\tbinom{2j-1}{j}}{  \tbinom{2j}{j}} = a_j^1\,.
$$ 
Furthermore, observe that $v^{j}_{+}(x)$ \eqref{EQ1n}/\eqref{aij_A} as given in the lemma obeys the asymptotic boundary conditions \eqref{BC}. 
Next, we compute the second derivative of $v^{j}_+(x)$ (using that $v_+^0(x) = e^{-x}/d$)
\begin{align*}
\begin{aligned}
(v^{j}_{+})''(x) &= C_{j} \left((v^0_{+})''(x) f^{j}(x)+2(v^0_{+})'(x) (f^{j})'(x)+ v^0_{+}(x) (f^{j})''(x)\right) \\
&= 
C_{j} \left(v^0_{+}(x) f^{j}(x)-2v^0_{+}(x) (f^{j})'(x)+ v^0_{+}(x) (f^{j})''(x)\right)\\
&= v^{j}_{+}(x) + C_{j} \left(-2v^0_{+}(x) (f^{j})'(x)+ v^0_{+}(x) (f^{j})''(x)\right)\,.
\end{aligned}
\end{align*} 
So, after substituting $v^{j}_{+}(x)$ \eqref{EQ1n}/\eqref{aij_A} and the above into the ODE \eqref{SET}, we get that the following relation 
\begin{align*}
C_{j} \left(-2v^0_{+}(y) (f^{j})'(x)+ v^0_{+}(x) (f^{j})''(x)\right) &= \tau v^{j-1}_{+}(x) = 
\underbrace{(-1)^{j-1}\dfrac{(2j-3)!!}{(2j-2)!!} \tau^{j-1}}_{:= \,\,C_{j-1}} \tau v^0_{+}(x) f^{j-1}(x)\,,
\end{align*}
most hold for \eqref{EQ1n}/\eqref{aij_A} to be true.
Dividing out $v^0_{+}(x)$ and using the relation
$$C_j = -(2j-1) \tau C_{j-1}/(2j)$$
gives the following relation for $f^{j-1}$
\begin{align*}
\dfrac{2j}{2j-1}  f^{j-1}(x) = 2 (f^{j})'(x) - (f^{j})''(x)\,, \qquad j=2,3, \ldots, k\,.
\end{align*}
Next, we substitute the expansion of $f^j(x)$ and $f^{j-1}(x)$ \eqref{aij_A} into the above relation to obtain
\begin{align*}
\dfrac{2j}{2j-1} \sum_{i=0}^{j-1} a_{j-1}^i x^i&=   2\sum_{i=1}^{j} i a_{j}^i x^{i-1} - \sum_{i=2}^{j} i(i-1) a_{j}^i x^{i-2} \\
 &= 2\sum_{i=0}^{j-1} (i+1) a_{j}^{i+1} x^{i} 
        - \sum_{i=0}^{j-2} (i+2)(i+1) a_{j}^{i+2} x^{i}\,.
\end{align*}
Equating equal powers of $x$ gives, after observing that the latter term does not include an $x^{j-1}$-term, two types of recurrence relations. For $i = j-1$ we get
\begin{align}
\label{a1}
a_{j}^{j} = \dfrac{1}{2j-1} a_{j-1}^{j-1}= 
 \dfrac{1}{(2j-1)!!} \, .
\end{align}
Furthermore,
\begin{align}
\label{ainew}
\dfrac{2j}{2j-1} a^i_{j-1} = 
 2 (i+1)a_{j}^{i+1} -  (i+1)(i+2)a_{j}^{i+2} \,, \qquad 0 \leq i \leq j-2 \,.
\end{align}
The final step in the proof is to show that the $a_j^i$ defined in \eqref{aij_A} indeed 
obey
the
recurrence relation \eqref{a1} and \eqref{ainew}. This follows directly from several straightforward calculations. In particular, \eqref{a1} follows from
$$
a_j^j := \dfrac{2^j}{j!}  \dfrac{\tbinom{2j-j}{j}}{  \tbinom{2j}{j}} = 
\dfrac{2^j}{j!}  \dfrac{1}{  \tbinom{2j}{j}} =  \dfrac{2^j j!}{(2j)!} = 
\dfrac{(2j)!!}{(2j)!}=
\dfrac{1}{(2j-1)!!}\,,
$$
while \eqref{ainew} follows from the observation that
\begin{align*}
2 (i+1)a_{j}^{i+1} &-  (i+1)(i+2)a_{j}^{i+2} \\:=& 
\dfrac{2(i+1)2^{i+1}}{(i+1)!}  \dfrac{\tbinom{2j-i-1}{j}}{  \tbinom{2j}{j}}
- (i+2)(i+1)
\dfrac{2^{i+2}}{(i+2)!}  \dfrac{\tbinom{2j-i-2}{j}}{  \tbinom{2j}{j}} \\
=&
\dfrac{2^{i+2}}{i! \tbinom{2j}{j}} \left(\dbinom{2j-i-1}{j} - \dbinom{2j-i-2}{j} \right) 
=
\dfrac{2^{i+2}}{i! \tbinom{2j}{j}} \dbinom{2j-i-2}{j-1} \\
=& \dfrac{2^{i+2} j}{i!}  \dfrac{j!(2j-i-2)!}{(2j)!(j-i-1)!}
=
\dfrac{2^{i+1} j}{i!} \dfrac{(j-1)!(2j-i-2)!}{(2j-1)!(j-i-1)!}
\end{align*}
and
\begin{align*}
\dfrac{2j}{(2j-1)} a^i_{j-1} :=&  \dfrac{2j}{(2j-1)}  \dfrac{2^i}{i!}\dfrac{\tbinom{2j-i-2}{j-1}}{  \tbinom{2j-2}{j-1}} \\ =&  \dfrac{2^{i+1} j }{i! (2j-1)} \dfrac{(2j-i-2)!(j-1)!(j-1)!}{(j-1)!(j-i-1)!(2j-2)!}
=
\dfrac{2^{i+1} j }{i!} \dfrac{(2j-i-2)!(j-1)!}{(j-i-1)!(2j-1)!}\,.
\end{align*}
This completes the proof of Lemma~\ref{L:EIGv}. 
\end{proof}
\section{Proof of Proposition~\ref{prop:reduced_system}}
\label{subsec:proof_CM}
Before we start with the actual proof it is convenient to collect a few preliminaries that are not relevant to mention in the main body of the paper and too lengthy and distracting to be discussed in the proof below. By the assumptions of Proposition~\ref{prop:reduced_system}, the operator $\mathcal{L}$ from  \eqref{Ldef} possesses a zero eigenvalue with algebraic multiplicity $\Np+1$ and geometric multiplicity 1. The rest of the spectrum is located in the left half-plane and is bounded away from the imaginary axis for positive $\varepsilon$. In addition to the kernel eigenfunction $\Psi_0$, there exist $\Np$ generalised eigenfunctions satisfying 
\begin{align}\label{e:Jordan}
 \mathcal{L} \Psi_0 = 0 \, , \quad \mathcal{L} \Psi_{j+1} = \varepsilon^2 \Psi_j \, , \quad  j = 0, \ldots, \Np-1 \,.
 \end{align}
The $L^2$-adjoint operator $ \mathcal{L}^* $ with respect to the duality product
$$
 \langle Z, \widetilde{Z} \rangle := \langle U, \widetilde{U} \rangle_{L^2} + \sum_{j = 1}^{\Np} \langle V_j, \widetilde{V}_j \rangle_{L^2} 
$$
has the same spectral properties as $ \mathcal{L}^*$ with a (not explicitly known) Jordan chain
\begin{align*}
 \mathcal{L}^* \Psi_0^* = 0 \, , \quad  \mathcal{L}^* \Psi^*_{j+1} = \varepsilon^2 \Psi_j^* \, , \quad  j = 0, \ldots, \Np-1 \,.
 \end{align*}
In the following we will use the adjoint eigenfunctions to perform projections. 
 For completeness, we also prove the following basic observations. 
\begin{lemma}\label{lemma:CM_1}
For $\Np>0$ and $k_1,k_2=0,1,\ldots,\Np$ we have
\begin{align}\label{eq:p_n}
 \langle \Psi_{k_1},  \Psi_{k_2}^* \rangle =: p_j \, , k_1 + k_2 = j \, , \quad j = 0,1, \ldots, 2\Np, 
\end{align}
and $p_j = 0$ for all $ j < \Np$. 
Moreover, $p_\Np\neq 0$.
\end{lemma}
\begin{proof}
Using the Jordan chain structure, we have that
\begin{align*}
 \langle \Psi_j, \Psi_k^* \rangle =  \langle \Psi_{j-1}, \Psi_{k+1}^* \rangle =  \langle \Psi_{j+1}, \Psi_{k-1}^* \rangle \,,
\end{align*}
which justifies 
definition \eqref{eq:p_n}. Moreover, 
\begin{align*}
 \langle \Psi_0, \Psi_k^* \rangle 
=  \langle \Psi_0, \calL^*\Psi_{k+1}^* \rangle = \langle \calL \Psi_0, \Psi_{k+1}^* \rangle 
 =  0 \,, \quad k = 0, \ldots, \Np-1\, ,
\end{align*}
(for $k=0$ we use $\Psi_0=\calL \Psi_1$) and by the previous relation
\begin{align*}
 \langle \Psi_j, \Psi_k^* \rangle=  0 \,, \quad j + k \leq \Np-1\, ,
\end{align*}
which is the second statement. 
Concerning the last statement, since $\mathrm{ran}(\calL^*) = (\mathrm{ker}(\calL))^\perp$, if $\langle \Psi_0, \Psi_{\Np-1}^*\rangle=0$ then $\Psi_{\Np-1}\in\mathrm{ran}\calL^*$, which is not possible as the dimension of the generalised kernel is $\Np$.
\end{proof}
\begin{lemma}\label{lemma:CM_2}
$\langle \Psi_{j}',  \Psi_{k}^* \rangle  = 0\, , \quad  j,k \in \{0, \ldots,\Np\}.
$
\end{lemma}

\begin{proof}
The stationary front $ Z_{\rm SF} $ is an odd function of $x$. Hence, $\Psi_0 =  Z_{\rm SF}'$ is even. Furthermore, since $\mathcal{L}$ preserves parity, all (generalised) eigenfunctions are even functions of $x$. The same is true for the adjoint (generalised) eigenfunctions functions. 
\end{proof}

\begin{proof}[Proof of Proposition~\ref{prop:reduced_system}]
We consider a fixed sufficiently small $\eps>0$ so that Lemma~\ref{L:N+1_J} applies. For readability, we suppress the $\varepsilon$-dependence of terms in the following. 
We will apply the center manifold theory as in \cite[e.g.]{haragus2010local} which will justify using the Ansatz \eqref{eq:Zsplit1} $Z(x,t) = Z_{\rm SF}(\eta) + \vec{\Psi}(\eta)\cdot\vec{b} + R(\eta, t)$ with $\eta=x-a(t)$ and $R(\cdot,t)$ orthogonal to $\Psi_j^*$, $j=0,1,\ldots,\Np$. For completeness, we next show the relevant steps in the derivation. Inserting the ansatz into \eqref{eq:multi-component-RD} in the form \eqref{eq:Zeqn} gives 
\begin{align*}
  \dot{a} \left( \Phi + \vec{\Psi}'\cdot\vec{b} +\partial_{\eta} {R}  \right) + \vec{\Psi}\cdot\dot{\vec{b}} +\partial_{t} {R} = \mathcal{L} \left( \vec{\Psi}\cdot\vec{b} \right) + \mathcal{L} {R}  + \check{\mathcal{N}}(\vec{b}, R,\check{P}) \,,
\end{align*}
with nonlinear part $\check{\mathcal{N}}(\vec{b}, R,\check{P}) = \calO\left(|\vec{b}|^2 + \|R\|^2 + |\cP|\right)$. Using  \eqref{e:Jordan} this simplifies to
\begin{align}\label{eq:preCMF}
  \dot{a} \left( \Phi + \vec{\Psi}'\cdot\vec{b} +\partial_{\eta} {R}  \right) + \vec{\Psi}\cdot\dot{\vec{b}} +\partial_{t} {R} =  \varepsilon^2 \sum_{j=0}^{\Np-1} b_{j+1} \Psi_j + \mathcal{L} R +  \check{\mathcal{N}}(\vec{b}, R, \check{P}) \,.
\end{align}
Successive projections of this onto the generalised eigenspace via the adjoint eigenfunctions, using Lemmas~\ref{lemma:CM_1} and \ref{lemma:CM_2}, and the assumption that $R$ is orthogonal to the generalized kernel
gives the system
\begin{align*}
 \dot{a} \widetilde{\rho}(R) + A
 \begin{pmatrix}
  \dot{a} \\ \dot{\vec{b}} 
 \end{pmatrix}
 = \varepsilon^2 
A\Jnil 
 \begin{pmatrix}  {a} \\ \vec{b}
\end{pmatrix}
 +
 \bar{\mathcal{N}}(\vec{b}, R,\check{P}) \, ,
\end{align*}
where  $\bar{\mathcal{N}}(\vec{b}, R,\check{P})=\calO(|\vec{b}|^2 + \|R\|^2+|\cP|)$, 
$\widetilde{\rho}(R)=\calO(\|R\|)$ and 
\begin{align*}
\widetilde{\rho}(R) := \begin{pmatrix}
 \langle \partial_{\eta} R, \Psi_{\Np}^* \rangle \\
 \!\!\langle \partial_{\eta} R, \Psi_{\Np-1}^* \rangle\!\! \\
  \vdots \\
  \langle \partial_{\eta} R, \Psi_0^* \rangle 
\end{pmatrix}, \,
    \bar{\mathcal{N}}(\vec{b}, R, \check{P}) := \begin{pmatrix}
 \langle \check{\mathcal{N}}(\vec{b}, R, \check{P}), \Psi_{\Np}^* \rangle \\
 \langle \check{\mathcal{N}}(\vec{b}, R, \check{P}), \Psi_{\Np-1}^* \rangle \\
  \vdots \\
  \langle \check{\mathcal{N}}(\vec{b}, R, \check{P}), \Psi_0^* \rangle 
  \end{pmatrix} \, ,\\
 \Jnil:=
{\small \begin{pmatrix}
   0   & 1 & 0 & \cdots & 0 \\[-1ex]
   \vdots   & \ddots & \ddots & \ddots &  \vdots  \\[-1ex]
   \vdots &  & \ddots & \ddots &0\\[-1ex]
   \vdots   &   &  &  \ddots & 1 \\[-1ex]
   0   & \cdots & \cdots & \cdots& 0 
\end{pmatrix} }\,,\,
A := \begin{pmatrix}
 p_{\Np} & p_{\Np+1} & \cdots & p_{2\Np} \\
    0  & \ddots & \ddots & \vdots \\
    \vdots  &   \ddots   & \ddots & p_{\Np+1} \\
    0  &   \cdots   & 0     & p_{\Np}
\end{pmatrix}.
\end{align*}
Due to Lemma~\ref{lemma:CM_1} $A$ is invertible and multiplication with 
$A^{-1}$ gives the system
\begin{align}\label{eq:arho-sys}
 \dot{a} \rho(R) +
 \left( 
 \begin{array}{c}
  \dot{a} \\ \dot{\vec{b}}  
 \end{array}
 \right)
 = \varepsilon^2 
\Jnil\begin{pmatrix}  {a} \\ \vec{b}
\end{pmatrix} +
 A^{-1}\bar{\mathcal{N}}(\vec{b}, R,\check{P}),
\end{align}

where ${\rho}(R) := A^{-1}\widetilde{\rho}(R) = \calO(\|R\|)$. 
For $\|R\|$ sufficiently small, extracting the first equation of \eqref{eq:arho-sys} and expanding $1/(1+\rho_1(R))$ gives 
\begin{align*}
\dot{a} = \left( \frac{1}{1+\rho_1(R)} \right) \left( \varepsilon^2 b_1 + 
(A^{-1}\bar{\mathcal{N}}(R, \vec{b},\check{P}))_1 \right)  = \varepsilon^2 b_1 + g_A(R, \vec{b},\check{P}) \, ,
\end{align*}
where $g_A(R, \vec{b},\check{P})=\calO(\|R\|^2+|\vec{b}|^2 + |\cP|)$. Multiplication with $\rho_j$ thus gives
\begin{align*}
 \dot{a} {\rho}_j(R) = \left( \varepsilon^2 b_{1} + g_A(R, \vec{b}),\check{P} \right) {\rho}_j(R) 
 = \calO\left(|\vec{b}|^2 \|R\|^2 + |\cP|\right)\, , \quad j = 2, \ldots,\Np+1 \,.
\end{align*}
Substitution into \eqref{eq:arho-sys} allows to write the last $\Np$ components of \eqref{eq:arho-sys} as
\begin{align*}
\frac{d}{dt} \vec{b}
 = \varepsilon^2 
\Jnil \vec{b}
 +
 \mathcal{N}(\vec{b}, R,\check{P}) \, ,
\end{align*}
where $ \mathcal{N}(\vec{b}, R,\check{P}) =\calO\left(|\vec{b}|^2 \|R\|^2 + |\cP|\right)$ 
accounts for all purely nonlinear terms,
which we do not need to specify at this point. Center manifold theory as presented in \cite[e.g.]{haragus2010local} or as employed for $N=2$ in \cite{CBvHHR19} provides the existence of a center manifold that can locally be described by a smooth function such that 
\begin{align*}
 R = H(\vec{b}, \check{P})  =\calO\left(|\vec{b}|^2 + |\cP|\right)\, .
\end{align*}
independent of $a$. In conclusion,  
\eqref{eq:preCMF}
is locally of the form
\begin{align}\label{eq:CMF0}
 \dot{a} =  \varepsilon^2 b_1 + g_A(H(\vec{b}, \check{P}) , \vec{b},\check{P})  \, , \quad
\frac{d}{dt} \vec{b}
 = \varepsilon^2 
\Jnil \vec{b}
 +
 \mathcal{N}(\vec{b}, H(\vec{b}, \check{P}),\check{P})\, .
\end{align}
with nonlinear terms $g_A, \mathcal{N} = \calO(|\vec{b}|^2 + |\cP|)$. In particular, 
\begin{align*}
\dot{a} = \varepsilon^2 c_1 \, \text{ with }
 \varepsilon^2 c_1 := \varepsilon^2 b_1 + g_A(H(\vec{b}, \check{P}) , \vec{b},\check{P}) \, ,
\end{align*}
and further performing a series of near-identity changes of variables (as, for instance, in \cite{CBvHHR19}, Lemma 3) and Taylor expansion finally yields the statement. The factor $ \varepsilon^2 $ of {the linear terms in} $\Gred$ follows {\it a posteriori} from the fact that, by construction, the eigenvalues of $\mathcal{L}$ of $\mathcal{O}(\varepsilon^2)$ coincide with the eigenvalues of the reduced system on the center manifold. Combining this with the leading order existence analysis implies that all terms in $\Gred$ are order $\eps^2$. This concludes the proof.
\end{proof}
\end{appendix}

\end{document}